 \documentclass{amsart}

\usepackage{amsmath,multicol}
\usepackage{amssymb}
\usepackage{amsthm}
\usepackage[all]{xypic}
\usepackage{amsfonts}
\usepackage{fullpage}
\usepackage{verbatim,color}
\usepackage{url}
\usepackage{hyperref}

\numberwithin{equation}{section}
\theoremstyle{plain}
\newtheorem{theorem}[equation]{Theorem}
\newtheorem{proposition}[equation]{Proposition}
\newtheorem{corollary}[equation]{Corollary}
\newtheorem{lemma}[equation]{Lemma}
\newtheorem{question}[equation]{Question}

\newtheorem{assumption}[equation]{Assumption}

\theoremstyle{definition}
\newtheorem{definition}[equation]{Definition} 
\newtheorem{definitions}[equation]{Definitions} 
\newtheorem{notation}[equation]{Notation}

\newtheorem{example}[equation]{Example}
\newtheorem{examples}[equation]{Examples}

\newtheorem{artin}[equation]{Artin's Conjectures}

\theoremstyle{remark}
\newtheorem{remark}[equation]{Remark}
\newtheorem{remarks}[equation]{Remarks}
\newcommand{\beq}{\begin{equation}}
\newcommand{\eeq}{\end{equation}}

\newcommand{\ang}[1]{\langle #1 \rangle}

\newcommand{\st}{\left\vert\right.}

\newcommand{\bbar}[1]{\overline{#1}}

\DeclareMathOperator{\Hom}{{Hom}}

\DeclareMathOperator{\End}{{End}}

\DeclareMathOperator{\Ext}{{Ext}}

\DeclareMathOperator{\Proj}{Proj}
\DeclareMathOperator{\Spec}{Spec}

\DeclareMathOperator{\GK}{GKdim}
\DeclareMathOperator{\gr}{gr}
\DeclareMathOperator{\Div}{Div}

\DeclareMathOperator{\lann}{{\ell}-ann}
\DeclareMathOperator{\rann}{r-ann}
\DeclareMathOperator{\sat}{sat}

\newcommand{\mc}{\mathcal}

\newcommand{\kk}{{\Bbbk}}

\newcommand{\ZZ}{{\mathbb Z}}

\newcommand{\NN}{{\mathbb N}}

\newcommand{\mb}{\mathbb}
\newcommand{\wt}{\widetilde}

\newcommand{\sA}{\mc{A}}
\newcommand{\sB}{\mc{B}}

\newcommand{\sF}{\mc{F}}
\newcommand{\sG}{\mc{G}}
\newcommand{\sH}{\mc{H}}

\newcommand{\sL}{\mc{L}}
\newcommand{\sM}{\mc{M}}
\newcommand{\sN}{\mc{N}}
\newcommand{\sO}{\mc{O}}

\newcommand{\sY}{\mc{Y}}

\newcommand{\divd}{\mathbf{d}}

\newcommand{\divc}{\mathbf{c}}
\newcommand{\divp}{\mathbf{p}}
\newcommand{\divq}{\mathbf{q}}
\newcommand{\divx}{\mathbf{x}}
\newcommand{\divy}{\mathbf{y}}

\newcommand{\divu}{\mathbf{u}}
\newcommand{\divv}{\mathbf{v}}

\DeclareMathOperator{\lfd}{\!-fd}
\DeclareMathOperator{\fd}{fd-\!}
\DeclareMathOperator{\rgr}{gr-\!}
\DeclareMathOperator{\lgr}{\!-gr}

\DeclareMathOperator{\rGr}{Gr-\!}

\DeclareMathOperator{\lqgr}{\!-qgr}

\DeclareMathOperator{\rQgr}{Qgr-\!}
\DeclareMathOperator{\rqgr}{qgr-\!}

\DeclareMathOperator{\GKdim}{GKdim}

\newcommand{\ver}[1]{^{(#1)}}

\newcommand{\Rbar}{\overline{R}}

\newcommand{\cog}{sporadic ideal}
\newcommand{\cogs}{sporadic ideals}
\newcommand{\Cogs}{Sporadic ideals}
  \newcommand{\ppe}{{\    \buildrel\scriptscriptstyle\bullet\over =  \  }}
\newcommand{\wh}{\widehat}
\newcommand{\too}{\longrightarrow}
\newcommand{\idealiser}{\mathbb{I}}     %\newcommand{\idealiser}{I\hskip-2pt I}

\DeclareMathOperator{\tors}{tors}
\newcommand{\tg}{\tors_g}
\DeclareMathOperator{\coh}{coh}

\begin{document}
\title[Orders in the Sklyanin algebra]{Classifying orders in the Sklyanin   algebra}
\author{D. Rogalski,  S. J. Sierra, and J. T. Stafford}
\address{(Rogalski)
Department of Mathematics, UCSD, La Jolla, CA 92093-0112, USA. }
\email{drogalsk@math.ucsd.edu}
 \address{(Sierra) School of Mathematics,
University of Edinburgh, Edinburgh EH9 3JZ, U.K.}
\email{s.sierra@ed.ac.uk}
\address{(Stafford) School of Mathematics,  The University of Manchester,   Manchester M13 9PL,
U.K.}
\email{Toby.Stafford@manchester.ac.uk}

\thanks{The first author was partially supported by NSF grants DMS-0900981 and DMS-1201572.}
 \thanks{The second author was  partially   supported by an NSF Postdoctoral Research
Fellowship, grant DMS-0802935.}
\thanks{The third
   author is a Royal Society Wolfson Research Merit Award Holder.}
%\date{\today}


  \begin{abstract} 
 One of the major open problems in noncommutative algebraic geometry is the classification of 
noncommutative surfaces, and this paper resolves a significant case of this problem.
  
Specifically,  let  $S$ denote the 3-dimensional Sklyanin algebra  over an algebraically closed
 field $\kk$ and assume that $S$  is not a finite module over its centre. (This algebra corresponds to a
  generic noncommutative~$\mathbb{P}^2$.) Let $A=\bigoplus_{i\geq 0}A_i$ be any  connected graded
  $\kk$-algebra    that is contained in and has the same quotient ring as a Veronese ring
$S^{(3n)}$.
Then we give a reasonably complete description of the structure of $A$. This  is most satisfactory
when 
$A$ is a maximal order, in which case we prove, subject to a minor technical condition, that   $A$   is a noncommutative blowup of
$S^{(3n)}$ at a (possibly 
 non-effective) divisor on  the associated elliptic curve  $E$. 
 It follows that  $A$ has surprisingly pleasant properties; for example it is automatically
noetherian, indeed strongly noetherian,  
 and has  a   dualizing complex.
\end{abstract} 

   \maketitle 
 \tableofcontents
 
\clearpage

 \section{Introduction}

Noncommutative (projective) algebraic geometry has been very successful in using   techniques and intuition 
from algebraic geometry  to study  noncommutative graded algebras and many classes of algebras have been classified using these ideas. 
In particular, noncommutative irreducible curves (or connected graded  domains  of Gelfand-Kirillov dimension 
$2$) have been classified \cite{AS} as have large classes of noncommutative irreducible surfaces (or 
connected graded 
noetherian domains $A$ with $\GKdim A=3$).   

Indeed, the starting point of this subject was really    the classification by 
Artin, Tate, and Van den Bergh \cite{ATV, ATV2} of noncommutative projective planes (noncommutative analogues of a polynomial ring $\kk[x,y,z]$).
The generic example here is the Sklyanin algebra
$$S=Skl(a,b,c)=\kk\{x_1,x_2,x_3\}/(ax_ix_{i+1}+bx_{i+1}x_i+cx_{i+2}^2 : i\in \ZZ_3),$$
where $(a,b,c)\in \mathbb{P}^2\smallsetminus \mathcal{S}$ for a (known) set $\mathcal S$.  The geometric methods of \cite{ATV} were necessary to understand  this algebra.   See \cite{SV} for a survey of many of these results.

 In the other direction, one would like to  classify \emph{all} noncommutative surfaces and a programme for this has been suggested by Artin \cite{Ar}.  This paper completes a significant case of this programme by classifying   the 
 graded noetherian orders  contained in the Sklyanin algebra.
 In this  introduction we will first describe  our main results and then discuss the historical background and give an idea of the proofs.

%%%%%%%%%%%%%%%%%

\subsection*{The main  results}  Fix a Sklyanin algebra $S=Skl(a,b,c)$ defined over an algebraically closed base field $\kk$. 
For technical reasons we mostly work inside the 3-Veronese ring $T=S^{(3)}$; thus $T=\bigoplus T_n$ with 
 $T_n=S_{3n}$ for each $n$, under the natural graded structure of $S$. The difference between these algebras  is not particularly significant; for example, the quotient category $\rqgr T$  
 of graded noetherian right $T$-modules modulo those of finite length, is equivalent to $\rqgr S$.
Then $T$ contains a canonical central 
element $g\in T_1=S_3$ such that the factor 
$B=T/gT$ is a \emph{TCR or twisted homogeneous coordinate ring} $B=B(E,\sM,\tau)$ of an elliptic curve $E$.
Here  $\sM$ is a 
line bundle  of degree 9 and $\tau\in \text{Aut}_{\kk}(E)$ (see  Section~\ref{HATTRICK} for the   definition).
We assume throughout the paper that $|\tau|=\infty$; equivalently that  $T$    is not a finite module over is centre. 

  Our main results are phrased in terms of certain 
  \emph{blowups} $T(\divd)\subset T$.  These are discussed  in more detail later in this  introduction.
  Here we will just note that, when $\divp$ is  a closed point of $E$, then  $T(\divp)$ is   the subring of $T$ 
  generated by those elements $x\in T_1$ whose images in $T/gT$ vanish at $\divp$.  
For effective divisors $\divd$ (always of degree  at most 8)
  these blowups $T(\divd)$ are very like their commutative analogues. However,  we  also need algebras that should be considered as
  blowups $T(\divd')$ of $T$ at  \emph{non-effective divisors} 
 of the form $\divd' =\divx-\divy+\tau^{-1}(\divy)$, where  $\divx$ and $\divy$ are effective divisors on $E$, 
 $0 \leq \deg \divd' \leq 8$  and    certain combinatorial conditions hold (see Definition~\ref{virtual-defn} for the details). Such a divisor will be called \emph{virtually effective}.
    
 Given   domains $U,U'$ with the same Goldie quotient ring $Q(U)=Q(U')=Q$, we say that $U$ and $U'$ 
 are  \emph{equivalent orders} \label{equiv-defn1}
 if  $aUb\subseteq U'$ and $a'U'b'\subseteq U$ for some 
   $a,b,a',b'\in Q\smallsetminus\{0\}$.
If $Q(U)=Q(V) $ for some   ring $V\supseteq U
$ then $U$ is called  
a \emph{maximal $V$-order} if there exists no ring $U\subsetneq U'\subseteq V$  equivalent to $U$.   
When $V=Q(U)$ then $U$ is simply termed a \emph{maximal order}.  
These  can be regarded as the appropriate noncommutative analogues of integrally closed domains. 
 The algebra $T$ is a maximal order.  When $Q_{gr}(U)=Q_{gr}(T)$ the concepts of maximal orders and maximal $T$-orders
 are essentially the only cases that will concern us and, as the next result shows,  they are closely connected.

  In this result, an $\mathbb{N}$-graded
$\kk$-algebra $A=\bigoplus_{n\geq 0} A_n$ is called  \emph{connected graded}
(\emph{cg})\label{cg-defn}   if  
$A_0=\kk$ and  $\dim_{\kk}A_n<\infty$ for all $n$. 
Also, for a cg algebra $U\subseteq T$, we write $\overline{U}=(U+gT)/gT$.

\begin{proposition}\label{correct-intro} {\rm (Combine Theorem~\ref{max32} with 
Proposition~\ref{correct25}.)}
 Let $U$ be a cg  maximal $T$-order,  such that $\overline{U}\not=\kk$.
 Then there exists a unique maximal order $F=F(U)\supseteq U$ equivalent to $U$. 
 Moreover, $F$ is a finitely generated $U$-module with 
   $\GKdim_U(F/U)\leq 1.$
\end{proposition}

We remark that there do exist graded maximal $T$-orders $U$ with   $U\not=F(U)$ (see Proposition~\ref{Tmax-eg}).

 Our results are most satisfactory for maximal $T$-orders, and our main result is the following complete classification of such algebras.
   
 \begin{theorem}\label{mainthm-intro1} {\rm (Theorem~\ref{max32})}
 Let $U$ be a cg   maximal $T$-order   with  $\overline{U}\not=\kk$.
 Then there exists a virtually effective divisor 
 $\divd'=\divd-\divy+\tau^{-1}(\divy)$ with $\deg(\divd')\leq 8$ such that   the  associated maximal order  
  $F(U)$    is a blowup  $F(U)=T(\divd')$ of $T$ at $\divd'$. \end{theorem}

\begin{remarks}\label{remark1}
   (1) Although in this introduction we are restricting our 
attention to the Sklyanin algebra $S=Skl$, this theorem and indeed  all  the results of this paper are  proved
simultaneously  for certain related algebras; see Assumptions~\ref{ass:T} and Examples~\ref{Skl-ex} 
for the details. 
   
(2) A   result analogous to Theorem~\ref{mainthm-intro1} also holds for  graded   maximal $T^{(n)}$-orders.

(3)
The assumption that $\overline{U}\not=\kk$ in  the theorem    is annoying but necessary (see 
Example~\ref{max6}).  It can     be bypassed at the expense of passing to a   Veronese ring 
and then regrading the 
algebra. 
 However, the resulting theorems  are not as strong as Theorem~\ref{mainthm-intro1}
   (see Section~\ref{ARBITRARY} for the details).
\end{remarks}

 One consequence of  Theorem~\ref{mainthm-intro1}  is that maximal $T$-orders have very pleasant properties.
The undefined terms in the next result are standard concepts and are defined in the body of the paper.

 \begin{corollary}\label{intro-maincor}  
 Let $U$ and $F=F(U)=T(\divd')$  be as in Theorem~\ref{mainthm-intro1}.
  \begin{enumerate}
 \item {\rm (Proposition~\ref{div-noeth} and Theorem~\ref{max32}(1))}
Both $U$ and $F$ are finitely generated $\kk$-algebras and are strongly noetherian: in other words, $U\otimes_{\kk}C$ 
and $F \otimes_{\kk} C$ are noetherian for any commutative, noetherian $\kk$-algebra $C$.\label{strong-defn}
\item  {\rm (Corollary~\ref{cohom})}  Both  $U$ and $F$  satisfy the Artin-Zhang $\chi$ conditions, 
have  finite cohomological dimension  and possess  balanced dualizing complexes.
% \item {\rm (Corollary~\ref{points})}  The point modules for $U$, and the simple objects in $\rqgr U$, are 
% parametrized by the elliptic curve $E$ together with a (possibly empty) finite scheme.
\item  {\rm (Proposition~\ref{8-special22} and Example~\ref{non-AG})}
If $F$ is the blowup at an effective divisor then $U=F$. In this case $F$ also satisfies the Auslander-Gorenstein and
 Cohen-Macaulay conditions.   These conditions do not necessarily hold  when $\divd'$ is
 virtually effective. 
\end{enumerate}
 \end{corollary}

% \clearpage

In the other direction, we prove

\begin{theorem}\label{mainthm-intro2} {\rm (Theorem~\ref{thm:converse}(3))}
For any virtually effective divisor $\divd'$ there exists a blowup   of $T$ at $\divd'$ in the sense
described above.
\end{theorem}

  The  fact that $U$ is automatically noetherian in 
  Theorem~\ref{mainthm-intro1}  is   one of the result's most striking features. 
In general, non-noetherian graded subalgebras of $T$ can be rather unpleasant  
and so, in order to   classify reasonable classes of  non-maximal orders in $T$, we make a noetherian hypothesis. 
  Given  a connected graded noetherian algebra $U$, one can easily obtain further noetherian rings by  taking Veronese rings,
   idealiser subrings 
  $\idealiser(J)=\{\theta\in U : \theta J \subseteq J\}$ for a right ideal $J$ of $U$,  or
  equivalent orders $U'\subseteq U$ containing an ideal $K$ of $U$.  
 We show that this suffices:

  \begin{corollary}\label{main-intro}
{\rm (Corollary~\ref{main})} 
   Let $U$ be a cg noetherian subalgebra of $T$ with $Q_{gr}(U)=Q_{gr}(T^{(n)})$ for some $n$.
  Assume that $\overline{U}\not=\kk$ (as in Remark~\ref{remark1} this can be assumed at the expense of taking a Veronese ring and regrading).
 
 Then $U$ can be obtained from some virtual blowup $R=T(\divd')$ by a combination of Veronese rings, idealisers
  and equivalent orders $K\subseteq U\subseteq V$, where $K$ is an ideal of $V$ with $\GKdim(V/K)\leq 1$.
  \end{corollary} 
  
       %%%%%%%%%%%%%%%%% 

 \subsection*{History}\label{subsection-history}
  We briefly explain the history behind these results and their wider relevance.
As we mentioned earlier, noncommutative curves and noncommutative analogues of the polynomial ring 
$\kk[x,y,z]$ have been classified. Motivated by these results, Artin suggested a program for classifying all 
noncommutative surfaces, but 
in order to outline this program we need some notation.
   
  Given  a   cg
domain $A$ of finite Gelfand-Kirillov dimension,   one can invert the nonzero homogeneous elements to
obtain the \emph{graded quotient ring} $Q_{gr}(A)\cong D[t,t^{-1};\alpha]$,\label{gqr-def}
 for some automorphism  $\alpha$ of the division
ring $D = Q(A)_0=D_{gr}(A)$. This division ring will be called   the \emph{function skewfield of~$A$}. 

Let $A$ be a noetherian, cg $\kk$-algebra.  A useful intuition is to regard $\rqgr A$ as the coherent sheaves over the (nonexistent)
noncommutative projective scheme $\Proj (A)$, although we will slightly abuse notation by regarding $\rqgr A$
itself  as that  scheme.  Under this intuition,   a noncommutative surface is $\rqgr A$ for a noetherian cg domain $A$
with   $\GKdim A =3$. (In fact, one should probably weaken this last condition to the assumption that  
  $D_{gr}(A)$ has lower  transcendence degree two in the sense of \cite{Zh}, but that is not  really relevant here.)
There are strong arguments for saying that noncommutative projective planes 
  are the categories $\rqgr A$, as $A$ ranges over the Artin-Schelter regular rings of dimension 3 with 
  the Hilbert series $(1-t)^{-3}$ of a polynomial ring in 3 variables (see \cite[Section~11.2]{SV} for more details). These
  are the algebras classified in \cite{ATV} and  for which the Sklyanin algebra $S=S(a,b,c)$ is the generic example.
In \cite{VdB2, VdB3} Van den Bergh has similarly classified  noncommutative analogues of quadrics and related surfaces.

% \clearpage

\begin{artin} 
  Artin    conjectured that the only  function skewfields of noncommutative surfaces 
   are the following:
\begin{enumerate} 
\item[(i)] division rings $D$  finite dimensional over their centres $F=Z(D)$, which are then fields of transcendence degree two;
\item[(ii)]  division rings of fractions of Ore extensions $\kk(X)[z; \sigma,\delta]$ for some curve $X$;
\item[(iii)]  the function skewfield $D=D_{gr}(S)$ of a Sklyanin algebra $S=S(a,b,c)$. In this case we assume  that  
$S$  is not a finite module over its centre.
\end{enumerate}

Artin  then asked for a classification of the   noncommutative surfaces $\rqgr A$ within each birational class;
that is,  the  cg noetherian algebras $A$ with  $D_{gr}(A)$ being a fixed division ring from this list.   
\end{artin}

The case of Artin's programme when $D=\kk(Y)$ is the function field of a surface and $\GKdim A=3$ 
has been completed in \cite{RS,Si} (if one strays from
algebras of Gelfand-Kirillov dimension 3, then things become more complicated,  as   \cite{RSi} shows). 
As explained earlier, in this paper we are interested 
in the other extreme,  that of case (iii)  from Artin's list.

The first main results in this direction come from \cite{Rog09}, of which this paper is a continuation. In particular 
\cite[Theorem~1.2]{Rog09} shows that the 
 maximal orders  $U\subseteq T=S^{(3)}$ that have $Q_{gr}(U)=Q_{gr}(T)$ and are generated in degree one
 are just the  blowups $T(\divd)$,  for an effective divisor $\divd$  on $E$ with    $\deg(\divd)\leq 7$. 
 We remark that in this case $T(\divd)$ is simply the subalgebra of $T$ generated by 
 those elements of $T_1$ whose images in $T/gT$ vanish on $\divd$.  
As such, $T(\divd)$ is quite similar to a commutative blowup and 
 $\rqgr T(\divd)$ also coincides with Van den Bergh's more categorical version of a blow-up in \cite{VdB}. In this paper we will also need 
 $T(\divd)$ when $\deg(\divd) = 8$, and this is harder to describe as it is not generated in degree one. Its construction and 
 basic properties are  described in the companion paper \cite{RSS2}.

  %%%%%%%%%%%%%%%%

  \subsection*{The proofs}\label{subsection-proofs}    For simplicity we assume here that $U$ is a cg subalgebra of 
  $T$ with $Q_{gr}(U)=Q_{gr}(T)$.

A key strategy in the description  of the Sklyanin algebra $S$, and  in  the classification of 
noncommutative projective planes in \cite{ATV}, was  to understand the factor ring $S/gS$, where $g\in S_3=T_1$ is the 
the central element   mentioned earlier. Indeed, one of the main steps  in that paper  was to show that 
$S/gS\cong B(E,\sL,\sigma)$ for the appropriate $\sL$ and $\sigma$.
We apply a similar strategy.  The nicest case is when $U\subseteq T$ is \emph{$g$-divisible} in the sense that $g\in U$ 
and $U\cap gT=gU$. In particular, $\overline{U}=U/gU$ is then a subalgebra of $\overline{T}$ with $\GKdim(\overline{U})=2$.
As such \emph{$\overline{U}$ and hence $U$ are automatically noetherian} (see Proposition~\ref{div-noeth}). Much of this paper 
concerns the classification of $g$-divisible algebras $U$ and the starting point    is the following result.

\begin{theorem}\label{intro4} {\rm (Theorem~\ref{thm-correct})  }
Let $U$ be a $g$-divisible subalgebra of $T$ with 
$Q_{gr}(U)=Q_{gr}(T)$. Then $U$ is an equivalent order to some blow-up $T(\divd)$ at an effective divisor $\divd$ on $E$
with $\deg \divd\leq 8$.
\end{theorem}

It follows easily from this result that  
\emph{a $g$-divisible maximal $T$-order $U $ equals $\End_{T(\divd)}(M)$
 for some finitely generated  right $T$-module $M$} (see Corollary~\ref{correct3}).  
When $U$ is  $g$-divisible, the rest of the proof of Theorem~\ref{mainthm-intro1} amounts to 
showing that,  
\emph{up to a finite dimensional vector space, 
$\overline{U}=B(E, \sM(-\divd'),\tau)$, for some virtually effective divisor 
$\divd'=\divd-\divy+\tau^{-1}(\divy)$} (see Theorem~\ref{thought9}).  
 This is also  the key property in the definition of  a 
blowup at such a divisor (see Definitions~\ref{def:blowup} and ~\ref{virtual-defn} for 
more details).

  Now suppose that $U$ is  not necessarily $g$-divisible and set $C=U\langle g\rangle$ with \emph{$g$-divisible hull}
  $$\wh{C}=\{\theta\in T : g^m\theta\in C  \text{ for some } m\geq 0\}.$$  The remaining step in the proof of 
  Theorem~\ref{mainthm-intro1} is to show that $U$, $C$ and $\wh{C}$ are equivalent orders. This in turn follows from the 
  following fact.
  \emph{Let $V$ be a   graded subalgebra of  $ T$ with $g\in V$ and $Q_{gr}(V)=Q_{gr}(T)$. Then $V$
  has a minimal \cog\ in the sense that $V$  has a unique ideal $I$ minimal 
with respect to $\GKdim(V/I)\leq 1$ and $V/I$ being $g$-torsionfree} (see Corollary~\ref{max30}).

  \subsection*{Further results} 
  The $g$-divisible subalgebras of $T$ are closely related to 
   subalgebras of the (ungraded) localised ring $T^\circ = T[g^{-1}]_0$. 
  The algebra $T^\circ$ is a hereditary noetherian domain of GK-dimension 2 
  and can be thought of as a noncommutative coordinate ring of the affine space $\mathbb{P}^2\smallsetminus E$.
  By \cite{RSS}, any subalgebra of $T^\circ$ is noetherian and so the algebras $U^\circ =U[g^{-1}]_0 \subseteq T^\circ$ 
  give a plentiful supply of noetherian domains of GK-dimension 2.
  All the above results   have parallel versions for orders in $T^\circ$.  For example:
  
  \begin{corollary}\label{affine-intro1}
 {\rm (Corollary~\ref{localthm} and  Corollary~\ref{A-cog})}   Let $A$ be a subalgebra of $T^\circ $ with   $Q(A)=Q(T^\circ).$

  \begin{enumerate}
  \item  The algebra  $A$ has finitely many prime ideals and  DCC on ideals.  
  \item  If $A$ is a maximal $T^\circ$-order then $A=T(\divd')^\circ$  for   some   
  virtually effective divisor $\divd'$. 
  \end{enumerate}
  \end{corollary}

 %%%%%%%%%%%%%%%%%  %%%%%%%%%%%%%%%%%

   \subsection*{Organisation of the paper} 
   In Section~\ref{HATTRICK} we prove basic technical results, including the important, though easy, fact  that any $g$-divisible
    subalgebra of $T$ is strongly noetherian (see Proposition~\ref{div-noeth}).
   Section~\ref{CURVES} is devoted to studying finitely generated graded orders in $\kk(E)[t;\tau]$.
   The main result (Theorem~\ref{thm:projcurve})  shows that any such order  is (up to finite dimension) an idealiser in a twisted homogeneous coordinate ring. This 
   improves on  one of the main results from  \cite{AS} and has useful applications to the study of  point modules over such an algebra.  
   Section~\ref{RIGHT IDEALS} incorporates needed results from \cite{RSS2} about right ideals of~$T$ and the blowups $T(\divd)$ at effective divisors.  
   
   Sections~\ref{equiv2-sec}--\ref{EXISTENCE} are devoted to $g$-divisible algebras in $T$.
   The main result of Section~\ref{equiv2-sec} is Theorem~\ref{intro4} from above.  
   Section~\ref{MAX-SECT} is concerned with  the structure of $V=\End_{T(\divd)}(M)$, where $M\subset T$ is a
    reflexive $T(\divd)$-module and $\divd$ is effective. Most importantly, Theorem~\ref{thought9} describes 
    the factor $V/gV$.
   Section~\ref{EXISTENCE}  pulls these results together,  proves Theorem~\ref{mainthm-intro2}
   for $g$-divisible algebras and draws various conclusions. 
   
       In Section~\ref{MAX1} we show that various algebras have 
   minimal \cogs. This is then used to complete the proof of Theorem~\ref{mainthm-intro1}.
   Section~\ref{ARBITRARY} studies subalgebras of the Veronese rings  $T^{(m)}$ and algebras $U$ with $\overline{U}=\kk$.  We    apply this to prove Corollary~\ref{main-intro}. Finally, Section~\ref{EXAMPLES} is devoted to examples.
   At the end of the paper we also provide an index of notation.
   
   \subsection*{Acknowledgements}
   Part of this material is based upon work supported by the National Science Foundation under Grant No. 0932078 000, while the authors were in residence at the Mathematical Science Research Institute (MSRI) in Berkeley, California, during the Spring 
   semester of 2013.   During this trip, Sierra was also partially supported by the Edinburgh Research Partnership in Engineering and Mathematics, and Stafford was partially supported by the Clay Mathematics Institute
   and  Simons Foundation.    The authors gratefully acknowledge the   support of all these organisations.

%%%%%%%%%%%%%%%%%%%
%%%%%%%%%%%%%%%%%%%

\section{Basic results}
\label{HATTRICK}

In this section we collect the basic definitions and results that will be used throughout the paper.

Throughout the paper $\kk$ is an algebraically closed field and all rings will be $\kk$-algebras.
If $X$ is a projective $\kk$-scheme, $\mc{L}$ is an invertible sheaf on $X$, and $\sigma: X \to X$ 
is an automorphism, then there is a  \emph{TCR}  or \emph{twisted homogeneous coordinate ring}    
\label{TCR-defn}
$B = B(X, \mc{L}, \sigma)$ associated to this data and defined as follows.  Write  
$\mc{F}^{\sigma} = \sigma^*(\mc{F})$  for a pullback of a sheaf  $\mc{F}$ on $X$ and set
$\mc{F}_n = \mc{F} \otimes \mc{F}^{\sigma} \otimes \dots \otimes \mc{F}^{\sigma^{n-1}}$ for $n\geq 1$.  Then 
\[
B = \bigoplus_{n \geq 0} H^0(X, \mc{L}_n), \ \ \  \text{with product}\  \ \ x * y = x \otimes (\sigma^m)^*(y), \ \ \text{for}\  \ x \in B_m, y \in B_n.
\]
In this paper  $X = E$ will usually be a smooth elliptic curve, 
and a review of some of the important properties of   $B(E, \mc{L}, \sigma)$ in this case 
can be found in \cite{Rog09}.    It is well-known, going back to \cite{ATV},  that much of the structure of the Sklyanin 
algebra $S$ is controlled by the factor ring $S/gS\cong B(E, \mathcal{L},\sigma)$, and this in turn can
 be analysed  geometrically.  
 
In fact, there are several different families of Sklyanin algebras, and we first  set up a framework which will allow our results  to apply to subalgebras of  any of these   (and, indeed, more generally). 
Recall that for an $\mb{N}$-graded ring $R = \bigoplus_{n \geq 0} R_n$  
the \emph{$d^{th}$ Veronese ring},  for $d\geq 1$, is $R^{(d)} = \bigoplus_{n \geq 0} R_{nd}$.  Usually this is   graded 
by setting $R^{(d)}_n=R_{nd}$.  However,  we will sometimes want to regard $R^{(d)}$ as a graded subring of $R$, 
in which case each $R_{nd}$ maintains its degree $nd$; we will call this the \emph{unregraded} Veronese ring.
\label{unregraded-defn}
In this paper it will be easier to work  with the 3-Veronese ring of the Sklyanin $T=S^{(3)} =\bigoplus_{n\in \mathbb{Z}} T_n$, largely because this ensures that  the canonical central element  $g\in T_1$.
Similar comments will apply to the other families, and so in the body of the paper we will work with algebras satisfying the following hypotheses.   

\begin{assumption}
\label{ass:T}
Let $T$ be a cg $\kk$-algebra which is a domain with a central element $g \in T_1$, such that there is a graded isomorphism
$T/gT \cong B = B(E, \mc{M}, \tau)$ for a smooth elliptic curve $E$, invertible sheaf $\mc{M}$ 
with $\mu = \deg \mc{M} \geq 2$, and infinite order automorphism $\tau$. 
Such a $T$ is called an {\em elliptic algebra of degree $\mu$}. 
\end{assumption}
\noindent
We will assume that Assumption~\ref{ass:T} holds  throughout the paper.
In the language of \cite{VdB}, the assumption can be interpreted geometrically 
to say that the surface $\rqgr T$ contains the commutative elliptic curve $\rqgr B \simeq \coh E$ as a divisor.
We will need  stronger conditions on $T$ in the main results of Section~\ref{MAX1} (see Assumption~\ref{ass:T2}).

\begin{examples}\label{Skl-ex}
The hypotheses of  Assumption~\ref{ass:T} are satisfied in a  number of examples, in particular for 
 Veronese rings of the following types of Sklyanin  algebras.

\begin{enumerate}
\item Let $S$ be the quadratic Sklyanin algebra
\[
S(a,b,c) = \kk\{x_0,x_1,x_2\}/(ax_ix_{i+1}+bx_{i+1}x_{i}+cx^2_{i+2} : i\in \mathbb{Z}_3),
\]
for appropriate $[a, b, c] \in \mathbb{P}^2_k$, and let $T = S^{(d)}$ for $d=3$.

\item Let $S$ be the cubic Sklyanin algebra
\[
S(a, b, c) = \kk \{x_0, x_1 \}/(ax_{i+1}^2x_i + b x_{i+1}x_ix_{i+1} + a x_ix_{i+1}^2 + c x_i^3 : i \in \mathbb{Z}_2),
\]
for appropriate $[a, b, c] \in \mathbb{P}^2_k$ and let $T = S^{(d)}$, for $d=4$. 

\item Let $x$ have degree $1$ and $y$ degree $2$, and  set
\[
S=S(a, b, c) = \kk \{ x, y \}/(ay^2x + cyxy + axy^2 + bx^5, \, ax^2y + cxyx + ayx^2 + by^2),
\] 
for appropriate $[a, b, c] \in \mathbb{P}^2_k$, and let $T = S^{(d)}$, for $d=6$.  

 \item There are other   examples satisfying these hypotheses; for example take $T=B(E,\sM,\tau)[g]$,
 where $\sM$ is an invertible sheaf on the elliptic curve $E$ with $\deg \sM\geq 2$ and $|\tau|=\infty$.
  
\end{enumerate}
The restrictions on the  parameters $\{a,b,c\}$ in (1--3) are determined as follows.
In each case, there exists a central element $g \in S_d$ such that 
$S/gS \cong B=B(E, \mc{L}, \sigma)$,\label{sigma-defn}  for some $\mathcal{L}$ and $\sigma$.
This factor ring also determines the Sklyanin algebra, since 
$g$ is the unique relation for $B$ of degree $d$. The requirements on $\{a,b,c\}$ are 
  that $E$ is an elliptic curve and that 
 $|\sigma|=\infty$. Explicit criteria on the parameters are known for $E$ to be an elliptic curve but not for $|\sigma|=\infty$; nevertheless this will be the case when the parameters are generic.  
 In these examples 
$\deg \mc{L} = 3, 2, 1$, respectively and hence  $T/gT \cong B(E, \mc{M}, \sigma^d)$, where $\mc{M} = \mc{L}_d $ 
has degree $\mu=d\cdot (\deg \mc{L}) =9,8,6$, respectively.  
%That each  formula for $\varphi$ defines an involutive anti-automorphism of $S$ (and hence $T$) is
% immediate from   the relations.
The details for these comments can be found in \cite{ATV,ATV2,Ste}.
\end{examples}

\medskip 
 
\begin{notation}\label{quot-not}
All  algebras $A$ considered in this paper are domains of finite Gelfand-Kirillov dimension,
written $\GKdim(A)$. If   $A$ is graded, 
then  the set $\mathcal{C}$  of non-zero homogeneous elements therefore forms an Ore set
 (see \cite[Corollary~8.1.21]{MR}
and \cite[C.I.1.6]{NV}). By  \cite[A.14.3]{NV}   the localisation 
$Q_{gr}(A)=A\mc{C}^{-1}$ is a \emph{graded division ring} in the sense that $Q_{gr}(A)$ is an 
Ore extension $Q_{gr}(A)=D[z,z^{-1};\alpha]$  of a division ring $D$ by an
 automorphism $\alpha$; thus $zd=d^\alpha z$  for all $d\in D$. 
The algebra $D$ will be denoted $D=D_{gr}(A)$ and called the \emph{function skewfield of $A$},  while $Q_{gr}(A)$ will be called the \emph{graded quotient ring of $A$}.  
\end{notation}

\begin{notation}\label{homomorphism-notation}  For the most part the algebras $A$ considered in this paper 
will be connected graded, in which case we usually work in the category 
  $\rGr A$   of $\mb{Z}$-graded right $A$-modules, with homomorphisms $\Hom_{\rGr A}(M,N)$
   being graded of degree zero. In particular an isomorphism of graded modules or rings will be assumed to be graded of degree zero, unless otherwise stated.
   The category of  noetherian graded right $A$-modules will be written $\rgr A$, \label{rgr-defn}
   while the category of ungraded
    modules will be written $\text{Mod-}A$, and we reserve the term $\Hom(M,N)=\Hom_A(M,N)$ \label{hom-defn}
    for homomorphisms in the ungraded category.  
 For $M,N\in \rGr A$,  
  the \emph{shift} $M[n]$ \label{shift-defn}  is defined by 
 $M[n]=\bigoplus M[n]_i$ for $M[n]_i=M_{n+i}$. 
Similar comments apply to $\Ext_{\rGr A}$ and $\Ext_{A}$ as well as to $\End_A(M)=\Hom_A(M,M)$.  
If $\fd A$ denotes the category of finite dimensional (right) $A$-modules,
then we write $\rqgr A$  for the quotient category $\rgr A/\fd A$. Similarly, $A \lqgr = A \lgr/ A \lfd$
is the quotient category of noetherian graded left modules modulo finite-dimensional modules.
\label{rqgr-defn}
The basic properties 
of this construction can be found in \cite{AZ}.
\end{notation}

 \medskip
\begin{notation}\label{startup}  Write $T_{(g)}$ for the homogeneous  localisation of $T$ at the completely prime ideal $gT$; 
thus $T_{(g)}=T\mathcal{C}^{-1} $ for   $\mathcal C$  the set of  homogeneous elements in $T\smallsetminus gT$.
Note that $T_{(g)}/gT_{(g)} \cong Q_{gr}(B) = \kk(E)[t, t^{-1}; \tau]$, a ring of twisted Laurent polynomials over the function field 
of $E$.   In particular, $T_{(g)}/gT_{(g)}$  is a graded division ring and by  \cite[Exercise~1Q]{GW1989} it is also simple as an ungraded ring. 
Also, as will be used frequently in the body of the paper, 
 \begin{equation}\label{startup1}
 \text{\it the only graded right or left ideals 
 of $T_{(g)} $ are the $g^nT_{(g)}$.}
 \end{equation}

For any graded vector subspace $X \subseteq T_{(g)}$,  set  
$$\widehat{X}=\{t\in T_{(g)} | tg^n\in X \text{ for some } n\in\mathbb{N}\}.$$
We say that $X$ is {\em $g$-divisible}\label{g-div}
 if $X \cap gT_{(g)} = gX$.   
Note that 
if $X$ is $g$-divisible and $1 \in X$ (as happens when $X$ is a subring of $T_{(g)}$), then $g \in X$.  
For any $\kk$-subspace $Y$ of $T_{(g)}$, write $\overline{Y}=(Y+gT_{(g)})/ gT_{(g)}$
 for the image of $Y$  in
$ T_{(g)}/gT_{(g)}$.

If $R \subseteq T_{(g)}$ is a subalgebra with $g \in R$, then the \emph{$g$-torsion} submodule \label{g-tors}
of  a right $R$-module $M$
  is $\tg(M) = \{m \in M | g^n m = 0\ \text{for some}\ n \geq 1 \}$.  We say that 
  $M$ is \emph{$g$-torsionfree} if $\tg(M) = 0$ and 
\emph{$g$-torsion} if $\tg(M) = M$. \end{notation}

We notice that the rings $T$ automatically satisfy some useful additional properties.
An algebra $C$ is called  \emph{just infinite} \label{ji-defn}
if every  every nonzero ideal $I$ of $C$  satisfies $\dim_{\kk}C/I<\infty$.

\begin{lemma}
\label{lem:T-prop}
Let $T$ satisfy Assumption~\ref{ass:T}.  Then
\begin{enumerate}
\item $T$ is generated as an algebra in degree $1$.
\item Any finitely generated, cg subalgebra  of $Q_{gr}(T/gT)=\kk(E)[z,z^{-1};\tau]$, in particular $T/gT$ itself, is 
just infinite. 
\end{enumerate}
\end{lemma}
\begin{proof}
(1)  Since $ \mu\geq 2$,
the ring $B = T/gT \cong B(E, \mc{M}, \tau)$ is generated in degree 1 \cite[Lemma~3.1]{Rog09}. 
  Thus $T_2 = (T_1)^2 + gT_1 = (T_1)^2$ and, by induction,  $(T_1)^n = T_n$ for all $n \geq 1$.

(2) This follows from \cite[Corollary~2.10 and Section~3]{RSS}.
\end{proof}

\noindent

 As the next few results show, $g$-divisible algebras and modules have  pleasant properties. 
 The first   gives a useful, albeit easy, alternative characterisation of $\wh{X}$ that will be used without particular reference.

\begin{lemma}\label{lem-triv}
Let $R \subseteq T_{(g)}$ be a cg subalgebra with $g \in R$, and let 
$X \subseteq T_{(g)}$ be a graded right $R$-module.  Then $X \subseteq \wh{X}$, and $\wh{X}$ is also 
a right $R$-module.    Moreover:
$$\qquad \text{$X$ is $g$-divisible\ \ 
$\iff$ \  $X = \wh{X}$\ \ 
$\iff$ \  $T_{(g)}/X$ is a $g$-torsionfree $R$-module.} \quad \qed$$
\end{lemma}

\begin{proposition}\label{div-noeth}
{\rm (1)}  If  $R$ is any  $g$-divisible  cg   subalgebra of $T$, then $R$  is  finitely generated as a $\kk$-algebra.

\begin{enumerate}
\item[{\rm (2)}]  Let $R$ be a  finitely generated  $g$-divisible  cg   subalgebra of $T_{(g)}$.  Then $R$ is   strongly noetherian. \end{enumerate}
\end{proposition}

\begin{proof}  (1)   We have $\overline{R} \cong  (R+gT)/ gT \subseteq \overline{T}\cong B(E,\sM,\tau)$  and so
\cite[Theorem~2.9]{RSS} implies that $\overline{R}$ is noetherian.
 By \cite[Lemma~8.2]{ATV}, $R$ is noetherian. Since 
the generators of $R_{\geq 1}$ as an $R$-module also 
generate $R$ as a $\kk$-algebra, $R$ is  
finitely generated as a $\kk$-algebra.

(2)  In this case,  $\bbar{R} =R/gR \cong  (R+gT_{(g)})/ gT_{(g)} \subseteq Q_{gr}(B) = \kk(E)[t, t^{-1}; \tau]$.
By \cite[Corollary 2.10]{RSS}  $\overline{R}$ is noetherian.   Also
$\GKdim \overline{R} \leq 2$, for instance by \cite[Theorem 0.1]{AS}, and so 
$\overline{R}$ is strongly noetherian by \cite[Theorem 4.24]{ASZ}.
Thus $R$ is strongly noetherian by \cite[Lemma~8.2]{ATV}. 
  \end{proof}

 \begin{lemma}
\label{lem:expectQ} 
Let $R$ be a $g$-divisible cg subalgebra of $T_{(g)}$ with $D_{gr}(R) = D_{gr}(T)$.  Then
\begin{enumerate}
\item $Q_{gr}(R)=Q_{gr}(T)$ and 
\item  $Q_{ gr}(\overline{R}) = Q_{gr}(\overline{T})$.
\end{enumerate}
\end{lemma}

\begin{proof}  (1) As $g\in R_1$ we have $Q_{gr}(T)=D_{gr}(T)[g,g^{-1}]
=D_{gr}(R)[g,g^{-1}] = Q_{gr}(R).$

(2)
Since $Q_{gr}(R)=Q_{gr}(T)$, there exists $0 \neq x \in R_d$ such that $x T_1 \subseteq R_{d+1}$.
Then $\overline{x} \overline{T}_1 \subseteq \overline{R}$.  As long as $\overline{x} \neq 0$, this shows that the graded quotient ring of $\overline{R}$ contains
a generating set for $\overline{T}$ and we are done.  On the other hand,  if $\overline{x} = 0$, then write 
 $x = g^iy$ with $y \in  T_{(g)} \smallsetminus g T_{(g)}$; equivalently $y \in R\smallsetminus gR $ by $g$-divisibility.
  Then $g^iy T_1 \subseteq R \cap g^iT_{(g)} = g^iR$, and so $y T_1 \subseteq R$. Thus  we are again done.
\end{proof}
 
If $A$ is a cg domain with graded quotient ring $Q=Q_{gr}(A)$ and 
  $M\subseteq Q$ is  a  finitely generated graded right $A$-submodule, we can and always will identify 
\begin{equation}\label{endo-defn}
\End_A(M)\ = \ \{q\in Q : qM\subseteq M\} \qquad\text{and}\qquad M^*=\Hom_A(M,A)=\{q\in Q : qM\subseteq A\}.
\end{equation}
Clearly  both $\End_A(M)$ and $M^*$ are  graded subspaces of $Q$.
 
\begin{lemma}\label{Sky2-1}  
Let $R$ be any $g$-divisible subring of $T_{(g)}$ with $Q_{gr}(R) = Q_{gr}(T_{(g)})$, and let $M, M' \subseteq T_{(g)}$ be finitely generated nonzero right $R$-modules. 
\begin{enumerate}
\item  If $M \nsubseteq g T_{(g)}$, then we can identify 
$ \Hom_R(M,M') = \{x \in T_{(g)} | x M \subseteq M' \} \subseteq T_{(g)}$.
\item If $M'$ is $g$-divisible, and $M \not \subseteq g T_{(g)} $ (in particular if $M$ is $g$-divisible) then $\Hom_R(M,M') \subseteq T_{(g)}$ is also $g$-divisible.
 
\item If $M$ is $g$-divisible, then $U = \End_R(M) \subseteq T_{(g)}$ is $g$-divisible, and $M$ is a finitely generated left $U$-module.  
Moreover, $\overline{U} \subseteq \End_{\overline{R}}(\overline{M})$.
\end{enumerate}
\end{lemma}

\begin{proof} 
(1)  
  Since  $M\not\subseteq gT_{(g)}$, it follows from \eqref{startup1}  that $MT_{(g)}=T_{(g)}$.  
In particular, $N = \Hom_R(M, M') \subseteq \Hom_{T_{(g)}}(MT_{(g)}, M'T_{(g)}) \subseteq T_{(g)}$.

(2)  Part~(1) applies 
and so $N= \Hom_R(M,M')  \subseteq T_{(g)}$.   
Next, let $\theta\in N \cap gT_{(g)}$; say $\theta=gs$ for some $s\in T_{(g)}$. Then $sgM = \theta M \in  M' \cap gT_{(g)}=M'g$  
since $M'$ is $g$-divisible.  Hence $sM\subseteq M'$ and $s\in N$.  Thus $N \cap g T_{(g)} = gN$.

(3)  By Part (2),  $U$ is $g$-divisible, and hence is  noetherian by Proposition~\ref{div-noeth}.  As  $Q_{gr}(R)=Q_{gr}(T_{(g)})$,  there exists  $x \in T_{(g)}\smallsetminus\{0\}$ so that $x M \subseteq R$.  
 Then $MxM \subseteq MR = M$. Hence (up to a shift)
 $M\cong Mx \subseteq U$  is finitely generated as a left $U$-module.

Now $\overline{U} =(U + gT_{(g)})/gT_{(g)} \subseteq \overline{T_{(g)}} = \kk(E)[t, t^{-1}; \tau]$.  Since $Q_{gr}(\overline{U}) = Q_{gr}(\overline{T_{(g)}})$ by 
Lemma~\ref{lem:expectQ}, 
as in \eqref{endo-defn} we identify $\End_{\overline{R}}(\overline{M})$ with 
$\{x \in \overline{T_{(g)}} | x \overline{M} \subseteq \overline{M} \}$.  But since 
$U M \subseteq M$, clearly $(\overline{U})( \overline{M}) \subseteq \overline{M}$.
\end{proof}
 
\begin{lemma}\label{lem:finehat}
Let $R$ be a graded subalgebra of $T_{(g)}$ with $Q_{gr}(R)=Q_{gr}(T_{(g)})$  and let  $M \subseteq T_{(g)}$ be a graded  right $R$-submodule of $T_{(g)}$ such that $M\not\subseteq gT_{(g)}$.  Then  
\begin{enumerate} 
\item  For any $x\in T_{(g)}\smallsetminus gT_{(g)}$, we have  $\wh{xM} = x \wh{M}$.
\item  If $R$ is $g$-divisible and $M$ is a finitely generated  $R$-module, then  so is  $\wh{M}$.

\item If $R$ is $g$-divisible,   then $T_{(g)}\supseteq M^*=\wh{M^*}$ and $M^*\not\subseteq gT_{(g)}$. Hence 
 $T_{(g)}\supseteq M^{**}=\wh{M^{**}}$.  Moreover, we have $(\wh{M}\,)^*=M^*$ and
 $(\wh{M}\,)^{**}=M^{**}.$  
\end{enumerate}
\end{lemma}

\begin{proof}  
(1) Let   $r \in \wh{M}$.   For some $n$
 we have $rg^n  \in M$, so $xrg^n  \in xM$. Since $xr\in T_{(g)}$ it follows that  $xr \in \wh{xM}$.  Conversely, if $r \in T_{(g)}$ with $rg^n  \in xM$, then 
 $rg^n=g^nr \in g^n T_{(g)} \cap xT_{(g)} $. As $gT_{(g)}$ is a completely prime ideal and $x\not\in gT_{(g)}$,
clearly $  g^n T_{(g)} \cap xT_{(g)} = g^n x T_{(g)}$. 
Thus  $r = xs$ for some $s\in T_{(g)}$ and
$xM\ni  rg^n  =xs g^n $. Therefore  $sg^n  \in M$ whence $s \in \wh{M}$ and $r\in x\wh{M}$.  Thus $\wh{xM} = x \wh{M}$, as claimed.  

(2) As in the proof of Lemma~\ref{Sky2-1}, there exists  $x \in T_{(g)}\smallsetminus\{0\}$ so that $x M \subseteq R$.  
If $x=gy $ for some $y\in T_{(g)}$, then 
$g(yM)\subseteq R$ and so $yM\subseteq R$ since $R$ is $g$-divisible.  Thus we can assume that $x\in T_{(g)}\smallsetminus gT_{(g)}$.
Again by $g$-divisibility,   $\wh{xM}  \subseteq \wh{R} = R$.  By Proposition~\ref{div-noeth} $\wh{xM}$  is a finitely generated right ideal of $R$.  Up to a shift   $\wh{M}\cong  x \wh{M}=\wh{xM} $ by  (1).  This  is    finitely generated as an $R$-module. 

(3)  By Lemma~\ref{Sky2-1}(2), $M^* = \Hom_R(M, R) \subseteq T_{(g)}$ and is $g$-divisible, i.e. 
$M^*=\wh{M^*}$.   Clearly then $M^* \nsubseteq gT_{(g)}$, and so by the left-handed analog of the same 
argument, $M^{**}=\wh{M^{**}} \subseteq T_{(g)}$ also.
 
Now as $M\subseteq \wh{M}$, certainly $(\wh{M}\,)^*\subseteq M^*$. On the other hand, if 
 $\theta\in M^*$ and $x\in \wh{M}$, say with $xg^n\in M$, then $(\theta x)g^n=\theta(xg^n)\in R= \wh{R}$. 
Hence  $\theta x  \in R.$ 
Thus  $\theta\in (\wh{M}\,)^*$ and $(\wh{M}\,)^*=M^*$. Taking a second dual gives  $(\wh{M}\,)^{**}=M^{**}$.
\end{proof}

We note next some special properties of modules of GK-dimension 1. 

\begin{lemma}
\label{lem:constant}
Let $R$ be a cg $g$-divisible subalgebra of $T_{(g)}$ and suppose that  $M$ is a finitely generated, $g$-torsionfree
 $R$-module with $\GKdim(M) \leq 1$. 
Then the Hilbert series of $M$ is eventually constant; that is, $\dim_{\kk} M_n = \dim_{\kk} M_{n+1}$ for all 
$n \gg 0$.  Moreover, $M$ is a finitely generated $\kk [g]$-module.
\end{lemma}
\begin{proof}
By \cite[Proposition~5.1(e)]{KL}, $\GKdim(M/Mg) \leq   0$ and so $\dim_{\kk}M/Mg<\infty$.  Thus $M_rg=M_{r+1}$  for all $r\gg0$; say for  $r \geq n_0$.   
In particular,  $M=M_{\leq n_0}\kk[g]$.
Moreover, since multiplication by $g$ is an injective  map from $M_r$ to $M_{r+1}$, 
it follows that  $\dim_{\kk} M_r = \dim_{\kk} M_{r+1}$ for all  for $r \geq n_0$.\end{proof}

 A graded  ideal $I$ in a cg algebra $R$ is called a \emph{\cog} \label{cog-defn}  if 
 $\GKdim(R/I)  =  1$ (these are called special ideals in \cite{Rog09}). 
The name is justified since, as will be shown in Section~\ref{MAX1}, 
orders in $T$ have very few such ideals.  
The next lemma will be useful in  understanding them.

 \begin{lemma}\label{GK-results}
Let $R$ be a $g$-divisible finitely generated cg subring of $T_{(g)}$ with $Q_{\gr}(R) = Q_{\gr}(T)$. 
 Then:
 \begin{enumerate}
 \item If $J$ is a non-zero $g$-divisible graded ideal of $R$, then $\GKdim(R/J)\leq 1$.
 \item Conversely, if $J$ is a graded ideal of $R$ such that $\GKdim(R/J)\leq 1$ then $\wh{J}/J$ is finite dimensional.
  \item If $K$ is any ideal of $R$ then $K=g^nI$, for some $n\geq 1$ and ideal $I$ satisfying $\GKdim(R/I)\leq 1.$
\item Suppose that $L, M$ are graded subspaces of $T_{(g)}$ with $L \nsubseteq gT_{(g)}$ and $M \nsubseteq g T_{(g)}$ 
and assume that $I = LM$ is an ideal of $R$. Then $\GKdim(R/I) \leq 1$.
 \end{enumerate}
 \end{lemma}
 
 \begin{proof}
 
 (1)  By Lemma~\ref{lem:expectQ}, $\overline{R} \subseteq k(E)[t, t^{-1}; \tau] = Q_{gr}(\overline{R})$ and,
 by Lemma~\ref{lem:T-prop}(2), 
$\overline{R}$ is  just infinite.
Since $J$ is $g$-divisible, $J \nsubseteq gR$ and so $\overline{J} \neq 0$; thus $\dim_\kk \overline{R}/\overline{J} < \infty$.   Equivalently, if 
 $R'=R/J$ then   $\dim_{\kk}R'/gR'<\infty$. It follows that $R'_m=gR'_{m-1}$ 
 for all $m\gg0$ and hence that   $\GKdim R'\leq 1$.

 (2)  Once again,   $\overline{R}$ is just infinite. Thus,  since $J \subseteq gR$ would lead to the contradiction 
$\GK(R/J) \geq 2$, we must have $\dim_\kk R/(gR + J) = \dim_\kk \overline{R}/\overline{J} < \infty$. 
   Since $\wh{J}$ is noetherian, $g^n\wh{J}\subseteq J$ for some $n$.
If $J'  $ is the largest right ideal inside $\wh{J}$ such that $J'/J$ is finite dimensional, then $J'$ is an ideal and we can replace 
$J$ by $J' $ without loss. If we still have $J\not=\wh{J}$, then  there exists  $x\in\wh{J}\smallsetminus J$ such that $xg\in J$. Thus   
$x(gR+J)\subseteq J$,   
and left multiplication by $x$ defines a surjection $R/(gR+J) \twoheadrightarrow (xR+J)/J$.  We have $\dim_\kk(xR+J)/J = \infty$ and $\dim_\kk R/(gR+J) < \infty$, a contradiction.  Thus $\wh{J} = J$.

  (3) Write $K=g^nJ$ with $n$ as large as possible and $J$ an ideal of $R$. 
 Then $J\not\subseteq gR$ and so Lemma~\ref{lem:T-prop}(2) again 
 implies that $\dim_{\kk}\overline{R}/\overline{J}<\infty$ and hence $\GKdim R/J \leq 1$.

(4)  
 Since $g T_{(g)}$ is completely prime, $I  =LM \nsubseteq gT_{(g)}$ and hence  $I \nsubseteq gR$.
Now apply  Part~(3).
 \end{proof}

Next, we want to prove some general results about equivalent orders that will be useful elsewhere.
We recall that two cg domains $A$ and $B$  with a common (graded)   quotient ring $Q=Q_{gr}(A)=Q_{gr}(B)$ are \emph{equivalent orders}\label{equiv-defn}
 if $aAb\subseteq B$ and $cBd\subseteq A$ for some   $a,b,c,d\in Q\smallsetminus\{0\}$. Clearing denominators on the appropriate sides, one  can always assume that $a,b,c,d\in B$.  One can also  assume that $a,b,c,d$ are homogeneous;
 indeed,  if $a$ and $b$ have leading terms $a_n$ and $b_m$ then $a_nAb_m\subseteq B$.

\begin{proposition}\label{thm:equivord}
Suppose that $U\subseteq R$ are $g$-divisible cg  finitely generated subalgebras of $T_{(g)}$  such that  $Q_{gr}(U)=Q_{gr}(R) = Q_{gr}(T_{(g)})$. 
Then the following are equivalent:
\begin{enumerate}
\item $U $ and $R $ are equivalent orders in $Q_{gr}(U)=Q_{gr}(T)$;
 \item $ U/gU  $ and $ R/gR  $ are equivalent orders
in $Q_{gr}(  U/gU)$.
\end{enumerate}
\end{proposition}
     
\begin{proof}   $ (1) \Rightarrow (2)$ 
Choose non-zero homogeneous elements $a, b \in U$ such that $aRb \subseteq U$. Write $a = g^n a'$ where $a' \in U\smallsetminus gU$.
Then $g^na' R b \subseteq U$ and so $a' Rb \subseteq U$ since $U$ is $g$-divisible.  Replacing $a$ by $a'$, 
we can assume that $a \not \in gU$ and, similarly,  that $b \not \in gU$.  Then
$\overline{a} \, \overline{R} \, \overline{b} \subseteq \overline{U}$, with $\overline{a}, \overline{b} \neq 0$, 
as required. 

$(2) \Rightarrow (1)$ Set $\overline{U}=U/gU  \subseteq\overline{R}=R/gR.$ 
    We first note that there is a subalgebra   $\bbar{U} \subseteq S \subseteq \bbar{R}$ so that 
    $S$ is a noetherian right $\overline{U}$-module 
    and   $\overline{R}$ is a noetherian left $S$-module.  Indeed,  write $a\overline{R}b\subseteq
 \overline{U}$ for some  nonzero $a,b\in \overline{U}$ and set   $S=\overline{U}+\overline{R}b\overline{U}$.
  Clearly     $aS\subseteq \overline{U}$ and $\overline{R}b\subseteq S$.
    As in the proof of Proposition~\ref{div-noeth}, all subalgebras of $\overline{R}$ are noetherian.  
    In particular, $S$ and $ \overline{U} $ are noetherian
 and so  these inclusions ensure that   $S_{\bbar{U}}$ and ${}_{S}\bbar{R}$ are finitely generated, as claimed.

 Let $F \subseteq R$ be a finite-dimensional vector space, containing 1, such that $\bbar{F}\bbar{U} = S$.    Set $M=\wh{FU}$  and $V=\End_U(M)$.   Clearly $Q_{gr}(V)=Q_{gr}(U)=Q_{gr}(T)$. Since   $1 \in M$ and hence $M \nsubseteq gT_{(g)}$, 
 we can and will use Lemma~\ref{Sky2-1}(1) to identify 
 $V=\{q\in Q_{gr}(U) : qM\subseteq M\} \subseteq T_{(g)}$.  By  Lemma~\ref{Sky2-1}(3), $V=\wh{V}$ and 
 ${}_V M$ is finitely generated, while by Lemma~\ref{lem:finehat}, $M_U$ is finitely generated.
As $R$ is $g$-divisible  and $FU \subseteq R$, we have $M \subseteq R$. 
Since $1 \in M$ this implies that  $MR = R$. Hence   $ V R = VMR = MR = R$ and   $V \subseteq   R$.

Let $G,H \subseteq R$ be finite-dimensional vector spaces with $VG = M$ and $S \bbar{H} = \bbar{R}$.  Then 
\[ \bbar{R} \supseteq \bbar{VGH}  \supseteq \overline{FUH} = S \bbar{H} = \bbar{R}.\]
Thus $\bbar{R} = \bbar{MH} = \bbar{VGH}$ is finitely generated as a left $\bbar{V}$-module. 
Since $g\in V_+=\bigoplus_{i>0}V_i \subseteq R_+$, this implies that $R/(V_+)R$ is a finitely generated left module over 
$V/V_{+}$. By the graded analogue of Nakayama's lemma, this implies that $R$ is finitely generated as a left $V$-module. 
  Thus $R$ and $V$ are equivalent orders.   As $V$ and $U$ are equivalent orders (via the bimodule $M$), it follows that  
$U$ and $R$ are equivalent.
\end{proof}
  
%%%%%%%%%%%%%%
%%%%%%%%%%%%%%

\section{Curves}\label{CURVES}

The main result of \cite{AS} shows that any cg domain $A$ of Gelfand-Kirillov dimension two has a Veronese ring that is an idealiser inside 
a TCR. In this section we   strengthen this result for elliptic curves by proving   that, for subalgebras of a TCR over such 
a  curve corresponding to an automorphism of infinite order, the 
result holds without taking a Veronese ring, although at the cost of a finite dimensional fudge factor. 
 
Given graded modules $M,N\subseteq P$ over a cg algebra $A$, we write $M\ppe N$\label{ppe-defn} if $M$ and $N$ agree 
up to a finite dimensional 
 vector space. If $M,N\in \rgr A$,  this 
 is equivalent to $M_{\geq n}=N_{\geq n}$ for some $n\geq 0$.

\begin{theorem}\label{thm:projcurve}
    Let $A$ be a cg ring such that $Q_{gr}(A)=\kk(E)[z,z^{-1};\tau]$ for some infinite order automorphism $\tau $ of a smooth elliptic curve $E$ and 
    $z\in Q_{gr}(A)_1$.  Then there are an ideal sheaf $\sA$  and an
ample invertible sheaf $\sH$   on $E$ so that 
\[ A \ppe \bigoplus_{n \geq 0} H^0(E, \sA \sH_n).\]
\end{theorem}

\begin{remarks}   (1)  \label{idealiser-defn}
 The \emph{idealiser} $\idealiser(J)=\idealiser_U(J)$ of a right ideal $J$ in a ring $U$
 is the subring \[\idealiser(J)=\{u\in U : uJ\subseteq J\}.\]
In the notation of the theorem  $J=\bigoplus_{n \geq 0} H^0(E, \sA \sH_n) $ is a right ideal of the 
TCR $B(E, \sH,\tau)$; further, $\idealiser_U(J) \ppe \kk+J$.  So, an equivalent way of phrasing the theorem is to assert that
(up to a finite dimensional vector space) $A$ is equal to the idealiser $\idealiser(J)$ 
 inside $B(E,\sH,\tau)$.
 
 (2) The assertion that $z\in Q_{gr}(A)_1$ can be avoided at the expense of    re-grading $A$, although in the process one must replace  
 $\tau$ by some $\tau^m$ in the definition of the $\sH_n$.
 
 (3) The sheaf $\sH$ is ample if and only if it has positive degree \cite[Corollary 3.3]{Ha}, if and only if $\sH$ is {\em $\tau$-ample:} \label{tau-ample}  that is, for any coherent $\sF$ and for $n \gg 0$, $\sF \otimes \sH_n$ is globally generated with $H^1(E, \sF\otimes \sH_n) = 0$ \cite[Corollary 1.6]{AV}.
\end{remarks}

\begin{proof}    The hypothesis on $z$ ensures that $A_p\not=0\not= A_{p+1}$ for all $p\gg 0$. Fix some such $p$.

The conclusion of the theorem is, essentially, the same as that of  \cite[Theorem~5.11]{AS}, although that result has two hypotheses  we need to remove.  The first, \cite[Hypothesis~2.1]{AS}
requires that the ring in question has a non-zero element in degree one, so does at least hold for the Veronese rings $A^{(p)}$ and $A^{(q)}$,
 for $q=p+1$. The remaining hypothesis, \cite[Hypothesis~2.15]{AS}, concerns $\tau$-fixed points of $E$.  
In our situation,  this   automatically  holds as $E$ 
 has no such fixed points (see the discussion before \cite[(2.9)]{AS}).   

By the discussion above, \cite[Theorem~5.11 and Remark~5.12(2)]{AS} can be applied to the Veronese rings $A^{(p)}$ and 
$A^{(q)}$. This provides    invertible sheaves $\sA,\sB,\sF,\sG$ with $\sF, \sG$ ample such that 
$$A^{(p)} \ppe \bigoplus_{n\geq 0} H^0(E,\, \sA\otimes\sF_{p,n})
\qquad\text{and}\qquad
A^{(q)} \ppe \bigoplus_{n\geq 0} H^0(E,\, \sB\otimes\sG_{q,n}),$$
where in order to take account of the Veronese rings we have written  
$\sM_{r,n} =\sM\otimes \sM^{\tau^r}\otimes\cdots\otimes \sM^{\tau^{(n-1)r}}$ for       an invertible sheaf $\sM$.
For $n\gg 0$ the sheaves $\sA\otimes\sF_{p,nq}$ and $\sB\otimes\sG_{q,np}$
 are generated by their sections $A_{npq}$ and so  $\sA\otimes\sF_{p,nq} =\sB\otimes\sG_{q,np}$ for such $n$. 
Replacing $n$ by $n+m$ we obtain
$$\sA\otimes \sF_{p,nq}\otimes \sF_{p,mq}^{\tau^{npq}}
= \sA\otimes \sF_{p,(n+m)q} = \sB\otimes \sG_{q,(n+m)p} =\sB\otimes \sG_{q,np}\otimes \sG_{q,mp}^{\tau^{npq}} 
= \sA\otimes \sF_{p,nq} \otimes \sG_{q,mp}^{\tau^{npq}} $$
for all $n+m>n\gg 0$.
Cancelling the first two terms and   applying $\tau^{-npq}$  gives
$\sF_{p,mq}=\sG_{q,mp}$ for all $m\geq 1$. In particular it holds for $m=n$ and hence $\sA=\sB$. 

Next, set $\sH=\sG\otimes(\sF^\tau)^{-1}$; thus the equation $\sF_{p,q}=\sG_{q,p}$ gives 
\begin{equation}\label{curve1}
\sH_{q,p}=\sF_{p,q}\otimes (\sF^\tau)_{q,p}^{-1}.
\end{equation}
We claim that $\sF$ is the unique  invertible sheaf $\wt{\sF}$ satisfying 
$\sH_{q,p}=\wt{\sF}_{p,q}\otimes (\wt{\sF}^\tau)_{q,p}^{-1}$.   To see this, suppose that $\wt{\sF}$ is a second 
sheaf satisfying this property and consider
 associated divisors.  Pick a closed point $x\in E$ and write $ \mb{O}_x =\{x(i)=\tau^{-i}(x): i\in \mathbb{Z}\}$  
for the orbit of $x$ under $\tau$.   Writing $\sF = \mc{O}_E(F)$ and $\sH = \mc{O}_E(H)$ for some divisors $F$ and $H$ and  restricting to $\mb{O}=\mb{O}_x$,  gives  $F|_{\mb O} =\sum m_ix(i)$ 
and $H|_{\mb O} =\sum r_ix(i)$.   Now in terms of divisors, \eqref{curve1} gives:
\[\begin{array}{rl}
&\sum r_ix(i)+r_ix(i+q)+\cdots+r_ix(i+(p-1)q) \ = \\ \noalign{\vskip 6pt}
&\qquad \qquad =\  \sum m_ix(i)+m_ix(i+p)+\cdots+m_ix(i+(q-1)p)
\\ \noalign{\vskip 6pt}
& \qquad \qquad \qquad -  \sum m_ix(i+1)+m_ix(i+1+q)+\cdots+m_ix(i+1+(p-1)q).
\end{array}\]
Subsequently we will write $r(i)=r_i$ and $m(i)=m_i$ to avoid excessive subscripts. Equating coefficients of  $x(t)$ in the last displayed equation 
 gives  
\[\begin{array}{rl}
& m(t)+m(t-p)+\cdots + m(t-(q-1)p)
-m(t-1)-m(t-1-q)-\cdots - m(t-1-(p-1)q) = 
\\ \noalign{\vskip 6pt}
& \qquad \qquad
=r(t)+r(t-q)+\cdots + r(t-(p-1)q).
\end{array}\]
Recall that $m_i=r_i=0$ for $|i|\gg 0$. Therefore,
solving this system from $t\ll 0$ through to $t\gg 0$ gives a unique solution for the $m_i$ in terms of the $r_j$.  Finally, doing this for 
every orbit involved in the divisors $F$ and $H$ shows that $F$ is uniquely determined by $H$, and so $\sF$ is uniquely determined as claimed.

A direct calculation shows that if $\wt{\sF} = \sH_{1,p}=\sH_p$, then 
$\sH_{q,p}=\wt{\sF}_{p,q}\otimes (\wt{\sF}^\tau)_{q,p}^{-1}$.  Thus $\wt{\sF} = \sF$ and consequently 
$\sH_{1, mp} = \sF_{p,m}$ for all $m \geq 1$.    
It follows from the equation $\sH = \sG \otimes (\sF^{\tau})^{-1}$ that 
$\sG = \sH_{1,q}$, and thus  $\sH_{1, mq} = \sG_{q,m}$ for all $m \geq 1$ as well.  
To summarise, we have found sheaves $\mathcal{A}$ and $\mathcal{H}$  such that 
\begin{equation}\label{projcurve1}
A^{(s)}\ppe \bigoplus_n H^0(E,\, \mathcal{A}\mathcal{H}_{1,ns})
 =  \bigoplus_n H^0(E,\, \mathcal{A}\mathcal{H}_{ns})\quad\text{ for $s=p,p+1$. }\end{equation}
It follows that \eqref{projcurve1} holds for all $s\gg 0$, but this is not  quite  enough to prove the theorem
since, as $s$ increases, one has no control over 
the finitely many values of $n=n(s)$ for which $A^{(s)}_n\not=  H^0(E,\, \mathcal{A}\mathcal{H}_{1,ns})$.
So we take a slightly different tack.

For $0\leq r \leq p-1$, write $M(r)=\bigoplus_{n\geq 0} A_{np+r}$; thus $A=\bigoplus_{r=0}^{p-1}M(r).$
Fix some such $r$.  We can find  $0\not=x\in A_{2p-r}$, since $2p-r > p$.
Thus $xM(r)\subseteq A^{(p)}$ and so, by \cite[Proposition~5.4]{AS} and \eqref{projcurve1}, there exists  an  ideal sheaf
$\mathcal{I} \subseteq \mathcal{O}_E$ such that 
$xM(r)   %\ppe \bigoplus_{n\geq 0} H^0(E, \mathcal{I}\otimes \sH_{1,np})
\ppe \bigoplus_{n\geq 0} H^0(E, \mathcal{I}\otimes \sH_{1,np}^{\tau^{2p}}) $
(in this formula,  the   twist by $\tau^{2p}$ is for convenience only but it will simplify the  computations). 
Since $A$ is a domain, $M(r)$ is isomorphic to the shift $ xM(r)[2p-r]$. Hence, for some  integer $n_0$ independent of $r$,
\cite[Lemma~5.5]{KRS} implies that 
\begin{equation}\label{projcurve2}
M(r)_{\geq n_0} =  \bigoplus_{n\geq n_0} H^0(E, \mathcal{I}' \otimes \sH_{1,np}^{\tau^{r}})
  =   \bigoplus_{n\geq n_0} H^0\bigl(E, \, \mathcal{J}(r) \otimes \sH_{r} \otimes \sH_{1,np}^{\tau^{r}}\bigr),
\end{equation}
for some invertible sheaves $
\mathcal{I}'$ and $\mathcal{J}(r)=\mathcal{I}'\otimes(\sH_r)^{-1}$.    Possibly after increasing $n_0$, we may also assume that 
the sheaves in \eqref{projcurve1} and \eqref{projcurve2} are generated by their sections for $n\geq n_0$. 
Now pick   $n\geq n_0$ such that $r+np=(p+1)m$ for some $m$. 
Then, comparing \eqref{projcurve1} and  \eqref{projcurve2} shows that $M(r)_{r+np}$ generates the sheaves
$$\mathcal{J}(r)\otimes \sH_{r+np}
= \mathcal{J}(r)\otimes \sH_r\otimes \sH_{1,np}^{\tau^{r}} 
=  \mathcal{A}\otimes\mathcal{H}_{1,(p+1)m}  =  \mathcal{A}\otimes\mathcal{H}_{r+np}.$$
Hence $\mathcal{J}(r)=\mathcal{A}$. Since this  holds for all $0\leq r\leq p-1$, it follows that   
$A_n=H^0(E, \mathcal{A}\mathcal{H}^{\tau^n})$, for all $n\geq n_0p$.
\end{proof}

\begin{remark}  We note that   \cite[Lemma~3.2(2)]{Rog09}   states a  
result similar to Theorem~\ref{thm:projcurve}, but the proof erroneously quotes the relevant theorems from \cite{AS} 
without removing the hypothesis that rings should have a non-zero element 
in degree one.  Thus the above  proof   also corrects this oversight.  In any case,  \cite[Lemma~3.2(2)]{Rog09} was 
only used   in \cite{Rog09} for rings generated in degree one.
\end{remark}

If $A$ is a cg algebra generated in degree one, then we define a \emph{point module}\label{pt-defn}  
to be a cyclic module $M=\bigoplus_{i\geq 0} M_i$,
with $\dim M_i=1$ for all $i\geq 0$. When $A$ is not generated in degree one, a point module has this 
asymptotic structure, but the precise definition can vary depending on circumstances,  and so
we will be careful to explain which definition we mean should the distinction be important.

 To end this section we give some applications of the previous theorem to the structure of point modules, for which we need a definition.  
If $M=\bigoplus_{n}M_n$ is a graded module over a cg algebra $A$,   
 we write $s^n(M) = (M_n A)[n].$
 The largest artinian submodule of a noetherian module $M$   is written $S(M)$.  
  
 \begin{corollary}\label{cor:curvept} Let $A$ satisfy the hypotheses of Theorem~\ref{thm:projcurve}.  
    Let $M$ and $M'$ be   1-critical graded right
 $A$-modules generated in degree zero.
Then: 
\begin{enumerate}
\item the isomorphism classes of such modules are in (1-1) correspondence with the closed points of $E$;
\item $\dim M_n \leq 1$ for all $n\ge0$, with $\dim M_n = 1$ for $n \gg 0$;
 \item for $n\geq 0$,  either $M_n=0$ or $s^n M$ is cyclic and  1-critical;
\item if $ s^n M \cong s^n M' \neq 0$ for some $n \in \NN$, then $M \cong M'$.
\end{enumerate}
\end{corollary}

\begin{proof}
It is well-known that there is an equivalence of categories $\rqgr A\sim \coh (E)$, and much of 
the corollary follows   from this; thus  we first review the details of the equivalence. 
By \cite[Theorem~5.11]{AS} and the left-right analogue of Theorem~\ref{thm:projcurve}, 
we can write \[A \ \ppe \bigoplus_{n\geq 0} H^0(E, \sH_n\sA^{\tau^{n-1}}) \ \subseteq\  B\ = \ B(E,\sH,\tau)\]
for some ideal sheaf $\sA$ and invertible sheaf $\sH$.
For  $n_0\gg 0$, the ideal $J= A_{\geq n_0} = \bigoplus_{n\geq n_0} H^0(E, \sH_n\sA^{\tau^{n-1}}) $
is a left ideal of $B$.
By \cite[Proposition~2.7]{SZ}, and its proof, $\rqgr A\sim \rqgr B$ under the maps $\alpha: N\mapsto N\otimes_AB$ 
and $\beta: N'\mapsto N'\otimes_BJ$.  Moreover, by \cite[Theorem~1.3]{AV}, $\rqgr B\sim\coh(E)$. 
Under that equivalence,  for a closed point $p$ of $E$ the skyscraper sheaf $\kk(p) \in \coh(E)$ maps to the 
module $M_p'=\bigoplus_{n \geq 0} H^0(E, k(p) \otimes \sH_n) \in \rqgr B$; thus 
if $M_p=M_p'/S(M_p')$, then $M_p$ is a 1-critical $B$-module with $\dim (M_p)_n=1$ for $n\gg 0$.
By \cite[Lemma~2.6]{SZ} the same is true of the 1-critical $A$-module $N_p=\beta(M_p)/S(\beta(M_p))$.
Furthermore, the image in $\rqgr A$ of any $1$-critical graded $A$-module is a simple object, and
so  every 1-critical $A$-module   is equal in $\rqgr A$ to some $N_p$.

(2)  We will reduce to the case of a TCR generated in degree one, where the result is standard.
If the result fails, there exists a 1-critical $A$-module $M$ such that (possibly after shifting) 
 $\dim M_n\leq 1$ for all $n\geq 0$ but $\dim M_{0}>1$.  By replacing $M$ by any submodule 
  generated by a two-dimensional subspace of $M_0$ we may assume that
   $\dim(M_0)=2$. Write  $M=(A\oplus A)/F$.

Now consider   $W=\alpha(M)/S(\alpha(M))$. Since $W$ is equal in $\rqgr B$ to some $M_p$,
certainly  $\dim(W_n)\leq 1$ for all $n\gg 0$.
Moreover, the natural $A$-module map $M \to W$ must be injective since $M$ is $1$-critical, and so $\dim W_0 \geq 2$.  
As $\alpha(M)$ is a factor of $B \oplus B/FB$ it follows that $\dim W_0 = 2$.  Unfortunately,
$B$ need not be generated in degree 1. However, for $\ell \gg 0$ (indeed $\ell \geq 2$) the Veronese ring $C=B^{(\ell)}
=B(E,\sH_\ell,\tau^\ell)$ will be generated in degree one (see \cite[Lemma~3.1(2)]{Rog09}). We claim that  
$X=W^{(\ell)}$ will still be a critical $C$-module. If  not, then picking an element $0\not=x  \in X_m$ in the socle of $X$, 
we will have $x C_{\geq 1} = 0$, and so $x \in W_{m\ell}$ satisfies $x B_{i\ell} = 0$ for all $i \geq 1$.  Since 
$B_i B_j = B_{i+j}$ for all $i, j \gg 0$ \cite[Lemma~3.1(1)]{Rog09},  it follows  that $x B_m = 0$ for all $m \gg 0$, contradicting the $1$-criticality of $W$.
Thus $X$ is indeed a critical $C$-module, with  $\dim X_0=2$; say $X_0=a\kk\oplus b\kk$.

Finally,  given $X$, or any 1-critical $C$-module, then \cite{AV} again implies that
$\dim X_n=1$ for all $n\geq n_0\gg 0$. 
By \cite[Proposition~9.2]{KRS}  the map $N\mapsto N_{\geq 1}[1]$ is an automorphism on the set of 
isomorphism classes of $C$-point modules.
Applying the inverse of this map to the shift of $X_{\geq n_0}$ shows that the two point modules $aC$ and $bC$ 
must be equal to this image and hence be  isomorphic; say $bC=\phi(aC)$.  Set $n=n_0+1$.  As $\dim_{\kk}X_n = 1, $
we can write $C_{n-1}=c\kk+\text{ann}_C(a)_{n-1}$ for some $c\in C_{n-1}$.
 Since $\text{ann}_C(a)=\text{ann}_C(b)$, it follows that 
$ac=\lambda bc$ for 
some $\lambda\in\kk$. Hence $(a-\lambda b)c=0$ which implies that $(a-\lambda b)C_{n-1}=0$. 
As $C$ is generated in degree one, this forces $a-\lambda b\in S(X)$. This contradicts the criticality of $X$ and proves the result.

(3) This is immediate from Part (2).

(4) If not,  pick    1-critical modules $M\not\cong M'$ such that  there is an isomorphism
$\gamma: s^nM\cong s^nM'\not=0$ for some $n>0$.   Let $n$ be the smallest integer with this property
and then let $W\subset M$ be as large as possible a submodule  of $M$ for which $\gamma$ extends
to an isomorphism $\gamma: W\to W'\subset M'$.
Set
 \[N\ = \ \frac{M\oplus M'}{Z} \qquad \text{for}\qquad  Z=
 \left\{(a,\gamma(a)) : a\in W \right\}  .\]
 
 We claim that $N$ is 1-critical. If not, pick a homogeneous element $(u,u')\in M\oplus M'$ such that $[(u,u')+Z]$ is  a non-zero element 
 of the socle of $N$.
If $p\in \rann(u)$, then $(u,u')p\in (u,u')A_{\geq 1}\subseteq  Z.$  As $up=0$ this forces $(0,u')\in Z$ and hence 
$u'p=0$.  Similarly,  $u'p = 0$ forces $up = 0$ 
and hence $\rann(u) = \rann(u')$.  Thus there is an isomorphism $\gamma': uA \cong u'A$, which restricts to 
an isomorphism $\gamma'': uA \cap W \to  u'A \cap W'$.

We claim that any other isomorphism $\psi: uA \cap W \to u'A \cap W'$ must 
be a scalar multiple of $\gamma''$.  Put $P = uA \cap W$, which is $1$-critical.  As we have already proved, 
$\dim_\kk P_n \leq 1$ for all $n$ with equality for $n \gg 0$.  Choose $n$ such that $P_n \neq 0$ and
 fix $0 \neq x \in P_n$.  Then $xA$ is also $1$-critical and so $xA \ppe P$.  Now $\psi(x) = \lambda \gamma''(x)$ for some 
$\lambda \in \kk^{\times}$, and this forces $\psi$ to equal $\lambda \gamma''$
on all of $xA$.  Given homogeneous $y \in P$ with $y \not \in xA$, then $0 \neq yz \subseteq xA$ for some $z \in A_m$, $m \gg 0$, and so it is easy to see that this forces $\psi(y) = \lambda \gamma(y)$ also.   Thus $\psi = \lambda \gamma''$ and 
the claim follows.

Therefore, possibly after multiplying by a scalar, we can assume that 
$
\gamma'' = \gamma'|_{uA\cap W} = \gamma|_{uA\cap W}.
$
Thus, we can 
extend $\gamma$ to $W+uA$, contradicting the maximality of $W$. 
Hences $N$ is indeed critical.
Finally, as  $M\not\cong M'$ with  $\dim M_0=1=\dim M'_0$, certainly $W\subseteq M_{\geq 1}$, 
and so  $\dim N_0 = 2$, contradicting Part~(2).
 
 (1) Since the  tails $M_{\geq n_0}$ 
of  1-critical $A$-modules are in one-to-one correspondence with the points of $E$,  this  follows from  Part~(4). 
   \end{proof}

We end the section with a technical  consequence of these results for subalgebras of $T$.     

\begin{lemma}\label{lem:ann}
 Let $U$ be a noetherian cg algebra  and $M$   a finitely generated, graded 1-critical right $U$-module.  
 Then $ \rann_U(M)$ is prime and $\rann_U(M)=\rann_U(N)$ for every nonzero submodule $N\subseteq M$.
\end{lemma}
\begin{proof}
This is a standard application of ideal invariance; use, for example, 
 \cite[Corollary~8.3.16]{MR} and the proof of (iii) $\Rightarrow$ (iv) of 
\cite[Theorem~6.8.26]{MR}.
\end{proof}

\begin{corollary}\label{prop:shiftpt}
Assume that  $T$   satisfies Assumption~\ref{ass:T} and  let $U$ be  a $g$-divisible  subalgebra of $T_{(g)}$ with  $Q_{gr}(U)=Q_{gr}(T)$.  Suppose that $M$ and $N$ are $1$-critical right $U$-modules which are cyclic, generated in degree $0$, with  $g\in \rann_U(M)$.  
  For some $n\geq 0$ with $M_n\not=0$ suppose that  there exists $m\geq 0$ such that  
$(\rann_U M_n)_{\geq m} = (\rann_U N_n)_{\geq m}$. Then $M \cong N$.
\end{corollary}

\begin{proof} By hypothesis,  $g^{m}\in (\rann_U M_n)_{\geq m} = (\rann_U N_n)_{\geq m}$.   Then 
$N_0 g^{n+m} \subseteq N_n g^m = 0$.  As $g$ is central and $N$ is generated by $N_0$, it follows that 
$g^{n+m}\in \rann_U N$ and hence $g\in \rann_U N$ by Lemma~\ref{lem:ann}. 
Thus both $M$ and $N$ are  modules over $A = \bbar{U}$  and to prove the lemma it suffices to consider modules over that ring.  By Lemma~\ref{lem:expectQ},
 $Q_{gr}(A) = Q_{gr}(\overline{T}) = k(E)[t, t^{-1}; \tau]$ and so 
$A$ satisfies the hypotheses of Theorem~\ref{thm:projcurve}.
 
  Clearly,  $N_n\not=0$. By 
    Corollary~\ref{cor:curvept}(2)    $\dim M_n=1=\dim N_n$
  and so  $ M_nA[n]\cong A/I$ and $N_nA[n]\cong A/J$ for some graded right ideals $I,J$.  
  By hypothesis,  $I_{\geq m}=J_{\geq m}$. However,  
 as $ I/I_{\geq m}$ is finite dimensional and $A/I$ is $1$-critical, $I/I_{\geq m}$
  is the unique largest finite dimensional submodule of $A/I_{\geq m}=A/J_{\geq m}$. Hence $I=J$ and $M_nA\cong N_nA$. 
 By Corollary~\ref{cor:curvept}(4),  $M\cong N$.
\end{proof}

 %%%%%%%%%%%%%%%%%%%
 %%%%%%%%%%%%%%%%%%%%
\section{Right ideals of \texorpdfstring{$T$}{LG} and the rings \texorpdfstring{$T(\divd)$}{LG}}
\label{RIGHT IDEALS}

Throughout this section, let $T$ satisfy Assumption~\ref{ass:T}.  
Our first aim in this section is to describe certain graded right ideals $J$ of $T$ such that $T/J$ 
is filtered by shifted point modules.   In fact, the main method we use in the next section to understand a general subalgebra $U$ of $T$ is to compare its graded pieces with the graded pieces of these right ideals $J$ and their left-sided analogs.   The easiest way to construct the required right ideals $J$ is to use some machinery 
from \cite{VdB}.   The details will appear in a companion paper to this one \cite{RSS2}.

\begin{definitions}\label{saturated-defn} 
Given a right ideal $I$ of a cg   algebra $R$, the \emph{saturation} $I^{sat}$ of $I$ is the 
 sum of the right ideals $L\supseteq I$ with $\dim_{\kk} L/I<\infty.$ 
 If $I=I^{sat}$,  we say that $I$ is \emph{saturated}.
 
   Recall that $T/gT \cong B = B(E, \mc{M}, \tau)$, where $\deg \mc{M} = \mu$.  
 For divisors $\mathbf{b}, \mathbf{c}$ on $E$, we write $\mathbf{b} \geq \mathbf{c}$ if $\mathbf{b} - \mathbf{c}$ is effective.   
 A list of divisors $(\divd^0, \divd^1, \dots, \divd^{k-1})$ on $E$ is an \emph{allowable divisor layering}  \label{layering-defn}
  if   $\tau^{-1}(\divd^{i-1}) \geq \divd^i$ for all $1 \leq i \leq k-1$.   By convention, we define $\divd^i = 0$ for all $i \geq k$.
Given an allowable divisor layering $\divd^{\bullet}=(\divd^0, \divd^1, \dots, \divd^{k-1})$ on $E$, let $J(\divd^{\bullet})$ 
 be the saturated right ideal of $T$ defined in \cite[Definition~3.4]{RSS2}.   
\end{definitions}

We omit the precise  definition of $J(\divd^{\bullet})$ because it is technical, and not essential in this paper.  
Instead, what matters are the following  properties of this right ideal, which help explain the name ``divisor layering.''  
For any graded right $T$-module $M$, we think of the $B$-module  $Mg^j/Mg^{j+1}$ as the $j$th layer of $M$.  
Recall that we write $\pi(N)$ for the image of a finitely generated graded $B$-module $N$ in the quotient category $\rqgr B$. 
Recall also from the proof of Corollary~\ref{cor:curvept}(1)  that there is an equivalence of categories $\coh E \simeq \rqgr B$   given by 
$\mc{F} \mapsto \pi(\bigoplus_{n \geq 0} H^0(E, \mc{F} \otimes \mc{M}_n))$.  

\begin{lemma} {\rm (\cite[Lemma~3.5]{RSS2})}
\label{lem:layer}
Let $\divd^{\bullet}$ be an allowable divisor layering and let $J = J(\divd^{\bullet})$ and $M = T/J$.
\begin{enumerate}
\item If $M^j = Mg^j/Mg^{j+1}$, then as objects in $\rqgr B$ we have 
\[
\pi(M^j) \cong \pi \bigg(\bigoplus_{n \geq 0} H^0(E, (\mc{O}_E/\mc{O}_E(-\divd^j)) \otimes \mc{M}_n) \bigg).
\]
In particular, the divisor $\divd^j$ determines  
%which simple objects in $\rqgr B \simeq \coh E$, with multiplicity, 
%occur in a composition series of this finite length object, or equivalently 
the point modules that occur in a filtration of $M^j$ by (tails of) point modules.
\item $(\overline{J})^{\sat} = \bigoplus_{n \geq 0} H^0(E, \mc{M}_n(-\divd^0)).$ 
\item If $\divd^{\bullet} = (\divd)$ has length $1$, then 
$J(\divd) = \bigoplus_{n \geq 0} \{ x \in T_n \st \bbar{x} \in H^0(E, \sM_n(-\divd)) \}$.   \qed
\end{enumerate}
\end{lemma}
\noindent Note that  as a special case of part (3) of the lemma, if $p \in E$ and $\divd = p$ 
then $J(p)$ is simply the right ideal of $T$ such that $P(p) = T/J(p)$ is \emph{the point module} corresponding to the point $p$. \label{pt-defn2} (We note that this will coincide with the earlier definition of a point module, should $T$ be generated in degree one.)

We will require primarily the following two special cases of the construction above.  Starting now, it will be sometimes convenient 
to employ the   notation:  
\begin{equation}\label{4.4}
p_i = \tau^{-i}(p), \ \text{for any}\ p \in E.
\end{equation}
\begin{definition}
\label{def:Q}
Given any $p \in E$, $i \geq 1$, and 
$0 \leq r \leq d \leq \mu$, we define $Q(i,d,r,p) = J(\divd^{\bullet})$, where 
\[
\divd^0 = dp + dp_1 + \dots + dp_{i-1}, \ \divd^1 = dp_1 + \dots + 
dp_{i-1}, \ \dots, \ \divd^{i-2} = dp_{i-2} + dp_{i-1}, \ \divd^{i-1} = r p_{i-1}.
\] 
\end{definition}

\noindent 
Intuitively, the divisor layers for $Q$ are in the form of a triangle, but the vanishing in the last layer is allowed to be of
 lower multiplicity than in the others.  The other special case we need is a similar triangle shape which allows for
  the involvement of points from different orbits. 
  
\begin{definition}
\label{def:M}
For any divisor $\divd$ and $k \geq 1$, we define $M(k, \divd) = J(\divc^{\bullet})$ where  
\[
\divc^0 = \divd + \tau^{-1}(\divd) + \dots + \tau^{-k+1}(\divd), \ \divc^1 = \tau^{-1}(\divd) + \dots + \tau^{-k+1}(\divd), \ \dots,  \divc^{k-1} = \tau^{-k+1}(\divd).
\]
It is useful to also define $M(k, \divd) = T$ by convention, for any $k \leq 0$.
\end{definition}

Note that  
$M(k, dp) = Q(k, d, d, p)$ for any $k,d \geq 0$.    The right ideals $M(k,\divd)$ are also useful for defining important subalgebras of $T$.

\begin{definition}
\label{def:star}
For any divisor $\divd$ with $\deg \divd < \mu$ we set 
\[
T(\divd) : = \bigoplus_{n \geq 0} M(n, \divd)_n, 
\]
which by \cite[Theorem~5.3(2)]{RSS2}  is a $g$-divisible subalgebra of $T$.  
More generally, for any $\ell \geq 0$ we define
\[
T_{\leq \ell} * T(\divd) :=\bigoplus_{n \geq 0} M(n - \ell, \divd^{\tau^\ell})_n, 
\]
which by \cite[Proposition~5.2(2)]{RSS2}  is a right $g$-divisible $T(\divd)$-module.
\end{definition}
\noindent  
 
When $\deg \divd \leq \mu - 2$, but not in general, 
 the module $T_{\leq \ell} * T(\divd)$ is equal to the right $T(\divd)$-module 
$T_{\leq \ell} T(\divd) \subseteq T$ \cite[Theorem~5.3(6)]{RSS2}, so the notation is chosen to suggest multiplication.  
As is discussed in \cite{Rog09} and  \cite[Section~5]{RSS2},  the ring $T(\divd)$ should be thought of as corresponding 
geometrically to a blowup of $T$ at the divisor~$\divd$.

There are left-sided versions of all of the above definitions and results, because Assumption~\ref{ass:T} is left-right 
symmetric.  We quickly state these analogues, because there are some non-obvious differences in the statements, which result from the fact that the equivalence of categories $\coh E \simeq B \lqgr$ has the slightly different form
$\mc{F} \to \pi( \bigoplus_{n \geq 0} H^0(E, \mc{M}_n \otimes \mc{F}^{\tau^{n-1}}))$.
Generally, $\tau^{-1}$ appears  in the left sided results wherever $\tau$ appears in the right sided version.  
A list of divisors 
$\divd^{\bullet} = (\divd^0, \divd^1, \dots, \divd^{k-1})$ on $E$ is a \emph{left allowable divisor layering}
\label{left-layering-defn} if 
$\tau(\divd^{i-1}) \geq \divd^i$ for all $1 \leq i \leq k-1$.  We indicate left-sided versions by a prime in the notation. 
 In particular, given a left allowable divisor layering, there 
is a corresponding saturated left ideal $J'(\divd^{\bullet})$ of $T$, defined in \cite[Section~6]{RSS2}, 
which satisfies the following analogue of Lemma~\ref{lem:layer}.

\begin{lemma} {\rm (\cite[Lemma~6.1]{RSS2})}
\label{lem:left layer}
Let $\divd^{\bullet}$ be a left allowable divisor layering and let $J' = J'(\divd^{\bullet})$ and $M = T/J'$.
\begin{enumerate}
\item If $M^j = Mg^j/Mg^{j+1}$ is the $j$th layer of $M$, then in $B \lqgr$ we have 
\[
\pi(M^j) \cong  \pi \bigg(\bigoplus_{n \geq 0} H^0(E, \mc{M}_n \otimes (\mc{O}_E/\mc{O}_E(-\tau^{-n+1}(\divd^j)))) \bigg).
\]
\item $(\overline{J'})^{\sat} = \bigoplus_{n \geq 0} H^0(E, \mc{M}_n(-\tau^{-n+1}(\divd^0))).$
\item If $\divd^{\bullet} = (\divd)$ has length $1$, then $J'(\divd) = 
\bigoplus_{n \geq 0} \{ x \in T_n \st \bbar{x} \in  H^0(E, \sM_n(-\tau^{-n+1}(\divd))) \}$.  \qed
\end{enumerate}
\end{lemma}
\noindent
Similarly as on the right, as a special case of Part (3) we have that $P'(p) = T/J'(p)$ is the left point 
module of $T$ corresponding to $p$.
\label{pt-defn3}

Of course, we also have left sided analogs of Definitions~\ref{def:Q} and \ref{def:M}, but we only need the former.  
Namely, given any 
$p \in E$, $i \geq 1$, and 
$0 \leq r \leq d \leq \mu$, we define $Q'(i,d,r,p) = J'(\divd^{\bullet})$, where 
\[
\divd^0 = dp + dp_{-1} + \dots + dp_{-i+1}, \ \divd^1 = dp_{-1} + \dots + dp_{-i+1}, \ \dots, \ \divd^{i-2} 
= dp_{-i+2} + dp_{-i+1}, \ \divd^{i-1} = r p_{-i+1}.  
\]

The right ideals $Q$, and their left-sided analogs $Q'$, will be used below to define filtrations in which every factor 
is a shifted point module; we will then study how an arbitrary subalgebra $U$ of $T$ intersects such filtrations.   
The relevant result for this is as follows.  

\begin{lemma}  {\rm (\cite[Lemma~6.5]{RSS2})}
\label{lem:Qfactor}
Let $i, r, d, n \in \mb{N}$, with $i < n$ and $1 \leq r \leq d \leq \mu$, and $p \in E$.  Then 
\begin{enumerate}
\item $Q(i, r, d, p) \subseteq Q(i,r-1, d, p)$, with factor
\[
[Q(i,r-1,d,p)/Q(i,r, d,p)]_{\geq n} \cong P(p_{i-n-1})[-n].
\]
\item $Q'(i,r,d,p) \subseteq Q'(i, r-1, d, p)$, with factor
\[  
[Q'(i,r-1,d,p)/Q'(i,r, d,p)]_{\geq n} \cong P'(p_{-i+n+1})[-n].    \eqno{\qed}
\]
\end{enumerate}
\end{lemma}
The left and right ideals defined above are actually closely related.  In fact, by \cite[Proposition~6.8]{RSS2} one always has
$Q(i,r,d,p)_n = Q'(i,r,d,p_{i-n})_n$, as we  will exploit in the next section.

We conclude this section with a review of some important homological concepts.

 \begin{definition}\label{AG-defn}
A ring $A$ is called  \emph{Auslander-Gorenstein} if it has finite injective dimension and satisfies the
 {\it Gorenstein condition}: If $p<q$ 
are non-negative integers and  $M$ is a finitely generated $A$-module, then
$\Ext_A^p(N,\,A)=0$ for every submodule $N$ of 
$\Ext_A^q(M,\,A)$.
Set
 $j(M)= \min\{ r : \Ext^r_A(M,A)\not= 0\}$ for the \emph{homological grade} of $M$.  
 Then, an Auslander-Gorenstein ring $A$ of finite Gelfand-Kirillov dimension is
 called \emph{Cohen-Macaulay} (or CM), provided 
 that  $j(M)+\GKdim(M)=\GKdim(A)$ holds
for every finitely generated
 $A$-module $M$.   A cg $\kk$-algebra $A$ is called \emph{Artin-Schelter (AS) Gorenstein} if 
$A$ has injective dimension $d$, and $\dim_{\kk} \Ext_A^j(\kk, A)  = \delta_{j,d}$ for all $j \geq 0$.
An AS Gorenstein algebra is called \emph{AS regular} if it is also has finite global dimension $d$.
\label{AS-defn} 
\label{ASG-defn}
\end{definition}
 
 As the next two results show, many of the algebras appearing in this paper do satisfy these conditions, and this automatically leads to some nice consequences.

 \begin{proposition}\label{8-special22}
 Let $R=T(\divd) \subseteq T$ for some effective divisor $\divd$ with $\deg\divd\leq \mu-1$, in the notation of 
Assumption~\ref{ass:T}.
 Then the following hold.
 \begin{enumerate}
 \item $R/gR=B(E,\mathcal{M}(-\divd),\tau)$.
\item If $\deg \divd < \mu - 1$, then $R$ is generated as an algebra in degree $1$, while 
if $\deg \divd = \mu - 1$ then $R$ is generated as an algebra in degrees $1$ and $2$.
 \item Both $R$ and $R/gR$ are  Auslander-Gorenstein and CM.
\item $R$ is   a maximal order in $Q_{gr}(R)=Q_{gr}(T)$. 
\end{enumerate}
 \end{proposition}
  
 \begin{proof} 
 Combine \cite[Theorem~5.3]{RSS2} and \cite[Theorem~6.6]{Le}.
   \end{proof}

 %Part (1) appears in \cite[??]{RSS2}, and Part (2) in \cite[??]{RSS2}.
 %Parts~(3) and (4) then follow by using 
%  the proofs of \cite[Theorem~6.3(3)]{Rog09}  and  \cite[Theorem~6.7]{Rog09}.
 % (The cited results from \cite{Rog09} assume that $\deg \mathcal{M}(-\divd)\geq 2$, but this is not needed for the proof.)

\begin{lemma}\label{thou11}  Fix a cg noetherian domain $A$ that is Auslander-Gorenstein and CM. 
Set  $\GKdim(A)=\alpha$. 
\begin{enumerate}

\item If  $N$ is a finitely generated graded right (or left)  $A$-submodule  of $Q=Q_{gr}(A)$ then  $N^{**}$ is the unique largest  submodule $M\subseteq Q$ with $\GKdim(M/N)\leq \alpha-2$.

\item  In particular there is no graded $A$-module $A\subsetneqq N\subset Q$ with $\GKdim(N/A)\leq \GKdim(A)-2$. 

\item If $J=J^{**}\not=A$ is a proper reflexive right ideal of $A$ then  $A/J$  is 
\emph{$(\alpha-1)$-pure} \label{pure}
in the sense that $\GKdim(I/J ) = \GKdim(A/J)=\alpha-1$ for every nonzero $A$-module $I/J\subseteq A/J$.

\item If $N$ is a finitely generated  $A$-module, then 
  $\Ext^{j(N)}_A(N, A)$ is a pure module  of
Gelfand-Kirillov dimension equal to $\GKdim(N)$. 
\end{enumerate}
\end{lemma}

\begin{proof} Part (1)  follows from \cite[Theorem~3.6 and Example~3.2]{BE}.
Parts~(2) and (3) are   special cases of (1) while Part~(4) follows from  by \cite[Lemma~2.8]{BE}. 
\end{proof}

   %%%%%%%%%%%%%%%%%%%%%%%%%%%%%%%%%%%%%%%%%%%%
\section{An equivalent   \texorpdfstring{$T(\divd)$}{LG}  }\label{equiv2-sec}

Throughout this section, $T$ will be an algebra satisfying Assumption~\ref{ass:T}, and we maintain all of the notation
 introduced in Section~\ref{RIGHT IDEALS}.  In this section we prove that if $U$ is a $g$-divisible graded subalgebra 
 of $T$ with $Q_{\gr}(U)=Q_{\gr}(T)$, then $U$ is an equivalent order to some $T(\divd)$.   This should be compared with 
 \cite[Theorem~1.2]{Rog09}:  the rings $T(\divd)$ with $\divd$ effective of degree $< \mu -1$ are precisely the maximal 
 orders $U\subseteq T$ with $Q_{gr}(U)=Q_{gr}(T)$ that are  generated in degree 1.  

We begin by studying $\bbar{U}$ and related subalgebras of $\kk(E)[t, t^{-1}; \tau]$.  
We say that two divisors $\divx$ and $\divy$ are {\em $\tau$-equivalent}\label{tau-equiv-defn} if, for every orbit $\mathbb{O}$ 
of $\tau$ on $E$, one has $\deg(\divx|_{\mathbb{O}} ) = \deg(\divy|_{\mathbb{O}} ).$
Two invertible sheaves $\sO_E(\divx)$ and $\sO_E(\divy)$ are then {\em $\tau$-equivalent} 
  if the divisors $\divx$ and $\divy$ are $\tau$-equivalent.

\begin{lemma}\label{equiv-TCR}
 Let $\sN, \sN'$ be ample invertible sheaves on $E$ of the same degree.  
Let $R := B(E, \sN, \tau)$ and 
 $R': = B(E, \sN', \tau)$, and let $F$ be an $(R', R)$-sub-bimodule of $\kk(E)[t, t^{-1}; \tau]$.  Then   $F_R$ is finitely 
 generated if and only if ${}_{R'}F$ is finitely generated. In this case,  $\sN$ and $\sN'$ are $\tau$-equivalent.  
\end{lemma}

\begin{proof} 
Suppose that $F$ is a finitely generated right $R$-module.  
By \cite[Theorem~1.3]{AV} there is an invertible sheaf $\sF$ on $E$ so that
 $F \ppe \bigoplus_n H^0(E, \sF \sN_n)$.  For $m,n\gg 0  $ ampleness ensures that the 
 sheaves $\sN'_m$ and $\sF\sN_n$ are generated by their sections and, by construction, those 
 sections are $R_m'=H^0(E,\sN_m')$, respectively $F_n=H^0(E, \sF\sN_n)$, again for $m,n\gg 0$.
Since $F$ is a left $R'$-submodule,  $R'_m F_n \subseteq F_{n+m}$ for all $m, n$ and so these observations imply that 
\[ \sN'_m \sF^{\tau^m} \sN_n^{\tau^m} \subseteq \sF \sN_{n+m}\]
for all $n, m \gg 0$.   By hypothesis, $ \sN'_m \sF^{\tau^m} \sN_n^{\tau^m} $ and $ \sF \sN_{n+m}$
  have the same degree and so they are therefore   equal.  In addition, for $n, m \gg 0$ the sheaves $\sN_m'$ 
  and $\sF^{\tau^m} \sN_n^{\tau^m}$ have degree $\geq 3$. Thus, by   \cite[Lemma~3.1]{Rog09}, the map 
  \[H^0(E, \sN_m') \otimes H^0(E, \sF^{\tau^m} \sN_n^{\tau^m}) \to H^0(E, \sN_m' \sF^{\tau^m} \sN_n^{\tau^m})\]
   is surjective.  
Thus,  $R'_m F_n = F_{n+m}$ for all $n, m \gg 0$ and  ${}_{R'} F $ is finitely generated.  By symmetry, if  ${}_{R'} F $ is
 finitely generated
then so is $F_R$.

In either case, it follows   that $ \sN'_m = \sF \sN_m (\sF^{-1})^{\tau^m} $ for all $m \gg 0$.  The identity 
$\sN' (\sN'_m)^{\tau} = \sN'_{m+1}$ gives 
\[\sN'\sF^{\tau}\sN_m^{\tau}(\sF^{-1})^{\tau^{m+1}} = \sN'(\sN'_{m})^\tau =\sN'_{m+1}
= \sF\sN_{m+1} (\sF^{-1})^{\tau^{m+1}} = \sF\sN\sN_m^\tau(\sF^{-1})^{\tau^{m+1}}.\] 
Rearranging gives  $ \sN' = \sF \sN (\sF^{-1})^\tau$, 
which is certainly $\tau$-equivalent to $\sN$.
\end{proof}

We next need two  technical results on subalgebras of $B(E,\sM,\tau)$ that modify the data given by Theorem~\ref{thm:projcurve}.  

\begin{notation}\label{tau-notation}
Recall from \eqref{4.4}  that given a closed point $p\in E$ we write $p_0=p$ and $p_n=\tau^{-n}(p)$ for all $n\in \mathbb{Z}$.
We will also write $\divx^\tau=\tau^{-1}(\divx)$ when $\divx$ is a   divisor (or closed point)  on $E$, to distinguish left 
and right actions and set $\divx_n=\divx+\divx^\tau+\cdots +\divx^{\tau^{n-1}}$.  \end{notation}

We start with a routine consequence of Theorem~\ref{thm:projcurve}.

\begin{corollary}\label{cor:geomdata}
 Let $A \subseteq B = B(E, \sM, \tau)$ be a cg algebra with $Q_{\gr}(A) = Q_{\gr}(B)$.
Then there exist $\divx,\divy \in \Div(E)$ and $k \in \ZZ_{\geq 1}$ so that
\beq\label{geomdata}
A_n = H^0(E, \sM_n(-\divy-\divx_n))   \qquad \text{ for all $n \geq k$}.
\eeq
Furthermore,  $\mu>\deg \divx\geq 0$,  and 
\beq\label{normal-eq}
\text{for any $n \geq k$ and divisor $\divc > \divy + \divx_n$ we have $A_n \not \subseteq H^0(E, \sM_n(-\divc))$.}
\eeq
\end{corollary}

\begin{proof}
 By Theorem~\ref{thm:projcurve}, there exist an integer $k\geq 1$,  an ideal sheaf $\sY$ and  an ample invertible sheaf $\sN$  
  on $E$ so that 
\begin{equation}\label{geomdata-dull}
 A_n =  H^0(E, \sY \sN_n)\qquad \text{for all } n\geq k.
 \end{equation}
 Let $\sY=\mathcal{O}_E(-\divy)$ for some divisor $\divy$ and write   
$    \sN= \sM(-\divx) $ for the appropriate divisor $\divx$ on $E$; thus  \eqref{geomdata} is just a restatement 
of \eqref{geomdata-dull}.
Further,  $\deg \divx = \mu - \deg \sN$ which,   as $\sN$ is ample, implies that $\deg \divx < \mu$. 
 On the other hand,
 Riemann-Roch implies that  
$$\mu n=\deg \sM_n =\dim B_n \geq \dim A_n=n\deg(\sN)-\deg \divy= n(\mu-\deg \divx)-\deg \divy\quad\text{ for $n\gg 0$.}$$ 
Therefore,  $\deg \divx \geq 0$.

Finally,   as $\sN$ is   $\tau$-ample, after possibly increasing $k$ we can assume that 
   $ \sM_n(-\divy-\divx_n)$ is generated by its sections $A_n$ for all  $n \geq k$ 
   (see for example  \cite[Lemma~4.2(1)]{AS}). Thus for any 
larger divisor $\divc> \divy + \divx_n $ we will have 
$H^0(E, \sM_n(-\divc)) \subsetneqq H^0(E, \sM_n(-\divy-\divx_n))$.
Thus \eqref{normal-eq}  also holds.
\end{proof}

We next want to modify Corollary~\ref{cor:geomdata} so that   $\divx$ is replaced by an effective divisor, although this 
will result in a     weaker version of \eqref{geomdata}.

\begin{proposition}\label{prop:find D}
 Let $A\subseteq B = B(E, \sM, \tau)$ be a cg algebra  with $Q_{\gr}(A)=Q_{gr}(B)$.  
Then there is an effective divisor $\divd$ on $E$, supported at points with distinct orbits and with $\deg \divd < \mu$, 
 so that $A$ and $C=B(E, \sM(-\divd), \tau)$ are equivalent orders. Moreover, $\divd$ and $k\in\mathbb{Z}_{\geq 1}$ can be chosen so that 
\beq\label{findD} 
\begin{array}{rl}
A_n\ \subseteq \ & H^0\bigl(E,\, \sM_n(-\divd^{\tau^k} - \dots - \divd^{\tau^{n-1}})\bigr) \\
\noalign{\vskip 6pt}
&  = H^0\bigl(E, \,\sO_E(\divd_k)\otimes \sM(-\divd)_n \bigr)
\qquad\text{for all $n \geq k$.}
\end{array}\eeq
\end{proposition}

\begin{proof} Let $\divx$ and $\divy$ be the divisors constructed in  the proof of Corollary~\ref{cor:geomdata}, and let $k $ be the integer from that result.
 Fix an orbit $\mathbb{O}$ of $\tau$ on $E$. By possibly enlarging $k$ we can pick $p\in\mathbb{O}$  so 
 that, using the notation of \eqref{tau-notation},
\beq\label{supp} \text{$\divx|_{\mb O} $ is supported on 
 $\{p=p_0, \ldots, p_k\}$, and $\divy|_{\mb O}$ is supported on $\{p_0, \dots, p_{k-1}\}$.}
\eeq
  Thus  
\[\divy|_{\mathbb{O}} =\sum_{i=0}^{k-1} y_ip_i \qquad\text{and}\qquad    
\divx|_{\mathbb{O}}  =\sum_{i=0}^k x_ip_i,\]
for some integers $y_i$ and $x_j$.
 For $n \in \NN$  we have
$(\divx_n)|_{\mathbb{O}}  = \sum_{i=0}^k x_i(p_i+p_{i+1}+\cdots +p_{i+n-1})$. 
 Thus, for $n \geq k$ we calculate that 
\begin{multline}\label{combinatorics}
 (\divy + \divx_n)|_{\mathbb{O}}   = 
 (y_0+ x_0) p_0 + \cdots + (y_j + \sum_{i \leq j} x_i)p_j + \cdots + (y_{k-1} + \sum_{i \leq k-1} x_i) p_{k-1} \\
+ (\sum_{i \geq 0} x_i) (p_{k} + \cdots + p_{n-1})
 + (\sum_{i \geq 1} x_i) p_n + \cdots + (\sum_{i \geq j} x_i) p_{n+j-1} + \cdots + x_k p_{n+k-1}.
\end{multline} 

Let $e_p= \sum x_i$.  Since $A \subseteq B(E, \sM, \tau)$,  the divisor $\divy+\divx_n$ is effective for $n \gg 0$, and so 
\beq\label{positivity} 
 y_j + \sum_{i \leq j} x_i \geq 0 \quad \quad \mbox{ and } \quad \quad \sum_{i \geq j} x_i \geq 0
  \eeq
for all $0 \leq j \leq k$.  In particular, $e_p \geq 0$.  
Let $\divd = \sum_p e_p p$, where the sum is taken over one closed point $p$ in each orbit $\mathbb{O}$ of $\tau$.  Take the maximum of the values of $k$ occurring for the different orbits in the support of $\divx $ and $\divy$, and call this also $k$.
From \eqref{combinatorics} and \eqref{positivity} we  see  that, on each orbit $\mathbb{O}$ and hence in general,  
\[\divy+\divx_n  \geq \divd^{\tau^{k}}+\cdots+\divd^{\tau^{n-1}}\]  
for all $n \geq k$.  In other words, \eqref{findD} holds for this $\divd$ and  $k$.  
By construction, $\deg \divd = \deg \divx < \mu$.

Finally, let $\sN = \sM(-\divx) $ and let $\sY = \sO_E(-\divy)$.  Let $C=B(E,\sM(-\divd),\tau)$ and $C' = B(E, \sN, \tau)$.  
Equation~\eqref{findD} can be rephrased as saying that   
\[ \sY \sN_n \subseteq \sM_{n}(-\divd^{\tau^k}-\cdots-\divd^{\tau^{n-1}})
= \sO(\divd_k)\otimes \sM(-\divd)_n\qquad
\text{for all $n \geq k$.}\]
Thus, for $n_0\gg 0$,
\[ C'_{\geq n_0} \subseteq N = \bigoplus_{n\geq n_0} 
H^0\bigl(E, \bigl(\sY^{-1}\otimes\sO(\divd_k)\bigr)\otimes  \sM(-\divd)_n\bigr).\]
Since $\sM(-\divd)$ is $\tau$-ample (because it has positive degree) and 
$\sY^{-1}\otimes\sO(\divd_k)$ is coherent, \cite[Lemma~4.2(ii)]{AS} implies
that $N$ is a finitely generated right $C$-module. Hence,   so is $ C'C$.  
Since   $\sN = \sM(-\divx)$ with $\deg \divd = \deg \divx$, we can apply Lemma~\ref{equiv-TCR}
to conclude that  $C'C$ is a
finitely generated  left $C'$-module. Thus $C$ and $C'$ are equivalent orders. 
 By the proof of  \cite[Theorem~5.9(2)]{AS} $C'$ is a finitely generated right $A$-module.
 Thus  $C'$ and $A$ 
are equivalent orders, and so $C$ and $A$ are also equivalent.
\end{proof}
 
 \begin{definition}\label{normalised-notation}
We say that $(\divy, \divx, k)$ as given by Corollary~\ref{cor:geomdata} is {\em geometric data for $A$}.
If \eqref{supp} holds for   $p \in \mb O$, we say $p$ is a {\em normalised orbit representative} for this data, and we say that $\divd =\sum e_p p$ is a {\em normalised divisor} for $(\divy,\divx,k)$.
To avoid trivialities, the only orbits considered here are the (finite number of)  orbits containing the support of $\divx$ and $\divy$.
By construction, $\deg \divd = \deg \divx < \mu$.
\end{definition}

We now use these results to study subalgebras of $T$, and begin with a general idea of the strategy.   Let $U$ be a $g$-divisible graded subalgebra of $T$ so that
$Q_{\gr}(U) = Q_{\gr}(T)$.  By Proposition~\ref{div-noeth},  $U$ is automatically a finitely generated,  noetherian $\kk$-algebra,
so the earlier results of the paper are available to us.
 Let $(\divy, \divx, k)$ be geometric data for $\overline{U}$ and let $\divd$ be a normalised divisor for $(\divy,\divx, k)$.
We will show that $U$ and $T(\divd)$ are equivalent orders.
 
Recall the right $T(\divd)$-module $T_{\leq k} * T(\divd) = \bigoplus_{n \geq 0} M(n-k, \divd^{\tau^k})_n$ 
from Definition~\ref{def:star}.  This is a $g$-divisible right $T(\divd)$-module, with 
$\overline{T_{\leq k} * T(\divd)} = \bigoplus_{n \geq 0} H^0\bigl(E, \mc{M}_n(-\divd^{\tau^k} - 
 \dots - \divd^{\tau^{n-1}})\bigr)$, by Lemma~\ref{lem:layer}.
In other words, by \eqref{findD}, 
\[ \overline{U} \subseteq \overline{T_{\leq k} * T(\divd)}.\]
Our next goal is to show that this holds without working modulo $g$:  that is, that 
\beq\label{wannaget} U \subseteq T_{\leq k}*T(\divd). \eeq 
This will force $\wh{U T(\divd)}$ to be finitely generated as a right $T(\divd)$-module, which is a key step towards proving that $U$ and $T(\divd)$ are equivalent orders.
%%% CUT 1

Suppose therefore that \eqref{wannaget} fails, and so  $U_{n_0} \not \subseteq \bigl( T_{\leq k}*T(\divd) \bigr)_{n_0}$ for some $n_0$.  Necessarily, $n_0 > k$.
We will find a right $T$-ideal $Q(i, r, e, p)$ and a left 
$T$-ideal $Q'(i, r, e, q)$ such that 
setting $I = U \cap Q(i, r, e, p)$ and $J = U \cap Q'(i, r, e, q)$,
then $U/I$ and $U/J$ are isomorphic to point modules in large degree.
Further, we can choose $p$ and $q$ so that $I_{n_0} = J_{n_0}$.  However,
Corollary~\ref{prop:shiftpt} can be used 
  to derive precise formul\ae\ for $I_{n_0}$ and $J_{n_0}$, and we will see that these formul\ae\
 are inconsistent, leading to a final contradiction.

In the next few results, we carry out this argument, using induction and a filtration of $T$ by the right ideals $Q$ defined in Section~\ref{RIGHT IDEALS}. Recall the definition of $I^{sat}$ from Definition~\ref{saturated-defn}, and the definitions of $J(\divd^{\bullet})$, $Q(i,r,d,p)$, and their left-sided analogues from Section~\ref{RIGHT IDEALS}.  As already noted in that section
\beq\label{star} \mbox{ by  Lemma~\ref{lem:layer}(3), $J(p)_n$ is codimension 1 in $T_n$ for all $n \in \NN$ and $p \in E$. } \eeq
 \begin{lemma}\label{step 1}
 Let $U$ be a   $g$-divisible graded subalgebra of $T$ with $Q_{\gr}(U) = Q_{\gr}(T)$. 
Suppose that $n > i \geq 1$, $1 \leq r \leq e \leq \mu$, and $j \in \mb{Z}$.
Suppose further that
\begin{enumerate}
\item[(A)] $U_{\geq n} \subseteq Q(i, r-1,e,p_j)$, but $U_n \nsubseteq Q(i, r, e, p_j)$; and 
\item[(B)] $\overline{U}_m \nsubseteq \overline{J(p_{i+j-n-1})_m} = H^0(E, \mc{M}_m(-p_{i+j-n-1}))$ for all $m \geq n$. 
\end{enumerate}
Let $I = U \cap Q(i,r,e,p_j)$, and let $M = U/I^{sat}$.  Then:  
\begin{enumerate}
 \item $M_n \neq 0$.
\item $M$ is 1-critical and $M g = 0$.
\item $\bigl( \rann_U(M_n) \bigr)_{m} = \bigl( U \cap J(p_{i+j-n-1})\bigr)_{m}$ for all $m \gg 0$.
\end{enumerate}
\end{lemma}
\begin{proof}    
(1) Let $L = U/I$, so that $M = L/L'$ where $L'$ is the largest finite-dimensional submodule of $L$.
Since $n > i$, it follows from hypothesis (A) and Lemma~\ref{lem:Qfactor}(1) that 
$\dim L_{m} \leq 1$ for all $m \geq n$ and that
\begin{equation}\label{step11}
 U_n J(p_{i+j-n-1}) \subseteq Q(i,r,e,p_j).
\end{equation}
If $M_n=0$, then  $L_n U_m = 0$ for all $m \gg 0$.  Then $U_n U_m \subseteq Q(i,r,e,p_j)$ for all $m \gg 0$.  
By hypothesis (B), $U_m + J(p_{i+j-n-1})_m = T_m$ for $m \geq n$, so $U_n T_m \subseteq Q(i,r,e,p_j)$ for $m \gg 0$ also.    Since $Q(i,r,e,p_j)$ is a saturated right $T$-ideal, $U_n \subseteq Q(i,r,e,p_j)$, contradicting the hypotheses.  
 
(2) By Lemma~\ref{lem:layer}(3),  $g \in J(p_{i+j-n-1})$, whence   $M_{\geq n} \cdot g = 0$ and   
$\rann_{U}(M)\supseteq U_{\geq n}g=gU_{\geq n}$.   By construction,  $M$ has no finite dimensional submodules, 
and so   $Mg=0$. Thus  $M$ is a $\overline{U}$-module. 
Also,   $\dim_{\kk}M_m\leq \dim_\kk L_{m}\leq 1$ for all $m\geq n$ and $M\not=0$ by Part~(1), so 
$\GKdim(M) = 1$.
Since $M$ is noetherian, it has a $\overline{U}$-submodule $M'$ maximal with respect to the property $\GKdim(M/M') = 1$.  Then $M/M'$ 
is $1$-critical.   However, by  %Proposition~\ref{thm:equivord} and 
Corollary~\ref{cor:curvept}(2) any 1-critical $\overline{U}$-module $N$ has $\dim N_m=1$ for all $m\gg 0$.  Thus 
$M'$ is finite-dimensional; hence $M' = 0$ and $M$ is $1$-critical.

(3) Since $M$ is $1$-critical, its cyclic submodule $N = M_n U$ must also be $1$-critical.  Thus 
$\dim_{\kk} M_n = 1 = \dim_{\kk} N_n$ for $n \gg 0$, forcing $M \ppe N$.
In particular, we must have $\rann_U(M_n)_m \subsetneqq U_m$
 for all $m \gg 0$.
By \eqref{step11},  $\rann_U(M_n) \supseteq U \cap J(p_{i+j-n-1})$.   
Thus \eqref{star} forces $\rann_U(M_n)_m =  (U \cap J(p_{i+j-n-1}))_m $ for all $m \gg 0$.
\end{proof}

\begin{corollary}\label{step 2}
Assume that we have the hypotheses of Lemma~\ref{step 1}.  Assume in addition to (A), (B) that we have $e< \mu$ and 
\begin{enumerate}
\item[(C)] $\overline{U}_{\geq n} \subseteq \overline{J(e p_{j + i-1})} = H^0(E, \mc{M}_n(-e p_{j + i-1}))$, but \\
$\overline{U}_n \nsubseteq \overline{J( (e+1) p_{j + i -1})} = H^0(E, \mc{M}_n(-(e+1) p _{j + i -1}))$.
\end{enumerate}
Then $U_n \cap Q(i,r,e,p_j) = U_n \cap J((e+1)p_{i+j-1})$.
\end{corollary}
\begin{proof}
  Let $I=(U \cap Q(i,r,e,p_j))^{\sat}$ and $M = U/I$. 
Similarly, let $H = (U \cap J((e+1)p_{i+j-1}))^{\sat}$ with  $N = U/H$.
Note that $Q(1, d,  d, p_{j+i-1}) = J( d p _{j + i -1})$ for any $d$.   Also, since $g \in J(d p_{j + i-1})$, 
hypothesis (C) is equivalent to $U_{\geq n} \subseteq J(e p_{j + i-1})$ but $U_n \nsubseteq J((e+1) p_{j + i-1})$.
Thus hypothesis (C) implies that the hypothesis (A)
of Lemma~\ref{step 1} also 
holds for $(i', r', e') = (1, e+1,e+1)$ and $j' = i+j-1$.
Also, hypothesis (B) for these values is the same as hypothesis (B) for the old values.
Since $e < \mu$, the hypotheses of Lemma~\ref{step 1} hold for $(i', r', e')$.

We may now apply Lemma~\ref{step 1} to $M$ and $N$.
Thus:
$M_n, N_n \neq 0$, 
both $M,N$ are 1-critical and killed by $g$, and $\rann_U(M_n)$ and $\rann_U(N_n)$ are both  equal  to $U \cap J(p_{i+j-n-1})$ in large degree.
By Corollary~\ref{prop:shiftpt},  we have $M\cong N$ and so $I=H$.  
Thus since $U_n \cap Q(i,r,e,p_j)$ and $U_n \cap J((e+1)p_{i+j-1})$ are already saturated in degree $n$
by Lemma~\ref{step 1}(1), we have
\[  U_n \cap Q(i,r,e,p_j) = I_n = H_n = U_n \cap J((e+1)p_{i+j-1}).\] 
\vskip-24pt\end{proof}
   
We also need the left-sided versions of the two preceding results.  Since the statements and proofs of these are largely symmetric, we give a combined statement of the left-sided versions, with an abbreviated proof.
  We note that  a consequence of \cite[Lemmas~3.5 and 6.1]{RSS2} is  that 
\beq \label{JJprime}  
J'(dp_j)_n = J(dp_{j+n-1})_n \quad \text{and} \quad \bbar{J'(dp_j)_n} = H^0(E, \sM_n(-dp_{j+n-1})) = \bbar{J(dp_{j+n-1})_n}.
\eeq

\begin{lemma}\label{left steps}
 Let $U$ be a   $g$-divisible graded subalgebra of $T$ with $Q_{\gr}(U) = Q_{\gr}(T)$. 
Suppose that $n > i \geq 1$, $1 \leq r \leq e <  \mu$, and $h \in \mb{Z}$.
Suppose further that
\begin{enumerate}
\item[(A$'$)] $U_{\geq n} \subseteq Q'(i, r-1,e,p_h)$, but $U_n \nsubseteq Q'(i, r, e, p_h)$; 
\item[(B$'$)] $\overline{U}_m \nsubseteq \overline{J'(p_{h - i + n + 1})_m} = H^0(E, \mc{M}_n(-p_{h - i + n+m}))$ for 
$m \geq n$; and 
\item[(C$'$)]  $\overline{U}_{\geq n} \subseteq \overline{J'(e p_{h-i+1})}$, but 
\\ $\overline{U}_n \nsubseteq \overline{J'( (e+1) p_{h-i+1})_n} = H^0(E, \mc{M}_n(-(e+1) p _{h-i+n}))$.
\end{enumerate}
Then 
$U_n \cap Q'(i,r,e,p_h) = U_n \cap J'((e+1)p_{h-i+1})$.  
\end{lemma}
\begin{proof}
 The equalities in (B$'$), (C$'$) follow from \eqref{JJprime}. The rest of the proof is symmetric to the proofs of 
 Lemma~\ref{step 1} and Corollary~\ref{step 2}.   In particular, one uses Lemma~\ref{lem:Qfactor}(2) in place 
of \ref{lem:Qfactor}(1).
\end{proof}

 The next   result is the heart of the proof that $U$ and $T(\divd)$ are equivalent orders. 
   
\begin{proposition}\label{prop:UinR}
 Let $U$ be a $g$-divisible cg subalgebra of $T$ with $Q_{\gr}(U) = Q_{\gr}(T)$.
Let $(\divy, \divx, k)$ be geometric data for $\bbar{U}$ and let $\divd = \sum e_p p$ be a  normalised divisor for this data.
Then
\[ U \subseteq T_{\leq k} * T(\divd).\]
\end{proposition}

\begin{proof}
If $\divd = 0$ the result is trivial, so we may assume that $\divd > 0$.
 Suppose that $ U \not\subseteq T_{\leq k} * T(\divd)$.

 By \cite[Lemma~6.6]{RSS2}, $T_{\leq k} * T(\divd) = \bigcap_p T_{\leq k} * T(e_p p)$, where 
the intersection is over the normalised orbit representatives $p$.   
Thus there is some such $p$ so that $U \not \subseteq T_{\leq k} * T(e_p p)$.  
 Let $e = e_p < \mu$.  
 By \cite[Lemma~6.6]{RSS2},  again, for $n \in \NN$ we have
\[ \bigl( T_{\leq k} * T(e p) \bigr)_n = 
 \bigcap  \Big\{Q(i,r, e, p_j)_n \ |\  i \geq 1, \,  k \leq j \leq n-i, \, 1 \leq r \leq e\Big\}.
 % \bigcap_{\substack{i \geq 1 \\ k \leq j \leq n-i \\ 1 \leq r \leq e}} Q(i,r, e, p_j)_n.
\]
Thus, there are $i \geq 1$, $1 \leq r \leq e$, and $n,j \in \NN$ 
with $1 \leq k \leq j \leq n-i$ such that  
\beq\label{false} U_n \not \subseteq Q(i,r,e,p_j)_n.\eeq
Without loss of generality we can assume that $i$ is  minimal such that we can achieve this for some such $n,j,r$. 
Note that $i \geq 2$, since $\overline{Q(1,r,e,p_j)} = H^0(E, \mc{M}(-rp_j))$ by Lemma~\ref{lem:layer}(2), 
and the sections in $\overline{U_n}$ vanish to multiplicity $e$ at $p_j$ by \eqref{combinatorics}.
Then choose $r$ minimal  (for this $i$) 
so that \eqref{false} holds for some such $n,j$.
 Intuitively, we are finding a ``divisor triangle" of minimal size $i$ such that the corresponding right ideal does not contain $U_n$, with deepest layer vanishing condition in this triangle to be of multiplicity $r$ as small as possible.

{\bf Claim 1:}
$ U_n \cap Q(i,r,e,p_j) =  U_n \cap J((e+1)p_{j +i -1})$.
\begin{proof}
We check the hypotheses of  Corollary~\ref{step 2}.  Hypothesis (A) follows
by minimality of~$r$ when $r > 1$.  When $r = 1$, then we need 
$U_{\geq n} \subseteq Q(i, 0, e, p_j)$.  Now, by \cite[(6.7)]{RSS2},   
\begin{equation}
\label{eq:Qint}
Q(i, 0, e, p_j) = Q(i-1, e, e, p_j) \cap Q(i-1, e, e, p_{j+1}).
\end{equation}

Since $U_{\geq n}$ is contained in both $Q(i-1, e, e, p_j)$ and $Q(i-1, e, e, p_{j+1})$ by the minimality of $i$, hypothesis (A) holds in this case as well.

Note that by \eqref{normal-eq}, the equation \eqref{combinatorics} gives exactly 
the vanishing (with multiplicities) at points on the $\tau$-orbit of $p$ for the sections in 
$\overline{U}_n \subseteq H^0(E, \mc{M}_n)$.  In particular, (B) holds because $i + j -n-1 < 0$.
Similarly, (C) holds by \eqref{combinatorics} since $k \leq j + i - 1 \leq n-1$. 
Thus Corollary~\ref{step 2} gives the result.
\end{proof}

{\bf Claim 2:}
$U_n \cap Q'(i, r, e, p_h) = U_n \cap J'((e+1)p_{h - i + 1})$ for $h = j + i - n$.
\begin{proof}
This similarly follows from Lemma~\ref{left steps} once we verify the hypotheses of that result.  
For (B$'$), note that $h - i + n+m = j + m \geq k+m$ and use 
\eqref{combinatorics}.  Hypothesis (C$'$) follows again from \eqref{combinatorics} 
since $h - i + n = j$ satisfies $k \leq j \leq n-1$.

It remains to verify (A$'$).  
  We will use the equality
\beq \label{QQprime} Q(k,r,m,p)_n = Q'(k,r,m,p_{k-n})_n,\eeq
proven in \cite[Proposition~6.8(3)]{RSS2}.  Thus,  
   $U_n \not\subseteq Q'(i, r, e, p_h)_n$.  
Now let $n' \geq n$.   Suppose that $r > 1$.  The minimality hypothesis on $r$ means that for any $j'$ with $k \leq j' \leq n'-i$, we have $U_{n'} \subseteq Q(i, r-1, e, p_{j'})_{n'}$.  
In particular, since $k \leq j+n'-n \leq n'-i$, we have $U_{n'} \subseteq Q(i,r-1, e, p_{j+n'-n})_{n'}= Q'(i, r-1, e, p_h)_{n'} $ by \eqref{QQprime}.  Thus $U_{\geq n} \subseteq Q'(i,r-1,e,p_h)$.  
If instead $r = 1$, then 
\[
Q'(i, 0, e, p_h)_{n'} = Q(i,0,e,p_{j-n+n'}) = Q(i-1, e, e, p_{j-n + n'}) \cap Q(i-1, e, e, p_{j - n + n' +1}),
\] 
by \eqref{eq:Qint} and \eqref{QQprime}.  But $U_{n'}$ is contained in both 
$Q(i-1, e, e, p_{j - n + n' +1})$ and $Q(i-1, e, e, p_{j-n + n'})$ by minimality of $i$.
Thus $U_{\geq n} \subseteq Q'(i,r-1,e,p_h)$ in this case as well and (A$'$) holds as needed. 
\end{proof}

{\bf Claim 3:} $U_n \cap Q(i,r,e,p_j) = U_n \cap J((e+1) p_j)$.
\begin{proof}
As in the proof of Claim 2, we have $Q(i, r, e, p_j)_n = Q'(i, r, e, p_{j + i - n})_n$, and so that claim gives
\[
U_n \cap Q(i,r,e,p_j) =  U_n \cap Q'(i, r, e, p_{j + i - n}) =U_n \cap J'((e+1)p_{j-n+1})  = U_n \cap J((e+1)p_j),
\]
where we use \eqref{JJprime} in the last step.
\end{proof}

We can now  complete the proof of Proposition~\ref{prop:UinR}.  
Combining Claims 1 and 3, we have
\begin{equation}\label{claim5} U_n \cap J((e+1) p_{j}) = U_n \cap Q(i,r,e,p_j) = U_n \cap J((e+1) p_{i+j-1}).
\end{equation}
Recall that $\overline{U}_n = H^0(E, \mc{M}_n(-\divy -\divx_n))$ and $i\geq 2$. Thus,  by \eqref{combinatorics}
and \eqref{normal-eq} we see that, after taking the image of \eqref{claim5} in $B$, 
the right hand side vanishes to order $e$ at $p_j$, 
while the left hand side vanishes to order $e+1$ at $p_j$.  This  contradiction completes the proof of the proposition.
\end{proof}

We can now quickly prove our first main theorem.

\begin{theorem}\label{thm-correct}
 Let $U$ be a $g$-divisible graded subalgebra of $T$ with $Q_{gr}(U)=Q_{gr}(T)$.  
Then there is an effective divisor $\divd$ on $E$, supported on points with distinct orbits and 
with $\deg \divd < \mu$, so that $U$ is an equivalent order to $T(\divd)$.
 
 In more detail, for some $\divd$ the $(U, T(\divd))$-bimodule $M = \wh{UT(\divd)}$ is a 
finitely generated $g$-divisible right $T(\divd)$-module with $MT=T$.
Set $W=\End_{T(\divd)}(M)$.  Then $U \subseteq W \subseteq T$, the bimodule $M$ is 
finitely generated as a left $W$-module, while $W$, $U$, and $T(\divd)$ are equivalent orders. 
\end{theorem}

\begin{remark} Recall from Lemma~\ref{lem:expectQ} that, if 
$U$ be a $g$-divisible graded subalgebra of $T$ with $D_{gr}(U)=D_{gr}(T)$,
then  $Q_{gr}(U)=Q_{gr}(T)$ 
also holds. However some condition on quotient rings  
 is required for the theorem, since clearly  $U=\kk[g]$ is not equivalent to any $T(\divd)$.
\end{remark}

\begin{proof} 
 By Lemma~\ref{lem:expectQ},  $Q_{gr}(\overline{U}) =Q_{gr}(\overline{T})$ and so we can apply Proposition~\ref{prop:find D}  to $A=\bbar{U}$.  
Let $\divd$, $k $ be as defined  there;   
thus if $R = T(\divd)$ then $\bbar{U}$ and $\bbar{R}$ are equivalent orders.
By  Proposition~\ref{prop:UinR},  $U\subseteq T_{\leq k} * R$.   
 
Let $M= \wh{U R}$ and  $W = \End_R(M)$.  By \cite[Theorem~5.3(5)]{RSS2}, 
$T_{\leq k} * R$ is a noetherian right $R$-module  
and so 
$M\subseteq T_{\leq k}* R$ is a finitely generated right $R$-module.  Clearly $MT=T$ since $1\in M\subseteq T$ 
and so $W\subseteq T$. Thus, by  Lemma~\ref{Sky2-1}(3), 
${}_W M$ is finitely generated, and so $W$ and $R$ are equivalent orders. 
A routine calculation shows that   $M$ is a left $U$-module  and so
 $U \subseteq W$.   

Consider the $(\bbar{W}, \bbar{R})$-bimodule $\bbar{M}$.  
This is finitely generated on both sides, since the same is true of ${}_W M_R$.  
Thus $\bbar{W}$ and $\bbar{R}$ are equivalent orders which,  as $\bbar{R}$ and $\bbar{U}$ are equivalent orders, implies  
 that $\bbar{W}$ and $\bbar{U}$ are likewise.  Finally, as $U\subseteq W\subseteq T$  the hypotheses of
  the theorem ensure that $Q_{gr}(U)=Q_{gr}(W)=Q_{gr}(T)$. Thus, by Proposition~\ref{thm:equivord}, 
  $U$ and $W$ are equivalent orders, and hence so are $U$ and $R$.
\end{proof}

 \begin{corollary}\label{tau-equiv} Suppose that $\divu$ and $\divv$ are two effective, $\tau$-equivalent divisors with 
 degree $\deg \divu\leq \mu-1$. Then $T(\divu)$ and $T(\divv)$ are equivalent orders.
 \end{corollary}
 
 \begin{proof} Consider the construction of the divisor $\divd$ in Theorem~\ref{thm-correct} starting from the 
 algebra $U=T(\divu)$.  Thus $\divd=\sum e_pp$ is the divisor  constructed in Proposition~\ref{prop:find D}
 and there is considerable flexibility in its choice.  To begin, in
 the proof of Corollary~\ref{geomdata}, one sees that $\divy=0$ and $\divx=\divu$.
  For each orbit $\mathbb{O}$ of $\tau$   a point  $p$ is then chosen such that $\divu|_{\mb O}$ 
  is supported on $X_{\mb O}=\{p_0=p(\mb O),p_1,\dots, p_k\}$. For each such  orbit, we can  replace $p_0$ 
  by some $p_{-r}$ and increase $k$ 
   so that both $\divu|_{\mb O}$ and $\divv|_{\mb O}$ are supported on $X_{\mb O}$. Then 
   $\divd=\sum_{\mb O} e_pp(\mb O)$,  for these choices of points $p({\mb O})$, and  $e_p=\deg(\divu|_{\mb O}) $. 
   As $\divu$ and $\divv$ are $\tau$-equivalent, 
   $\deg(\divu|_{\mb O}) =\deg(\divv|_{\mb O}) $ for each orbit $\mb O$, and hence the divisor $\divd$ is the same 
   whether we started with  
   $T(\divu)$ or  $T(\divv)$. Hence, by Theorem~\ref{thm-correct}, $T(\divu)$ and $T(\divv)$ are both 
   equivalent  to $T(\divd)$ and hence to each other.
   \end{proof}

\begin{remark}\label{correct2}  
One disadvantage of  Theorem~\ref{thm-correct} is that the $(U,T(\divd))$-bimodule $M$ constructed there
need not be finitely generated 
as a left $U$-module. Using \cite[Proposition~3.1.14]{MR} and the fact that our rings are noetherian,
one can easily produce such a bimodule. However,   this typically  lacks the  extra structure 
inherent in $M$ (notably that $MT=T$) and so is less useful for our purposes.
As  will be seen in the next section, such a problem disappears when one works with maximal orders 
(see Corollary~\ref{correct3}, for example) 
and this will in turn give  extra information about the structure of such an algebra.
\end{remark}

%%%%%%%%%%%%%%%%%%%%%%%%%%
%%%%%%%%%%%%%%%%%%%%%%%

\section{On endomorphism rings of
\texorpdfstring{$T(\divd)$-modules}{LG}} \label{MAX-SECT}

 Given a $g$-divisible algebra $U\subseteq T$, Theorem~\ref{thm-correct} provides a module $M$
 over some blowup $T(\divd)$ with $U\subseteq \End_{T(\divd)}(M)$.  In this section, 
 we reverse this procedure by obtaining detailed properties of   such endomorphism rings
 (see Proposition~\ref{correct25} and Theorem~\ref{thought9}). These results 
  provide important information about the structure of maximal $T$-orders that will in turn
   be refined over the next two sections to prove the main result Theorem~\ref{mainthm-intro1}
    from the introduction.

 We begin with an expanded version of a definition from the introduction.

\begin{definition}\label{maxorder-defn}
Let $U\subseteq V$ be Ore domains with the same quotient ring $Q(U)$.
We say that $U$ is a \emph{maximal $V$-order}
if there exists no order $U\subsetneqq U'\subseteq V$ that is equivalent to $U$.
We note that
if $U$ and $V$ are graded (in which case requiring that $Q_{gr}(U) = Q_{gr}(V)$ is sufficient), then
this is the same as being maximal among \emph{graded} orders equivalent to $U$ and contained in $V$.
Indeed, suppose that  
$U$ has the latter property, but that $U\subsetneq A\subseteq V$ for some equivalent order $A$.
If  $A$ is given 
the filtration induced from the graded structure of $V$, then the associated graded ring $\gr A$ will still satisfy
$U\subsetneq \gr A\subseteq V$ and be equivalent to $U$, giving the required contradiction.

  When $V=Q(U)$, or  $V = Q_{gr}(U)$
if $U$ is graded, a maximal $V$-order is simply called
\emph{a maximal order}.
   \end{definition}

 We are mostly interested in maximal $T$-orders. We introduce this concept because 
   maximal $T$-orders  need not 
 be  maximal orders
  (see Proposition~\ref{Tmax-eg}), although the difference is not large (see Corollary~\ref{correct3}). 
We first want to study   the endomorphism ring $\End_{T(\divd)}(M)$ arising from Theorem~\ref{thm-correct}
and begin with two useful lemmas.

 \begin{lemma}\label{cozzens}  Let $A$ be a noetherian domain with quotient division ring $D$. If $N$ is a finitely
  generated right $A$-submodule of $D$ then $\End_A(N^{**})$ is the unique maximal order among  orders containing 
  and equivalent to $\End_A(N)$.
\end{lemma}

\begin{proof} This is what is proved in \cite[Theorem~2.7]{Co}, since $\End_A(N^*) = \End_A(N^{**})$.  
\end{proof}

\begin{lemma}
\label{endo}
Let $A$ and $B$ be  rings such that $A$ is left  noetherian and suppose that $M$ is an $(A,B)$-bimodule that is finitely generated on both sides, and that $N$ is a finitely generated right $B$-module.  Then  $\Hom_B(N, M)$ is a finitely generated left $A$-module.  
In particular, $\End_B(M)$ is a finitely generated  left $A$-module and, if $B$ is left noetherian, then 
 $N^* = \Hom_B(N, B)$ is a finitely generated left $B$-module.
\end{lemma}

\begin{proof} 
A  surjective $B$-module homomorphism  $B^{\oplus n} \to N$   induces an injective  left $A$-module 
homomorphism $\Hom_B(N, M) \hookrightarrow \Hom_B(B^{\oplus n}, M) \cong M^{\oplus n}$. 
 Since $M$ is a noetherian  left $A$-module, $\Hom_B(N,M)$ is a finitely generated left $A$-module.
\end{proof}

We are now ready to prove the first significant result of the section. 
 Until further notice, all duals $N^*$ will be taken as $R$-modules, for $R=T(\divd)$.

\begin{proposition}\label{correct25}
Let $\divd$ be an effective divisor on $E$ with $\deg \divd < \mu$ and let $R = T(\divd)$.
Let $M \subseteq T_{(g)}$ be a $g$-divisible finitely generated graded right $R$-module with $MT=T$ 
and set $W= \End_R(M)$ and $F= \End_R(M^{**})$.  Then:
\begin{enumerate}
\item $F$, $V=F\cap T$ and $W$ are $g$-divisible algebras with   
$Q_{gr}(W)=Q_{gr}(V)=Q_{gr}(F)=Q_{gr}(T).$
\item  $F$ is the unique 
maximal order containing and equivalent to  $W$, while $V$ is the unique maximal $T$-order 
containing and equivalent to  $W$;
\item   there is an ideal $K$ of $F$ 
with $K \subseteq W$ and $\GKdim F/K \leq 1$.
\item  $R=\End_W(M)= \End_{F}(M^{**})$. 
\end{enumerate}
\end{proposition}  
 
\begin{proof}  Since $Q_{gr}(R)=Q_{gr}(T)$ by Proposition~\ref{8-special22}, clearly the same is true for $W$, $V$ and $F$.
As in \eqref{endo-defn}, given a right $R$-module $N\subset Q_{gr}(R)$ we 
identify $$N^*=\Hom_R(N,R)=\{\theta\in Q_{gr}(R)\,  |\, \theta N\subseteq R\},$$ and similarly for left modules.
By Lemma~\ref{Sky2-1}(3), $W$ is $g$-divisible and ${}_WM$ is finitely generated.  
Thus the left-sided version of Lemma~\ref{endo} shows that $\End_W(M)$ is a finitely
 generated right $R$-module. Moreover, by Proposition~\ref{8-special22},    $R$ is  a maximal order and so 
  $R=\End_W(M)$. 
  
By Lemma~\ref{lem:finehat}(3), $M^{**}$ is $g$-divisible with $M^{**}\subset T_{(g)}$. Since 
$M^{**}$  is clearly a finitely generated right $R$-module, the same logic ensures that $F$ is $g$-divisible,
 ${}_FM^{**}$ is finitely generated and  $\End_F(M^{**})=R$.  
By Lemma~\ref{cozzens}  $F\supseteq W$ and $F$ is the unique maximal order containing 
and equivalent to $W$.  This automatically ensures that $V=F\cap T$ is maximal among $T$-orders containing and equivalent to $W$. Clearly $V$ is also $g$-divisible.

It   remains to find the ideal $K$.  By Proposition~\ref{div-noeth}, both $W$ and $F$ are noetherian.
 By Proposition~\ref{8-special22} %, $R$ is CM and Auslander-Gorenstein  and so,   by 
 and Lemma~\ref{thou11}(1),
  $\GKdim_R(M^{**}/M)\leq \GKdim(R)-2=1$.   
Since $M$ is $g$-divisible, $X=M^{**}/M$ is   $g$-torsionfree and so,  by Lemma~\ref{lem:constant},
$X$ is a finitely generated right $\kk [g]$-module.
 Since $M\subseteq M^{**}\subset T_{(g)}$ the action of $g$   is central on  $X$   and  so  
  $X$  is also a finitely generated left $\kk[g]$-module.
  Now, it  is routine to check that $M^{**} $ and hence $X$ are left $W$-modules, while $\kk[g]\subseteq W$ since $W$ is $g$-divisible.
  Thus, $X$ and hence $M^{**}$ are  finitely generated left $W$-modules.
  Moreover,  $\GKdim_W(X)\leq \GKdim_{\kk[g]}(X)\leq 1$
 and so, by  \cite[Lemma~5.3]{KL},  $I= \lann_W(X)$   satisfies $\GKdim(W/I)\leq 1$.  

 Now consider $F$.  First, $(IF)M\subseteq IFM^{**}\subseteq IM^{**}\subseteq M$ and hence $IF=I\subseteq W$. 
Thus $F$ is a finitely generated right $W$-module  and (on the left)  $\GKdim_W(F/W)\leq\GKdim(W/I)\leq  1$.
  On the other hand, as ${}_WM^{**}$ is finitely generated,    Lemma~\ref{endo} implies that $F = \End_R(M^{**})$ is a finitely generated left $W$-module.  Thus, by \cite[Lemma~5.3]{KL}, again,  the right annihilator $I' =\rann_W(F/W)$ satisfies $\GKdim W/I'\leq 1$.  Thus  $K=I'I$   is an ideal of both $F$ and $W$.  
By  symmetry and partitivity, \cite[Corollary~5.4 and Theorem~6.14]{KL},  $\GKdim(F/K) = 1$. 
\end{proof}

Pairs of algebras $(V,F)$ satisfying the conclusions  of the proposition    will appear multiple times in this paper 
and so we turn those properties into a definition.  For a case when $F\not=V$, see Proposition~\ref{Tmax-eg}.
 
  \begin{definition}\label{max-pair-defn} A pair $(V,F)$ is called a \emph{maximal order pair} if:
  \begin{enumerate}
  \item  
  $F$ and $V$ are $g$-divisible, cg algebras with  $V\subseteq F\subseteq T_{(g)}$ and $V\subseteq T$;
  \item $F$ is a maximal order in $Q_{gr}(F)= Q_{gr}(T)$ and $V=F\cap T$ is a maximal $T$-order;
  \item   there is an ideal $K$ of $F$ 
with $K \subseteq V$ and $\GKdim F/K \leq 1$.
\end{enumerate}
\end{definition} 

The next result  illustrates the significance of Proposition~\ref{correct25} to the structure of   maximal $T$-orders.

\begin{corollary}\label{correct3}    Let $U\subseteq T$ be a $g$-divisible cg maximal $T$-order. 
\begin{enumerate}
\item    There exists an effective divisor $\divd$ on $E$, with $\deg \divd < \mu$, and a $g$-divisible
 $(U,T(\divd))$-module
 $M\subseteq T$ with $MT = T$ that is   finitely generated as both a left $U$-module and a right   $T(\divd)$-module.
 Moreover, $U=\End_{T(\divd)}(M)$ and $T(\divd)=\End_U(M)$. 
 \item $(U, F=\End_{R}(M^{**}))$ is a maximal order pair; in particular, 
if $U$ is a maximal order then $U = F$. 
 
 \item Suppose that every ideal $I$ of $T(\divd)$ with $\GKdim(T(\divd)/I) = 1$ satisfies $\GKdim T/IT\leq 1$ (in particular, this holds
if $T(\divd)$ has no such ideals $I$). Then $U = F$ is a maximal order.
\end{enumerate}
\end{corollary}
 
\begin{proof}  (1) By Theorem~\ref{thm-correct}, there is an effective divisor $\divd$ with $\deg \divd < \mu$ so that 
\[ U \subseteq V= \End_{T(\divd)}(M) \subseteq T,\]
where $M = \wh{U T(\divd)}$ is a finitely generated $g$-divisible graded right $T(\divd)$-module with $MT=T$.
 By Theorem~\ref{thm-correct} again, $V$ and $U$ are equivalent orders.   Since $U$ is a maximal $T$-order, 
 this forces $U=V$.  Finally, $T(\divd) = \End_U(M)$ by Proposition~\ref{correct25}.

(2)  As $U=V$, this is a restatement of Proposition~\ref{correct25}(2). 

(3)    Just as in the proof of Proposition~\ref{correct25},
$J = \rann_R M^{**}/M$ is an ideal of $R$ with $\GKdim(R/J) \leq 1$.  
Note that since $M$ is $g$-divisible, either $M  = M^{**}$ and $J = R$, or else $\GKdim(R/J) = 1$.

In either case, the hypotheses imply that  $\GKdim T/JT \leq 1$.
Now $M^{**}JT\subseteq MT=T$. Thus  
\[
\GKdim (\alpha T +T)/T
\leq \GKdim T/JT\leq 1, 
\]
for any $\alpha \in M^{**} $.  
By Proposition~\ref{8-special22} and  Lemma~\ref{thou11}(1),
this implies that $M^{**}\subseteq T$. This in turn implies that $M^{**}T=T$ and hence 
that $F\subseteq T$.  Since $U$ is a maximal $T$-order, $U=F$ is a maximal order.
 \end{proof}

We now turn to the second main aim of this section, which is to describe the structure of $\overline{U}$
for suitable endomorphism rings  $U=\End_{T(\divd)}(M)$.  The importance of this result is that the pleasant properties
of $\overline{U}$ can be pulled back to $U$. 

\begin{theorem}\label{thought9} Let $\divd$ be an effective divisor on $E$ with $\deg \divd < \mu$, and let $R=T(\divd)$.
Let $M$ be a finitely generated $g$-divisible graded right $R$-module with $R \subseteq M \subseteq T$.
Let $U= \End_R(M)$ and  $F= \End_R(M^{**})$.
Then there is an effective divisor $\divy$ on $E$ so that
\begin{equation}\label{thought98}
 \bbar{F} \ppe \bbar{U} \ppe \End_{\bbar{R}}(\bbar{M}) \ppe B(E, \sM(-\divx), \tau)
\qquad \text{for\  \ $\divx=\divd-\divy+\tau^{-1}(\divy)$.}
\end{equation}

Moreover,  if $V = F \cap T$ then $U\subseteq V\subseteq F$ and $(V,F)$ is a maximal order pair.
\end{theorem}

The proof of Theorem~\ref{thought9} depends on  a series of  lemmas that will take the rest of this section.   
Before getting to those results we make some comments and a definition. We first want to regard the ring 
$F$ from the theorem as a  blow-up of $T$ at the   divisor 
$\divx $ on $E$, even if $\divx$ is not effective.  We   formalise this  as follows.

\begin{definition}\label{def:blowup}
Let $\divx$ be a (possibly non-effective) divisor on $E$ with $0 \leq \deg \divx < \mu = \deg \sM$.  We say that a cg algebra $F \subseteq T_{(g)}$
is {\em a  blowup of $T$ at $\divx$}, if:
\begin{itemize}
\item[(i)] $F$ is part of a maximal order pair $(V,F)$ with $Q_{\gr}(F) = Q_{\gr}(T)$; and 
\item[(ii)] $\bbar{F} \ppe B(E, \sM(-\divx), \tau)$.
\end{itemize} 
\end{definition}

\begin{remarks}\label{thought91}
 (i) The reader should regard this definition of a blowup is temporary in the sense that it
 will be refined in Definition~\ref{virtual-defn}
 and justified in Remark~\ref{remark-7.45}.
 One  caveat   about the concept 
  is that there may not be a unique blowup of $T$ at the  divisor $\divx$; in the context of   
   Theorem~\ref{thought9} there may 
be different $R$-modules $M$   leading to distinct  blow-ups $F$, which nonetheless 
have factors $\overline{F}$ which are equal in large degree.  
See Example~\ref{non-AG} and Remark~\ref{non-AG4}(2). 

(ii)  It follows easily from Theorem~\ref{thought9} that  a maximal order pair $(V,F)$  does give 
a  blowup of $T$ at an appropriate (possibly non-effective) 
 divisor $\divx$. The details are given in Theorem~\ref{thm:converse}
 which also gives a converse to Theorem~\ref{thought9}.

 (iii)   We conjecture that, generically,  $T(\divd)$ will  have no \cogs\ in Theorem~\ref{thought9} and so, 
  by Corollary~\ref{correct3}(3),  $U = F$ will then be  a maximal order. 
  For an example where this happens see Example~\ref{non-AG} and, conversely,  for an example 
  when  $U\not=F$ and $F\not\subseteq T$ see Proposition~\ref{Tmax-eg}.
  \end{remarks}

\begin{notation}\label{dual-notation}
 For the rest of the section, we write   $N^* =\Hom_U(N,U)$ provided  that   the ring $U$ is clear from the context.
 In particular, given  a $g$-divisible left ideal $I$ of  $R$ then   $\overline{I}^* = \Hom_{\Rbar}(I/gI,\,\Rbar)$ 
while $\overline{I^*}= \overline{\Hom_R(I, R)}$. 
  Recall from Lemma~\ref{thou11} that a $R$-module $M$ is \emph{$\alpha$-pure} provided 
$\GKdim(M)=\GKdim(N)=\alpha$ for all non-zero submodules $N\subseteq M$.
\end{notation}

 The main technical result we will need is the following,   
showing that ``bar and star commute" (up to a finite-dimensional vector space).
 
 \begin{proposition}\label{thou7}
Let $R=T(\divd)$ for an effective divisor $\divd$ with $\deg \divd < \mu$.
\begin{enumerate}
\item Let $I$ be a proper, $g$-divisible left  ideal of $R$ for which $R/I$ is $2$-pure.  
Then $I^*/R$ is  a $g$-torsionfree, 2-pure right module; further,   $I^*\subseteq T_{(g)}$ and 
 $ \overline{I^*} \ppe   \overline{I}^* $.
\item If $M$ is a finitely generated $g$-divisible graded right $R$-module  with $R \subseteq M \subseteq T$, then 
$\overline{M^*} \ppe   \overline{M}^* $.
\end{enumerate}
\end{proposition}

\begin{proof}    
(1) By Lemma~\ref{thou11}(2),  $I^*/R$ is  2-pure.   By Lemma~\ref{Sky2-1},   $I^* \subseteq T_{(g)}$   and since $R$ is $g$-divisible, $T_{(g)}/R$ and 
hence $I^*/R$ are $g$-torsionfree.

   From  the  exact sequence $0 \to Rg \to R \to \bbar{R} \to 0$ we obtain  the long exact sequence of right $R$-modules \begin{equation}\label{thou8}
\begin{array}{ll}
&0  \too \Hom_R(R/I, \bbar{R})   \too \Ext^1_R(R/I, Rg) \too \Ext^1_R(R/I, R) \too \Ext^1_R(R/I, \bbar{R})  \\ \noalign{\vskip 6pt}
&\hfill \qquad \overset{\phi}{\too} \Ext^2_R(R/I, Rg) \overset{\psi}{\too} \Ext^2_R(R/I, R) \too \Ext^2_R(R/I, \bbar{R}) \too \cdots
\end{array}
\end{equation} 
By   Proposition~\ref{8-special22},  $\bbar{R}$ is Auslander-Gorenstein and CM. Thus  
    $N= \Ext^2_{\overline{R}}(\overline{R}/\overline{I}, \overline{R})$  has grade $j(N)\geq 2$ and  
hence  $\GKdim(N)\leq 2-2=0$. Therefore, by \cite[Lemma~7.9]{RSS2}  
  $ \Ext^2_R(R/I, \bbar{R}) = N $ is finite dimensional      and     the map  $\psi$ in \eqref{thou8} is surjective in
large degree.  If $E = \Ext^2_R(R/I, R)$, this says that $\psi: E[-1] \to E$ is surjective in large degree.
Since  $\dim_{\kk}E_n<\infty $ for each $n$,  
this forces $\dim_{\kk} E_n \geq  \dim_{\kk} E_{n+1}$ for all $n \gg 0$ and so 
$\dim_{\kk} E_n$ is eventually constant.  In turn, this forces $\phi$ to be zero in large degree.
  
  Next, observe that  $\Hom(R/I, \bbar{R}) = 0$ since $R/I$ is $g$-torsionfree.
Since $\phi$ is zero  in high  degree, the
complex 
\[0   \too \Ext^1_R(R/I, Rg) \too \Ext^1_R(R/I, R) \too \Ext^1_R(R/I, \bbar{R}) \too 0\]
is exact in high degree. 
Using \cite[Lemma~7.9]{RSS2}   this can be identified with the complex 
\[
0 \too (I^*/R)[-1] \overset{\alpha}{\too} I^*/R \too \Ext^1_{\overline{R}}(\overline{R}/\overline{I}, \overline{R})\too 0,
\]
where $\alpha$ is multiplication by $g$. As $I^*$ is $g$-divisible by Lemma~\ref{Sky2-1}(2), it follows that 
\[\overline{I^*}/\overline{R} \cong I^*/(R+I^*g) = \text{coker}(\alpha)
\ppe \Ext^1_{\overline{R}}(\overline{R}/\overline{I}, \overline{R})  = \overline{I}^*/\overline{R}.\]
In particular,     $\dim_\kk \overline{I^*} = \dim_\kk \overline{I}^*$for all $n\gg 0$, 
 and as there is an obvious inclusion $\overline{I^*} \subseteq \overline{I}^*$  we 
conclude that $\overline{I^*} \ppe \overline{I}^*$.   

(2)  Note that $M^{**}/M$ is a $g$-torsionfree module 
of GK-dimension 1, as in the proof of Proposition~\ref{correct25}.   
By Lemma~\ref{lem:constant}, $\dim_\kk \bigl( (M^{**}/M) \otimes_R \bbar{R} \bigr) < \infty$.  
Thus $\overline{M^{**}} \ppe \overline{M}$.  

 Let $J = M^*$. Since $J$ is a reflexive left ideal of $R$, the module $R/J$ is $2$-pure by Lemmas~\ref{8-special22}(3) and \ref{thou11}(3).  Thus Part (1) applies and shows that $\overline{J}^* \ppe \overline{J^*}$.  Next, $\overline{J} \ppe \overline{J}^{**}$ by another use of Lemmas~\ref{8-special22}(3) and \ref{thou11}(3).  Finally, it is easy to see that for any finitely generated graded $\overline{R}$-modules $N$ and $Q$ contained in $Q_{gr}(\overline{R})$, if $N \ppe Q$ then $N^* \ppe Q^*$.
Putting the pieces above together, we conclude that 
\[
\overline{M^*} = \overline{J} \ppe \overline{J}^{**} \ppe \bigl( \overline{J^*}\bigr)^* \ppe \overline{M}^*. \qedhere 
\]
\end{proof}

The last ingredient we need for the proof of Theorem~\ref{thought9} is the following 
description of the endomorphism ring of a torsion-free rank one module over a  
twisted homogeneous coordinate ring. 

\begin{lemma}\label{thought8}
Let $B  = B(E, \mc{L}, \tau)$, where $E$ is a smooth elliptic curve, $\deg \mc{L} \geq 1$, and $\tau$ is of infinite order.  
  Let $N$ be a finitely generated, graded right $B$-submodule of $\kk(E)[t, t^{-1};\tau]$; by  \cite[Theorem~1.3]{AV},   
  $N \ \ppe  \  \bigoplus_{r\geq 0} H^0\left(E,   \sO(\divq)  \otimes  \mathcal{L}_r  \right)$ for 
some divisor $\divq$.  Let $N^* = \Hom_B(N,B) \subseteq \kk(E)[t, t^{-1};\tau]$.
\begin{enumerate}
\item $\End_B(N) \ \ppe\  B\bigl(E,\, \mathcal{L}(\divq-\tau^{-1}(\divq)),\,\tau\bigr).$  
\smallskip
\item  $N N^* \ppe \End_{B}(N)$.
\item $N^* \ \ppe\    
 \bigoplus_{n\geq 0} H^0\Bigl(E,\,   \sL_n \otimes  \sO( -{\tau^{-n}}\bigl(\divq)\bigr)\Bigr);$
\end{enumerate}
\end{lemma}
\begin{proof}     
(1)  Write $G=\End_B(N) \subseteq  \kk(E)[t, t^{-1};\tau]$ and, for each $n$,  
let $\mc{G}_n$ be the subsheaf of the constant sheaf  $\kk(E)$ generated by   $G_n \subseteq \kk(E)$.
Let $\mc{N}_n  = \sO(\divq)  \otimes  \mathcal{L}_n$; thus $N_n = H^0(E, \mc{N}_n)$, and $N_n$ generates 
the sheaf $\mc{N}_n$, for $n\gg0$, say $n \geq n_0$.

For $n \geq n_0$ and $r \geq 0$, the equation $G_rN_n \subseteq N_{n+r}$ forces 
$\mc{G}_r \mc{N}_n^{\tau^r} \subseteq \mc{N}_{n+r}$ and thus
\[ 
  \sG_r \otimes \Bigl( \sO(\divq) \otimes \mathcal L_n \Bigr)^{\tau^r}  
\   \subseteq  \         \mathcal O(\divq) \otimes \mathcal L_{n+r}.
\]
Equivalently
$\mc{G}_r \subseteq \mc{L}_r(\divq - \tau^{-r}(\divq)) = \big(\mc{L}(\divq - \tau^{-1}(\divq))\big)_r$.  This shows 
that $G \subseteq B\bigl(E,\, \mathcal{L}(\divq-\tau^{-1}(\divq)),\,\tau\bigr)$.

Reversing this calculation shows that $\big(\mc{L}(\divq - \tau^{-1}(\divq))\big)_r \, \mc{N}_n^{\tau^r} \subseteq \mc{N}_{n+r}$
for $r, n \geq 0$ and taking sections for $n \geq n_0$ shows that 
$ B\bigl(E,\, \mathcal{L}(\divq-\tau^{-1}(\divq)),\,\tau\bigr) \subseteq \End_B(N_{\geq n_0})$.
To complete the proof we need to prove that $G \ppe  \End_B(N_{\geq n_0})$.
This follows   by \cite[Lemma~2.2(2)]{Rog09} and  \cite[Proposition~3.5]{AZ}  or by a routine computation.

(2)  Clearly $N N^*$ is an ideal of $\End_B(N)$.  However,  by Lemma~\ref{lem:T-prop}(2),
$\End_B(N)$ is just infinite, 
  and so $N N^* \ppe \End_B(N)$.

(3) The proof is similar to that of (1) and, as it will not be used  in the paper,   is left to the reader.
\end{proof}

\begin{proof}[Proof of Theorem~\ref{thought9}]
We first check that $\overline{F}  \ppe  \overline{U}.$
By Proposition~\ref{correct25}   there exists an ideal $K$ of $F$ contained in $U$ and satisfying $\GKdim(F/K)\leq 1$.  In particular, $\GKdim(F/U) \leq 1$.  By Lemma~\ref{Sky2-1}(3),  $U$ 
is $g$-divisible,  and so 
$N = F/U$ is $g$-torsionfree.  It follows from Lemma~\ref{lem:constant} 
that $\GKdim(\overline{F}/\overline{U}) = 0$, and so $ \overline{U} \ppe \overline{F}  .$

% Let $J = M^*$, which  is a left ideal of $R$,  since $R \subseteq M$.   
%Now it is obvious  that $\overline{MJ}= (\overline{M})(\overline{J})$ and  that
Now it is obvious that 
 $U\supseteq M M^*$.  Thus, using Proposition~\ref{thou7}(2),
\[ \bbar{U} \supseteq  (\bbar{M}) (\bbar{M^*}) \ppe (\bbar{M}) (\bbar{M}^*). \] 

Conversely,  by Lemma~\ref{Sky2-1}(3), 
$\overline{U} = \overline{\End_R(M)} \subseteq  \End_{\overline{R}}(\overline{M})$.
Now, $\bbar{R} = B(E, \sM(-\divd), \tau)$.  
Applying Lemma~\ref{thought8} to $\mc{L} = \mc{M}(-\divd)$ and $N = \overline{M}$ gives  
\[
(\bbar{M}) (\bbar{M}^*) \ppe \End_{\overline{R}}(\overline{M}) \ppe  B(E, \sM(-\divx),\tau),
\]
where in the notation of that lemma,    
 $\divy= \divq$ and $\divx = \divd-\divy+\tau^{-1}(\divy)$.    
 That $\divy$ is effective follows from $\bbar{R} \subseteq \bbar{M}$.
Combining the last two displayed equations gives
\eqref{thought98}.

Since $R \subseteq M$, necessarily $MT=T$.  
Thus the second paragraph of the theorem is just a restatement of  Proposition~\ref{correct25}. \end{proof}

 %%%%%%%%%%%%%%%%%%%
 %%%%%%%%%%%%%%%%%%%%
 
\section{The structure of  \texorpdfstring{$g$}{LG}-divisible orders}\label{EXISTENCE}

In this   section we first refine the results from the last two sections to give strong structural results 
 for a $g$-divisible maximal $T$-order $U$  (see Theorem~\ref{thm:converse}) .   Then we use these results to
analyse both arbitrary $g$-divisible orders and ungraded subalgebras of $D=D_{gr}(T)$
(see Corollaries~\ref{main-ish} and \ref{localthm}, respectively).  
In particular, we show that  $U$ is part of a maximal order pair $(U,F)$ for which 
$F$ is a  blowup of $T$ at a (possibly non-effective) divisor $\divx=\divd-\divy+\tau^{-1}(\divy)$
 in the sense of Definition~\ref{def:blowup}.
Here, the divisor  $\divy$ can have arbitrarily high degree but  is not arbitrary, as 
we first explain.

 \begin{definition}\label{virtual-defn}
 Let $\divx$ be a divisor on $E$. For each  $\tau$-orbit $\mathbb{O}$ in $E$ pick    
$p=p_0\in \mathbb{O}$ such that 
  $\divx|_{\mathbb{O}} = \sum_{i=0}^k x_ip_i$, where $p_i=\tau^{-i}(p)$.
 Then $\divx$  is called a \emph{virtually effective} divisor if,  
 for each orbit $\mathbb{O}$ and all $j\in \mathbb{Z}$  the divisor    $\divx$ 
satisfies 
\begin{equation}\label{thought811}
\sum_{i\leq j} x_i\geq 0 \qquad\text{and}\qquad \sum_{i\geq j} x_i\geq 0.
\end{equation}
If $F$ is a blowup of $T$ at a virtually effective divisor $\divx$ then $F$ is called a \emph{virtual blowup}
of $T$.
 \end{definition}

The relevance of this condition is shown by the next result, in which the notation $\divu_k$ for a divisor $\divu$ 
comes from Notation~\ref{tau-notation}.

\begin{proposition}\label{thought81}
  \begin{enumerate}
\item  The divisor $\divx$ in Theorem~\ref{thought9} is virtually effective.
\item   A divisor $\divx$ is virtually effective  if and only if 
 $\divx$ can be written  as $\divx =  \divu -\divv+\tau^{-1}(\divv)$, where $ \divu $ is an  effective divisor
  supported on distinct $\tau$-orbits, and   $\divv$ is an effective  divisor such that 
   $0 \leq \divv \leq  \divu _k$ for some $k$. 
\end{enumerate}
\end{proposition}
\begin{proof}  (1) By Theorem~\ref{thought9},  $\overline{F} \ppe B(E,\mathcal{N},\tau)$, where $\sN= \sM(-\divx)$.
Since $\overline{F} \ppe \overline{U} \subseteq \overline{T} = B(E, \mc{M}, \tau)$, we must have 
$\sN_n \subseteq \sM_n$ for $n \gg 0$.  Now compare this with the computations in the proof of Corollary~\ref{cor:geomdata}.
In the notation of that proof $\mathcal{Y}=\mathcal{O}_E$ and hence $\divy=0$.  
Therefore, as is explained in the proof of \eqref{positivity}, this forces  Equation~\ref{thought811} to hold.

(2)  It is enough to prove this in the case that $\divx$ is supported on a single $\tau$-orbit $\mathbb{O}$ in $E$.

$(\Rightarrow)$ As in Definition~\ref{virtual-defn}, write $\divx=\sum_{i=0}^k x_ip_i$ for a suitable point $p_0\in \mathbb{O}$.   
 Set $e = \sum_{i \in \ZZ}x_i$ and  $ \divu  = ep$.  
 For $j \in \NN$, let $v_j = \sum_{i \geq j+1}x_i$ and put $\divv = \sum_{j \geq 0} v_j p_j$.

 By \eqref{thought811}, $\divv$ is effective.
Also, since $\sum_{i \leq j} x_i \geq 0$ for all $j$,  we have
$$v_j \, = \,  e - \sum_{i\leq j} x_i\,  \leq\,  e\quad \text{for $0 \leq j \leq k-1$, }$$
  while $v_j=0$ for $j \geq k$.  Therefore, $0 \leq \divv \leq  \divu _k = \sum_{i =0}^{k-1} ep_i$.
Finally
\[\begin{array}{rl}   \divu  - \divv + \tau^{-1}(\divv) \ = &
 ep_0 - \sum_{j \geq 0}\bigl(\sum_{i\geq j+1}x_i\bigr)p_j
  +\sum_{j\geq 0}\bigl(\sum_{i \geq j+1}x_i\bigr)p_{j+1} \\
 \noalign{\vskip 8pt}
  = & ep_0-\Bigl(\bigl(\sum_{i \geq 1}x_i\bigr)p_0 +
 \sum_{j \geq 1}\bigl(\sum_{i\geq j+1}x_i\bigr)p_j\Bigr)
  +\sum_{j\geq 1}\bigl(\sum_{i \geq j}x_i\bigr)p_{j}  
  \\
 \noalign{\vskip 8pt}
  = & \sum x_i p_i 
 \ = \  \divx.
\end{array}
\]

$(\Leftarrow)$  Although this is similar to Part~(1), it seems easiest to give a direct proof. 

Write $ \divu =ep=ep_0$ and $\divv= \sum v_i p_i$ for some point $p$ and some  $v_j\geq 0$.  
By definition  $ \divu _k = \sum_{i =0}^{k-1} ep_i$,
and so, by our assumptions,  $0 \leq v_i \leq e$ for $0 \leq i \leq k-1$, and  $v_i=0$ for all other $i$.
Therefore,
\[ \divx =  \divu -\divv+\tau^{-1}(\divv) = (e-v_0)p_0 + \sum_{i \geq 1} (v_{i-1}-v_i)p_i.\]
If $j \leq -1$ then $x_j=0$ and $\sum_{i \leq j} x_i =0$.  If $j \geq 0$, then 
$ \sum_{i \leq j} x_i = e-v_j \geq 0.$
Similarly, if $j \leq 0$ then $\sum_{i \geq j} x_i = e \geq 0$, while 
if $j \geq 1$ then 
\[ \sum_{i \geq j} x_i = \sum_{i = j}^k(v_{i-1}-v_i) = v_{j-1}-v_k = v_{j-1} \geq 0.\]
Thus \eqref{thought811} is satisfied.
\end{proof}

   We are now ready to state our main result on the structure of $g$-divisible maximal $T$-orders.

\begin{theorem}\label{thm:converse} 
\begin{enumerate} 
\item 
Let $V \subseteq T$ be a $g$-divisible cg   maximal $T$-order.  Then the following hold.
\begin{enumerate}
\item  There is a maximal order $F\supseteq V$ such that $(V,F)$ is   a maximal order pair.  
\item    $F$ is a virtual blowup of   $T$ at a virtually effective 
divisor $\divx = \divu-\divv+\tau^{-1}(\divv)$ satisfying $0 \leq \deg \divx < \mu$. 
\item  $\bbar{V}\ppe \bbar{F}  \ppe B(E, \sM(-\divx), \tau)$.
\end{enumerate}
\item If $U\subseteq T$ is any $g$-divisible cg subalgebra with  $Q_{gr}(U)= Q_{gr}(T)$,  there exists 
a maximal order pair  $(V,F)$ as in (1) such that $U$ is contained in and equivalent to $V$.
\item Conversely, let $\divx $ be a virtually effective divisor with $\deg \divx < \mu$. Then 
there exists a blowup $F$ of $T$ at $\divx$.
\end{enumerate}
\end{theorem}

\begin{proof} 
(1) By definition,  $Q_{gr}(V)= Q_{gr}(T)$. 
Now combine Corollary~\ref{correct3}(1,2),  Theorem~\ref{thought9} and Proposition~\ref{thought81}.

(2) By Theorem~\ref{thm-correct}, $U$ is contained in and equivalent to some $\End_{T(\divd)}(M)$
which, in turn,  is contained in and  equivalent to a maximal $T$-order by Proposition~\ref{correct25}.

(3)  Write $\divx = \divu-\divv+\tau^{-1}(\divv)$, where 
 $\divu$, $\divv$, and $k$ are defined  by applying Proposition~\ref{thought81} to $\divx$.  
 By \cite[Lemma~5.10]{RSS2},
   there is a $g$-divisible finitely generated right $T(\divu)$-module $M$ with 
 $T(\divu) \subseteq M \subseteq MT=T$ so that $\bbar{M} \ppe \bigoplus_n H^0(E, \sM_n(-\divu_n+\divv))$. 
 Let $F = \End_R(M^{**})\supseteq U = \End_R(M)$.  By Theorem~\ref{thought9}
and   Lemma~\ref{thought8}(1,2), we have
\[ \bbar{F} \ppe \bbar{U} \ppe  \bbar{M} (\bbar{M})^* \ppe B(E, \sM(-\divx), \tau),\]
and $(F \cap T, F)$ is a maximal order pair. 
\end{proof}

 \begin{remark}\label{remark-7.45} 
 We should explain why   $F$ is called  a virtual blowup 
  of $T$ at $\divx$ both  in this theorem and in  Definition~\ref{virtual-defn}.
 When $\divx$ is effective this is amply justified   in \cite{Rog09} and,
 in that case,  $T(\divx)$ satisfies many of the basic properties of a commutative blowup; in particular it agrees with Van den Bergh's more categorical blowup \cite{VdB}.
 For non-effective $\divx$ there are  several reasons why the notation is reasonable. 
 \begin{enumerate}
 \item As we have shown repeatedly in this paper  the factor
 $ \overline{U}$ of a $g$-divisible algebra $U$  controls much of $U$'s behaviour and so 
 Theorem~\ref{thm:converse}(1c) shows  that $F$ will have
 many of  the basic properties of a blowup at an effective divisor.
\item  This is also supported by the fact, by Theorem~\ref{thm-correct},
 $F$ and $T(\divu)$  are 
equivalent maximal orders  and,   again, many 
properties pass through such a Morita context. 
\item Finally, in the commutative case virtual blowups are blowups,
 both because virtually effective divisors are then effective  and because 
equivalent maximal orders are then equal.
\end{enumerate}  
 \end{remark}
 
   Theorem~\ref{thm:converse} can be easily used to describe arbitrary $g$-divisible subalgebras of $T$.
   We recall that the \emph{idealiser} of a left ideal $L$ in a ring $A$ is the subring 
   $\idealiser(L)=\{\theta\in A :  L \theta \subseteq L\}$.

 \begin{corollary}\label{main-ish} 
 Let $U\subseteq T$ be a $g$-divisible subalgebra with 
 $Q_{gr}(U)=Q_{gr}(T)$. Then,  $U$ is a iterated subidealiser inside a virtual blow-up of~$T$. 
 More precisely, we have the following chain of rings:
 \begin{enumerate}
 \item There is a  virtually effective  divisor $\divx=\divu-\divv+\tau^{-1}(\divv)$ with 
 $\deg(\divx) <\mu $  and a blowup $F$ of $T$ at $\divx$ 
 such that $V=F\cap T$   contains and is equivalent  to $U$,
 while  $(V, F)$ is a maximal order pair.
 
 \item   There exist a $g$-divisible algebra $W$ with $U\subseteq W\subseteq V$ such that $U$ is a 
 right sub-idealiser inside $W$ and $W$ is a left  sub-idealiser inside $V$. In more detail, 
 \begin{enumerate}
 \item  There exists a graded $g$-divisible left ideal $L$ of $V$ such either $L = V$ or else $V/L$ is 
 2-pure, and a $g$-divisible ideal  $K$ of  $X=\idealiser(L)$ such that
  $K\subseteq W\subseteq X$
 and $\GKdim_X(X/K)\leq 1$;
 \item $V$ is a finitely generated left $W$-module, while $X/K$ is a finitely generated $\kk[g]$-module and so $X$ is finitely generated over $W$ on both sides;  
 \item   the   properties given for $W\subseteq V$ also hold for the pair $U\subseteq W$, but with left and right interchanged.
\end{enumerate}
 \end{enumerate}
 \end{corollary}
 
  \begin{proof}  (1)  Use    Theorem~\ref{thm:converse}(1,2).
      
 (2)  By (1), $aVb\subseteq U$ for some $a,b\in U\smallsetminus\{0\}$. Set $W'=U+Vb$ and $W=\wh{W'}$. 
  By Lemma~\ref{lem:finehat}(1) $aW = \wh{aW'}\subseteq \wh{U}=U.$  
 By Proposition~\ref{div-noeth}, $W$ is noetherian and so (modulo a shift) 
 $V\cong Vb$ is  a finitely generated left $W$-module. 
 Similarly,  $W$ is a finitely generated right $U$-module.  
  We will now just prove Parts~(2a,2b) leaving the reader to check that the same argument does indeed work
   for the pair $(U,W)$.

 Write $V=\sum_{i=1}^vWe_i$ for some $e_i$. Then the right annihilator  
 $K=\text{r-ann}_W(V/W)=\bigcap\text{r-ann}(e_i)\not=0.$  
 Let $L/K$ be the largest left $V$-submodule  of $V/K$ with $\GKdim(L/K)\leq 1$.    Then either $L = V$, or else $V/L$ is $2$-pure.  For 
$a\in W$, the module $(La+K)/K$  is a homomorphic image of $La/Ka$  and hence of $L/K$.  Thus
 $\GKdim((La+K)/K)\leq 1$ and $La\subseteq L$; in other words, $L$
  is still a $(V,W)$-bimodule.    

As $W=\wh{W}$, it is   routine to see that $K$ is $g$-divisible, but since we use the argument several times we give the details.
So, suppose that $\theta g\in K$ for some $\theta\in V$. Then $(V\theta) g\subseteq \wh{W}=W$, whence $V\theta\subseteq  W$ and $\theta\in K$, as required.
It follows that $L/K$ is $g$-torsionfree and so,  by Lemma~\ref{lem:constant},  $L/K$ is a finitely generated 
right $\kk [g]$-module. Thus, by  \cite[Lemma~5.3]{KL},   $I= \ell\text{-ann}_V(L/K)$ satisfies $\GKdim_V(V/I) = \GKdim(L/K)  \leq 1$.  Again, $I$ is $g$-divisible.  Also, if $\theta\in V$ has $\theta g \in L$
  then $(I \theta)g \subseteq  K$ and so $I \theta\subseteq K$.  
Hence $\GKdim(V\theta+K)/K  \leq \GKdim(V/I)\leq 1$ 
  and $\theta\in L$.  So $L$ is also $g$-divisible. 

Finally,  let $X = \idealiser_V(L) = \{ x \in V | Lx \subseteq L \}$.  As usual,   $X$ is $g$-divisible.  
Clearly $IL$ is an ideal of $X$, and since $I$ and $L$ are $g$-divisible,  
$\GKdim(X/IL)\leq 1$,   by Lemma~\ref{GK-results}(4).  Since $X\supset K \supseteq IL$, 
it follows that $\GKdim X/K\leq 1$.
Finally, since $X/K$ is $g$-torsionfree of GK-dimension $1$, it must be a finitely generated $\kk [g]$-module 
by Lemma~\ref{lem:constant}; in particular, $X/W$ and hence $X$ are finitely generated as right $W$-modules.
\end{proof}

  %%%%%%%%%%%%%%%%%%%%%%%%%%%%%%
  %%%%%%%%%%%%%%%%%%%%%%%%%%%%

There is a close correspondence between   subalgebras $A$  of 
the  function skewfield $D = D_{\gr}(T)$   and 
$g$-divisible subalgebras of $T_{(g)}$ and so 
we  end the section by  studying the consequences of our earlier results for   such an algebra $A$.
  
For  a cg subalgebra $R\subseteq T_{(g)}$ with $g\in R$,   define \label{circ-defn}
\[R^\circ = R[g^{-1}]_0 = \bigcup_{n \geq 0} R_ng^{-n} \subseteq D = D_{gr}(T).\]
 Conversely,  given an algebra $A\subseteq  T^\circ$,  define \label{Omega-defn}
\[   \Omega A =\bigoplus_{m\geq 0} (\Omega A)_m
\quad \text{for}\quad   (\Omega A)_m = \{a\in T_m : ag^{-m}\in A\}.\] 
Clearly $\Omega A$ is $g$-divisible with $(\Omega A)^\circ=A$ and, 
if $R\subseteq T$, then $\Omega(R^\circ)= \wh{R}$; thus we obtain 
 a (1-1) correspondence between cg $g$-divisible subalgebras of $T$ and 
  subalgebras of $T^{\circ}$.
  
Given a left ideal $I$ of $R$ or a left ideal $J$ of $A$ we define $I^\circ$ and $\Omega J$ by the same  formul\ae.  
If $R$ is $g$-divisible, the map $I \mapsto I^{\circ}$  gives a (1-1) correspondence between $g$-divisible left ideals of $R$ and 
left ideals of $R^\circ$, with analogous results for two-sided ideals (see \cite[Proposition~7.5]{ATV2}).

An algebra $A \subseteq T^\circ$ is filtered  by 
 $ A=\bigcup \Gamma^nA$   for $  \Gamma^nA= (\Omega A)_ng^{-n}$.
  By \cite[Lemmas~2.1 and  2.2]{RSS},  
\begin{equation}
\label{super-eq}
\gr_\Gamma A \ = \ \bigoplus \Gamma^nA/\Gamma^{n-1}A \ \cong\ \Omega A/g\Omega A,
\end{equation}
 where the isomorphism  is induced by 
 $x = rg^{-n} \in \Gamma^nA \smallsetminus \Gamma^{n-1}A\mapsto r \in \Omega A$.

 \begin{lemma}\label{lem:local}
 Let $A, A'$ be orders in $T^\circ$.  Then $A$ and $A'$ are equivalent orders if and only
  if $\Omega A$ and $\Omega A'$ are equivalent orders in $Q_{gr}(T)$.
 \end{lemma}
 \begin{proof}
 Let  $0\not=a \in \Gamma_m A'$ and $0\not=b \in \Gamma_n A'$.  
To prove the lemma, it suffices to show that 
  $a A b \subseteq A'$  if and only if    $a g^m (\Omega A) b g^n \subseteq \Omega A'$.  
 However, if $0\not=\alpha\in \Omega A$, write $\alpha = x g^k$ for some $k$ and  $x\in A$.  Then 
 \[\text{$axb \in A'$ $\iff$  $axb \in \Gamma_{m+n+k}A'$  $\iff$ $ag^m (xg^k) bg^n \in \Omega A'$,}\] as desired.    \end{proof}

 \begin{corollary}\label{local1}
    A subalgebra $A \subseteq T$ is a   maximal $T^\circ$-order if and only if $\Omega(A)$ is a 
     maximal $T$-order. \qed
   \end{corollary}

By \cite[Theorem~1.1]{RSS}, every subalgebra of $T^{\circ}$ is finitely generated and noetherian;
 these subalgebras  thus give a rich supply of noetherian domains of GK-dimension  2.   
 Our earlier results about cg maximal $T$-orders translate easily to results about 
  about  maximal $T^\circ$-orders.
 An ideal $I$ of a $\kk$-algebra $A$ is called \emph{cofinite} if $\dim_{\kk}(A/I)<\infty$. 

\begin{corollary}\label{localthm}    Let $A $ be a subalgebra of  $T^\circ$ with $Q(A)=Q(T^\circ )$. 
\begin{enumerate}
\item There exists a  maximal order pair $(V,F)$, where
     $F$ is a blowup of $T$ at  some  virtually effective divisor 
     $\divx$, such that $A$ is contained in and equivalent to 
the maximal $T$-order $V^\circ$.   
\item In Part (1),  $F^\circ$ is a maximal order in $Q(T^\circ)=D_{\gr}(T)$.
 \item The algebras   $V^\circ$ and $F^\circ$ have a cofinite ideal $K^{\circ}$  in common. 
Also,   $\gr_\Gamma V \ppe B(E, \sM(-\divx), \tau)$.
 \item Suppose that  all nonzero ideals $I$ of $T(\divd)^\circ$ generate cofinite right ideals of $T^\circ$ (in particular, this happens  if $T(\divd)^\circ$ is simple)
and that   $A$ is a maximal $T^\circ$-order. Then $A$ is a maximal order. 
 \end{enumerate}
\end{corollary}

\begin{proof} 
(1)  By  Theorem~\ref{thm:converse}(2), $\Omega A$ is contained in and equivalent to some such $V$. 
Now use  Lemma~\ref{lem:local} and Corollary~\ref{local1}. 

(2)  Since $F$ need not be contained in $T$, this does not   follow directly from the above discussion. However, it does follow from 
Lemma~\ref{cozzens} combined with the fact that, in the notation of Corollary~\ref{correct3},
\[F^\circ =\End_{T(\divd)^\circ}((M^{**})^\circ)
=\End_{T(\divd)^\circ}((M^{\circ})^{**}).\]

(3) By definition and Lemma~\ref{lem:constant}, $V$ and $F$ have an ideal $K$ in common such that $F/K$ is 
finitely generated as a $\kk[g]$-module. Consequently $F^\circ/K^\circ$ and $V^\circ/K^\circ$ are finite dimensional.
The final assertion  follows from Theorem~\ref{thm:converse}(1c).

(4) Use Corollary~\ref{correct3}(3).
 \end{proof}

  %%%%%%%%%%%%%%%%%%%%%%%%%%%%%%%%%%%%

We also have a converse to Corollary~\ref{localthm}(3).

\begin{corollary}\label{localcor}
Let $\divx$ be a virtually effective divisor on $E$ with $\deg \divx < \mu$.  
Then there exists a  maximal $T^\circ$-order $A$ with $\gr_\Gamma A \ppe B(E, \sM(-\divx), \tau)$.
\end{corollary}

\begin{proof}
Let $U$ be the $g$-divisible maximal $T$-order given by Theorem~\ref{thm:converse}(3); thus  
  $\bbar{U} \ppe B(E, \sM(-\divx), \tau)$ by Part (1c) of that result. 
 By \eqref{super-eq}  $A= U^\circ$ satisfies the conclusion of this corollary.
\end{proof}

  \subsection*{Del Pezzo surfaces}   
The blowup of $T$ at  $\leq 8$ points on $E$ can  be thought of as a 
  noncommutative   del Pezzo surface.   
  More carefully, it should be thought of as the anticanonical ring of a noncommutative del Pezzo surface; 
  this corresponds to the fact that the central element $g$ is in degree 1.
  Let $U$ be a blowup of $T$ at a virtually effective divisor $\divd'$ of degree $\leq 8$. 
  By analogy, we should think of $U$  as  a (new type of) noncommutative del Pezzo surface, and 
  the localisation $U^\circ $ as  a kind of noncommutative affine del Pezzo surface.  
 Corollary~\ref{localthm}(3) can then  be reinterpreted as saying that,    any  maximal  order $A\subseteq T^\circ$
  is the coordinate ring of just such a noncommutative affine del Pezzo surface.
       
    We caution that these noncommutative del Pezzos are not the algebras discussed in \cite{EG}.  
   For example, we consider $A = T^\circ \cong T/(g-1)$, but the algebra $A' = S/(g-1)$ is considered in \cite{EG}.
   The algebra  $A'$ is a rank 3 $A$-module, so ``$\Spec A'$ '' is a triple cover of ``$\Spec A$'' (inasmuch as these
    terms made sense in a noncommutative context).
 
 %%%%%%%%%%%%%%%%%%%
 %%%%%%%%%%%%%%%%%%%%
 
 \section{ \Cogs\ and \texorpdfstring{$g$}{LG}-divisible hulls}\label{MAX1}  

 One of the main results in \cite{Rog09} showed  that 
the algebras considered there have minimal  \cogs,  in a sense we define momentarily.
In this section we show that, under minor assumptions,   this generalises to    cg   subalgebras
 $U\subseteq T$ with $g\in U$  (see Corollary~\ref{max30} for the precise statement). 
  The significance of this result is that it provides a tight connection between 
 the algebra $U$ and its $g$-divisible hull $\wh{U}$  and provides the final step in  the proof of 
 Theorem~\ref{mainthm-intro1}, that maximal orders are noetherian blowups of $T$
  (see Theorem~\ref{max32}).
 
 Recall that a graded ideal $I$ 
of a cg graded algebra $R$ is called sporadic if $\GKdim(R/I) = 1$.

  \begin{definition}  \label{cog-defn2}
An ideal $I$ of a  cg algebra $R$ is called a  \emph{minimal \cog} \label{min-cog}
  if   $\GKdim(R/I) \leq 1$ and,  for all  \cogs\ $J$, we have 
$\dim_{\kk} I/(J \cap I) < \infty$.
\end{definition}
\noindent  Note that one can make the minimal \cog\ $I$ unique by demanding that it be saturated, but we will not do so since this causes extra complications.

Beginning in this section, we need to strengthen our hypothesis on the ring $T$.

\begin{assumption}\label{ass:T2}   In addition to Assumption~\ref{ass:T}, we assume that $T$ has a minimal \cog\ and that there exists
an  uncountable algebraically closed field extension $K\supseteq \kk$ such that,  in the notation of \cite[Definition~7.2]{RSS2},
$\Div (T\otimes_{\kk}K)$ is   countable.    
 
  We emphasise that, by 
%\cite[Proposition~7.5 and %Corollary~7.9]{ATV2} and 
\cite[Theorem~8.8 and Proposition~8.7]{RSS2}, these extra assumptions do hold both for the   algebras $T$  from Examples~\ref{Skl-ex}(1,2)
and for their  their blowups $T(\divd)$ at effective divisors $\divd$  with 
$\deg \divd < \mu$. 
\end{assumption}
\noindent

For the rest of this section we assume that our algebras $T$ satisfy Assumptions~\ref{ass:T} and \ref{ass:T2}.
We do not know if Assumption~\ref{ass:T2} holds for Stephenson's algebras from
Example~\ref{Skl-ex}(3).  By a routine exercise, 
Example~\ref{Skl-ex}(4) does not have a minimal \cog, so Assumption~\ref{ass:T2} is strictly stronger 
than Assumption~\ref{ass:T}.
%In fact, we know of no example of a cg algebra  satisfying Assumption~\ref{ass:T} which does not also  satisfy 
%Assumption~\ref{ass:T2}.  In particular, 

As noted above,  the blowups $T(\divd)$ with 
$\deg \divd < \mu$  have a minimal \cog, and the first goal of this section  is to extend this 
  to more  general subalgebras of $T_{(g)}$.  We start with the case of   $g$-divisible algebras.  
\begin{lemma}\label{thou10}   
Let $(V,F)$ be a maximal order pair, in the sense of Definition~\ref{max-pair-defn}.
Then both $F$ and   $V$ have a minimal  \cog.  
\end{lemma}

\begin{proof}  By  Corollary ~\ref{correct3},  there exists an
 effective divisor $\divd$  with $\deg\divd<\mu$ and a right  $R$-module $M\supseteq R$, where $R=T(\divd)$,
 such that   
\[F=\End_{R}(M^{**}) \supseteq F\cap T =  V=\End_{R}(M).\] 
We will use a minimal \cog\ of $R$  to construct such an ideal for   $F$ and for $V$.

Set $J=M^*=M^{***}\subseteq R$; thus $F=\End_{R}(J)$ as well. 
Also, write  $X=JJ^*$, a nonzero ideal of $R$ and $W=J^*J$,  a nonzero ideal of $F$. 
 By Lemma~\ref{lem:finehat}(3),  $J$ and $ J^*=M^{**}$ are $g$-divisible; in 
particular $J\nsubseteq gT_{(g)}$ and $J^* \nsubseteq gT_{(g)}$.  
 Thus,  by Lemma~\ref{GK-results}(4),   $\GKdim(R/X) \leq 1$ and $\GKdim(F/W) \leq 1$.    
By Assumption~\ref{ass:T2} and \cite[Proposition~8.7]{RSS2} we can choose a minimal \cog\ $X'$ of $R$ such that $X'\subseteq X$.
 Let $I  = J^* X' J$.  Since $\GKdim(X') \leq 1$ and $R$ is $g$-divisible, $\GKdim R/gR  = 2$
and so $X' \nsubseteq gT_{(g)}$ also.  Thus $I$ is an ideal of $F$ with  $\GKdim (F/I)\leq 1$ by Lemma~\ref{GK-results}(4).

Now consider an arbitrary \cog\  $L$ of $F$, if such an ideal exists.   Since $F$ is $g$-divisible, $L \nsubseteq gT_{(g)}$
 and so, just as in the previous paragraph,  $JLJ^*$ is an ideal of $R$ satisfying $\GKdim_R(R/JLJ^*)\leq 1$.
Hence $JLJ^*\supseteq X'H$, for an ideal $H$ of $R$ with 
$\dim_{\kk}(R/H)<\infty.$   Now, $L\supseteq (J^*J)L(J^*J)\supseteq J^*X'HJ$
 and \cite[Proposition~5.6]{KL} implies that $\dim_{\kk}(J^*X'J)/(J^*X'HJ)<\infty$.
Thus $I$ is a minimal \cog\  of $F$.

Finally,  $F$ and $V$ have a common ideal $K$ with $\GKdim(F/K)\leq 1$ (see  Proposition~\ref{correct25}).
Thus $KIK$ is   a minimal \cog\ for $F$ that lies in $V$ and so it is also  a minimal \cog\ for $V$.
  \end{proof}

\begin{proposition}\label{thou100}  Suppose that $T$  satisfies Assumptions~\ref{ass:T} and \ref{ass:T2}.
Let $U \subseteq T$ be a $g$-divisible graded  algebra with  
$Q_{\gr}(U) = Q_{\gr}(T)$. Then  $U$ has a minimal \cog. 
\end{proposition}

 \begin{proof}  
By Theorem~\ref{thm:converse}(2),  $U$ is  contained in and 
equivalent to some $g$-divisible maximal $T$-order $V$,  say with    $aVb\subseteq U$ for 
some nonzero homogeneous $a,b\in U$.  Set $U' = U + UaV \subseteq V$, and $W=\wh{U'}$.
Thus $aV \subseteq U'\subseteq W$ and $U'b \subseteq U$.
 By Lemma~\ref{lem:finehat}(1) $Wb = \wh{U'b}\subseteq \wh{U}=U.$ 
 Set $J  = \lann_W V/W$, noticing that  $J$ is a nonzero ideal of $W$ since $a\in J$
 and a right ideal of $V$. Also,  
   as $W$ is $g$-divisible, it follows that $J$ is $g$-divisible.  
Thus, by Lemma~\ref{GK-results}(3),  $\GKdim W/J\leq 1$.

If  $K$ is a minimal \cog\ in $V$ given by Lemma~\ref{thou10} we claim that $JK$ is a minimal \cog\ in $W$.
To see this,  let $L$ be any ideal of $W$ with $\GKdim  W/L\leq 1$.  Then  $I=VLJ$ is an ideal 
 of $V$.  Since none of $V$, $L$, or $J$ is contained in $g T_{(g)}$, $\GKdim V/I \leq 1$ 
by Lemma~\ref{GK-results}(4).  Hence $I \supseteq KM$ for some ideal $M$ of $V$ 
with $\dim_{\kk} (V/ M) < \infty$ and so  $L \supseteq JVLJ \supseteq JKM$.  This implies that $JK$ 
is a minimal \cog\  for $W$.
Finally a symmetric argument, using the fact that $W$ is $g$-divisible with  a minimal \cog,  proves
that $U$ has such an ideal.    
\end{proof}

 As in Section~\ref{EXISTENCE}, results on $g$-divisible rings have close analogues for subalgebras of $T^{\circ}$.
 
 \begin{corollary}\label{A-cog}  Suppose that $T$  satisfies Assumptions~\ref{ass:T} and \ref{ass:T2}.
  Let $A$ be a subalgebra of $T^\circ$ with $Q(A)=Q(T^\circ)$. Then $A$ has a unique minimal nonzero ideal $I$, and  
   $\dim_{\kk}A/I<\infty$.  Further, $A$ has DCC on ideals and finitely many primes. 
   \end{corollary}
 
 \begin{proof}  Recall from Section~\ref{EXISTENCE} that there is a (1-1)  correspondence  between 
$g$-divisible ideals of $\Omega A$ and ideals of $A$.  Since every nonzero $g$-divisible ideal of $\Omega A$ 
is sporadic,  when combined with Proposition~\ref{thou100} this gives the existence of $I$ as described. 
Since $A/I$ is artinian it has finitely many prime ideals and  DCC on ideals.
Thus the same   holds for $A$.
 \end{proof}
 
    We now turn to a more general  subalgebra $U$ of $T$, with the aim of controlling its \cogs\ also.    
We achieve this by relating $U$ to its  \emph{$g$-divisible hull}  $\wh{U}$ and 
   we  begin with a straightforward   lemma on subalgebras of TCRs.
Recall that, for any subalgebra $U\subseteq T_{(g)}$,  we write 
   $\overline{U} = U+gT_{(g)}/T_{(g)}$.    

    \begin{lemma}\label{lem-x}
Let $B = B(E, \mc{M}, \tau)$ for some smooth elliptic curve  $E$, invertible sheaf $\mc{M}$ of degree $d > 0$ and $\tau$ of infinite order.  
 Then for any $0 \neq x \in B_k$, 
we have $B_n x + x B_n = B_{n+k}$ for $n \gg 0$.

In particular, if $A$ is a graded subalgebra of  $B$ such that $A\not=\kk$, then $B$ is a noetherian
$(A, A)$-bimodule.
\end{lemma}
\begin{proof}  By \cite[Theorem~1.3]{AV} 
and its left-right analogue, there exist effective
divisors $\divx$ and $\divx'$ such that 
\[x B_{\geq n_0} =  \bigoplus_{n \geq n_0}   H^0(E, \sM_{n+k}(-\divx)) \quad\text{and}\quad
(B_{\geq n_0}) x = \bigoplus_{n \geq n_0} H^0(E, \sM_{n+k}(- \tau^{-n} \divx'))\]
(With a little thought one can see  that this holds with $n_0=0$ and $\divx=\divx'$, but that is not relevant here.)
 Since  $|\tau|=\infty$, we may choose $n_0 $ so that $\divx \cap \tau^{-n} \divx' = \emptyset$ for all $n \geq n_0$.
  For such $n$  there is  an exact sequence:
\[ 0 \to \sO_E(-\divx - \tau^{-n} \divx') \to \sO_E(-\divx) \oplus \sO_E( - \tau^{-n} \divx') \to \sO_E \to 0.\]
Tensoring with $\sM_{n+k}$ and taking global sections gives a long exact sequence that reads in part:
\[ 
\xymatrix{
H^0(E, \sM_{n+k}(-\divx)) \oplus H^0(E, \sM_{n+k}(-\tau^{-n}\divx')) \ar[r] \ar@{=}[d] 
& H^0(E, \sM_{n+k}) \ar[r] \ar@{=}[d]
& H \\
x B_n \oplus B_n x \ar[r]^{\theta} & B_{n+k}
}\]
for $H=H^1(E, \sM_{n+k}(-\divx - \tau^{-n} \divx'))$ and $\theta $ the natural map.
Since $\deg ( \sM_{n+k}( - \divx - \tau^{-n} \divx))   > 0$ for $n\gg 0$,  Riemann-Roch ensures that 
$H=0$ and hence that 
 $\theta$ is  surjective for such $n$.
 
 This  implies that $B$ is a noetherian $(\kk \langle x \rangle, \kk \langle x \rangle)$-bimodule, 
 which certainly suffices to prove the final assertion of the lemma.
 \end{proof}

 We now show that, under mild hypotheses,    $\wh{U}$ is equivalent to $U$. In this result
  the hypothesis that $\overline{U}\not=\kk$ is annoying but necessary (see Example~\ref{max6}) 
  but, as will be shown in Section~\ref{ARBITRARY}, there are ways of circumventing it.
 
\begin{proposition}\label{max2}  
Suppose that $T$  satisfies Assumptions~\ref{ass:T} and \ref{ass:T2}.
Let $U$ be a cg subalgebra of $T$ with $Q_{gr}(U)=Q_{gr}(T)$, \  $g \in U$ and $\overline{U}   \neq \kk$.

\begin{enumerate}
\item There exists $n \geq 0$ such that $U \cap Tg^m = \wh{U} \cap Tg^m= g^m\wh{U}$
for all $m\geq n$.  Thus $U$ and $\wh{U}$
 are equivalent orders.

\item If $U$ is right  noetherian then $\wh{U}$ is a finitely generated right $U$-module.
\end{enumerate}
\end{proposition}

\begin{proof}  (1)    Let $V = \wh{U}$.  Since $T$ is $g$-divisible, $V \subseteq T$.  Working inside $Q_{gr}(T)$ we get 
\[
\{x \in T | xg^k \in U \}  = g^{-k}U \cap T,
\]
and hence  $V = \bigcup_{k \geq 0} g^{-k}U \cap T$.
Now define $Q^{(k)} = \left(g^{-k}U\cap T + gT\right)/gT \subseteq  \overline{T}$.
Then since $g\in U$,  
$$\overline{U} = Q^{(0)} \subseteq Q^{(1)} \subseteq \dots\subseteq \bigcup_k Q^{(k)} = \overline{V}.$$ 
Each $Q^{(i)}$ is an $\bbar{U}$-sub-bimodule of $\overline{T}$ and so, by Lemma~\ref{lem-x},
  $Q^{(n)} = \overline{V}$ for some $n$.

We claim that $U \cap Tg^m = V \cap Tg^m$ for all $m\geq n$. If not, 
there exists   $y = xg^m \in V \cap Tg^m \smallsetminus U$ for some such $m$.  
Choose  $x$ of minimal degree with this property. This  ensures that $y\not\in g^{m+1}T$, since 
  otherwise one could write $y=g^{m+1}x'$ with $\deg(x')=\deg(x)-1$. 
    Since $\overline{x} =[x+gT] \in \overline{V} = Q^{(n)}$, 
we have $\overline{x} = \overline{w}$ where $wg^n \in U$.  Thus $wg^n - xg^n \in V \cap Tg^{n+1}$ 
and so $w-x = vg$ where 
$vg^{n+1} \in V \cap Tg^{n+1}$.  Since  $\deg v<\deg x$, the minimality of $\deg x$ ensures that
 $v g^{n+1} \in U$.  
Then $xg^n = wg^n - vg^{n+1} \in U$, and so $y = xg^n(g^{m-n}) \in U$, a contradiction.   
Thus $U \cap Tg^m = V \cap Tg^m$ as claimed. 
Finally, as $gV=V\cap gT$, an easy induction shows that $ V \cap Tg^m=g^mV$.

(2) This is immediate from Part~(1).
\end{proof}
 
In the next result, we  construct an ideal with a   property that is slightly weaker than being 
a minimal \cog. However, it will have the same consequences. 

 \begin{corollary}\label{max30}  Suppose that $T$  satisfies Assumptions~\ref{ass:T} and \ref{ass:T2}.
Let $C$ be a  cg subalgebra of $T$ with $Q_{gr}(C)=Q_{gr}(T)$.
Assume that  $g\in C$  and $\overline{C}\not=\kk$.
Then   $C$   has a \cog\ $K$ (possibly $K=C$) that is minimal among
  \cogs\ $I$ for which  $C/I$ is $g$-torsionfree. 
\end{corollary}

\begin{proof}   
Note that 
   $\wh{C}$ is  noetherian by Proposition~\ref{div-noeth} and has a minimal \cog, say $J$, by Proposition~\ref{thou100}.  
     By Lemma~\ref{GK-results}(2), $\wh{J}$ is also a
   minimal \cog\ of $\wh{C}$.   Thus,  replacing $J$ by $\wh{J}$,  we can assume that $\wh{C}/J$ is $g$-torsionfree.

We will show  that $K = J\cap C$ satisfies the conclusion of  the corollary.  So, let $I$ be  a \cog\ of $C$ 
such that $C/I$ is $g$-torsionfree (if such an ideal exists). We first show that $J \cap C \subseteq I$. 
By  Proposition~\ref{max2}, $H=g^n\wh{C}\subseteq C$ for some $n\geq 1$
and so  $I \supseteq HIH =g^{2n}\wh{C}I\wh{C}$.  By Lemma~\ref{GK-results}(3),
  $HIH=g^rL$ for some $r$ and ideal    $L$ of $\wh{C}$ with $\GKdim(\wh{C}/L)\leq 1$.
As $J$ is sporadic,  $\dim_{\kk}J/(J\cap L)<\infty$   and so 
   $L\cap J\supseteq J_{\geq s} \supseteq  g^sJ$ for some integer $s$. 
  Combining these observations   shows that $I\supseteq g^tJ$ for some integer $t$.   
  Pick $u$ minimal such that $I\supseteq g^u(J\cap C)$. If $u\not=0$ then 
   $$\frac{I+g^{u-1}(J\cap C)}{I} \ = \  \frac{I+g^{u-1}(J\cap C)}{I+g^{u}(J\cap C)} $$
   is   $g$-torsion, and hence zero since $C/I$ is $g$-torsionfree by assumption.  Hence $u=0$ and 
   $I\supseteq J\cap C$.  
  
It remains to show that $\GKdim C/(C \cap J) \leq 1$.  Since $ \wh{C}/J$ is $g$-torsionfree,   Lemma~\ref{lem:constant} 
implies that $M=\wh{C}/J$ is a finitely generated $\kk[g]$-module.  
Then the $C$-submodule $(C + J)/J \cong C/(J \cap C)$ is also.
  Therefore, by \cite[Corollary~5.4]{KL}, 
  $$ \GKdim_C(C/(C\cap J)) = \GKdim_{\kk[g]}(C/(C\cap J))\leq 1.$$
Thus $K = J \cap C$ satisfies the conclusions of the corollary.
\end{proof}
 
\begin{lemma}\label{lem:ACC}
The set of orders  $\left\{ \text{$C\subseteq T$ with $\bbar{C} \neq\kk $ and $g\in C$} \right\}$ satisfies ACC.
\end{lemma}

\begin{proof}
By Zorn's Lemma, it suffices to prove that  any such ring $C$ is finitely generated as an algebra; equivalently, 
that $C_{\geq 1}$ is finitely generated as a right ideal.  

We first show that, for any $m \geq 1$, $C/(g^{m}T\cap C)$ is finitely generated as an algebra.
The result holds for $m=1$ by \cite[Theorem~1.1(1)]{RSS}.  
By induction, choose $a_1, \dots a_k \in C_{\geq 1}$ whose  images generate $C/(g^{m-1}T \cap C)$ as an algebra.  
    Set $X=(g^{m-1}T \cap C)/(g^mT \cap C)$. Then, up to shifts, 
    \[X\ \cong \frac{T\cap g^{1-m}C}{gT\cap g^{1-m}C}  \cong  \frac{(T\cap g^{1-m}C)+gT}{gT}  \subseteq  T/gT,\]
    as $\overline{C}$-bimodules. Thus, by  Lemma~\ref{lem-x}, $X$ is  a finitely generated  $\bbar{C}$-bimodule,
    say by the images of     $b_1, \dots, b_n \in g^{m-1}T \cap C$.
      Then $\{ a_1, \dots, a_k, b_1, \dots, b_n\}$ generate $C / (g^m T \cap C)$, completing the induction.  

By Proposition~\ref{max2}, there exists $\ell \in \mathbb{N}$ such that $g^m \wh{C} = C \cap g^m T \subseteq C$
for all $m\geq \ell$.
By the above, choose  $c_1, \dots, c_N \in C_{\geq 1}$ whose  images   generate $C/g^{\ell+1} \wh{C}$
as an algebra.  
Then for any $f\in C_{\geq 1}$, there exists $x \in \sum c_i C$ so that $f-x \in g^{\ell+1}\wh{C}=g(g^{\ell}\wh{C}) $; thus
$f \in gC + \sum c_i C$.  Therefore, $C_{\geq 1} $ is generated as a right ideal by $g, c_1, \dots, c_N$.
\end{proof}

\begin{proposition}\label{max31}  Suppose that $T$  satisfies Assumptions~\ref{ass:T} and \ref{ass:T2}.
Let $U$ be a  cg subalgebra of $T$ with $\overline{U}\not=\kk$ and   
$D_{gr}(U)=D_{gr}(T)$. Then there exists a nonzero ideal of  $C=U\langle g\rangle $ that is finitely generated as both a 
left and a right $U$-module. 
\end{proposition}

\begin{proof}
By Lemma~\ref{lem:ACC}, there is a finitely generated cg subalgebra $W$ of $U$ with  
$C = U\ang{g}=W\ang{g}$. Note that $Q_{gr}(C)=Q_{gr}(T)$ as $g\in C$.  

 Fix $n\in \mathbb{N}$. Observe that  $CW_{\geq n}  = \sum_m W_{\geq n}g^m 
=W_{\geq n}C$ is an ideal of $C$. Moreover, $C/CW_{\geq n}$ is a homomorphic image of the polynomial ring
$\left( W/W_{\geq n}\right)[g]$. Since $\dim_{\kk}(W/W_{\geq n})<\infty$, it follows that 
$C/CW_{\geq n}$ is a finitely generated $\kk[g]$-module. In particular,
by \cite[Corollary~5.4]{KL},
 \[\GKdim_C(C/CW_{\geq n}) =\GKdim_{\kk[g]}(C/CW_{\geq n})\leq 1.\]
Moreover,  $K_n=\tors_g(C/CW_{\geq n})$ is finite dimensional.

Let $Z_n= C\cap\wh{CW_{\geq n}}$; thus    $Z_n/CW_{\geq n} =K_n$. 
Note that   $C/Z_n$ is a finitely generated torsion-free, hence free, $\kk[g]$-module. Therefore, 
if $d_n$ denotes the rank of  that free module, then  
\[\begin{array}{rll}
d_n \ = &  \dim_{\kk}\bigl((C/Z_n)_m\bigr)   &\quad\text{for $m\gg 0$}
 \\  \noalign{\vskip 5pt}
= &  \dim_{\kk}\ \bigl((C/CW_{\geq n})_m\bigr) \   &\quad\text{for $m\gg 0$}. 
\end{array}\]
 Also,  $CW_{\geq n} \supseteq CW_{\geq n+1}$, whence $Z_n \supseteq Z_{n+1}$ and $d_n\leq d_{n+1}$.

Let $J$ be a minimal \cog\ of $\wh{C}$ such that $\wh{C}/J$ is $g$-torsionfree; 
thus, by Corollary~\ref{max30} and its proof, 
$C\cap J$ is minimal among \cogs\ of $C$ such that the factor is $g$-torsionfree.
By construction, each $Z_n$ is either sporadic or equal to $C$; in either case $Z_n\supseteq C \cap J$. 
Now   $C/(C\cap J)\hookrightarrow \wh{C}/J$ which,  by Lemma~\ref{lem:constant},  has an eventually constant 
Hilbert series;  say $\dim_{\kk}(\wh{C}/J)_m = N$ for all $m\gg 0$.
  Hence $ \dim (C/Z_n)_m\leq N$ for all such $m$ and, in particular, $d_n\leq N$. Since $d_n\leq d_{n+1}$, it follows that 
 $d_n=d_{n+1}$ for all $n\gg 0$; say for all $n\geq n_0$.
 Thus, by the last display,  $CW_{\geq n} \ppe CW_{\geq n_0}$ for all $n\geq n_0$.
 
 Finally, if  $W$ is generated as an algebra by elements of degree at most $e$, then 
 $CW_{\geq n_0}W_{\geq 1} \supseteq CW_{\geq n_0+e}$.  By   the last paragraph, 
 $\dim_{\kk}\left(CW_{\geq n_0}/CW_{\geq n_0+e}\right)<\infty$,
 and so  $\dim_{\kk}\left(CW_{\geq n_0}/CW_{\geq n_0}W_{\geq 1}\right)<\infty$. Thus, by the graded Nakayama's lemma,
 $CW_{\geq n_0}=W_{\geq n_0}C =W_{\geq n_0}U\ang{g}$ is finitely generated as a right $W$-module, 
and hence  as a right $U$-module.
\end{proof}

Finally, we can reap the benefits of the last few results.

 \begin{theorem}\label{max32} Suppose that $T$  satisfies Assumptions~\ref{ass:T} and \ref{ass:T2}.
 For some $n\geq 1$, 
let $U$ be a cg maximal $T^{(n)}$-order  with $\overline{U}\not=\kk$.
 Then $U$ is   strongly noetherian; in particular, noetherian and finitely generated as an algebra. Moreover:
 \begin{enumerate}
\item If $n=1$, so $Q_{gr}(U)=Q_{gr}(T)$,  then $U$ is $g$-divisible and $U = F \cap T$, where $F$ is a blowup of $T$ at a 
virtually effective divisor $\divx=\divu-\divv+\tau^{-1}(\divv)$ of degree $< \mu$. 
 
\item If $Q_{gr}(U)\not=Q_{gr}(T)$, then   there is a virtually effective divisor $\divx$ of degree $< \mu$ and a blowup $F $ of $T$ at $\divx$ so that
$U = (F \cap T)^{(n)}$.
\end{enumerate}
\end{theorem}

\begin{proof}  (1)  Let $C=U\langle g\rangle$; thus  $Q_{gr}(U) = Q_{gr}(\wh{C})=Q_{gr}(T)$. 
By Proposition~\ref{max31}, there exists an ideal  $X $ of  $C$ that 
 is finitely generated as a right $U$-module.    In particular, as $U$ is a right Ore domain 
 and $X\subseteq Q_{gr}(U)$, we can clear denominators from the left to find $q\in Q_{gr}(U)$ 
 such that $X\subseteq qU$. As $X$ is an ideal of $C$, we have $pC\subseteq X$ for any $0\not=p\in X$ and hence 
 $C\subseteq p^{-1}qU$.  Thus $C$ and $U$ are equivalent orders. By Proposition~\ref{max2} it follows that 
 $U$ and $\wh{C}$ are equivalent orders and hence $U=\wh{C}$.  
Now apply  Proposition~\ref{div-noeth} and Theorem~\ref{thm:converse}.

(2) Keep $C$ and $X$ as above.
  In this case, as $Q_{gr}(U) = \kk(E)[g^{n},g^{-n},\tau^n]$, clearly $U$ and $C'=U\langle g^n\rangle$ 
have the same graded quotient ring and, moreover, $C'=C^{(n)}$. Therefore $X^{(n)}$ is an ideal of $C^{(n)}$ which, 
since it is a $U$-module summand of $X$, is also finitely generated as a right (and left) $U$-module.
The argument used in (1) therefore implies that $U$ and $C^{(n)}$ are equivalent orders and hence that $U=C^{(n)}$.
In particular, $C=\sum_{i=0}^{n-1}g^iC^{(n)}$ is a finitely generated right $U$-module.

Consider $\wh{C}$. As $g\in C$, we have $Q_{gr}(\wh{C})=Q_{gr}(T)$ and so,  
  by Corollary~\ref{main-ish}(1),   there exists a cg maximal $T$-order $V=\wh{V}\subseteq T$   containing and 
   equivalent to $\wh{C}$. By Proposition~\ref{max2},  $V$ is   equivalent to $C$.
Further, $V = F \cap T$ where $F$ is a blowup of $T$ at some virtually effective divisor $\divx$ on $E$ with $\deg \divx < \mu$.

Now,  
  $aVb\subseteq C$ for some $a,b\in C\smallsetminus\{0\}$. By multiplying by further elements of $C$ we may suppose that 
$a,b\in C^{(n)}=U$ and hence that  $aV^{(n)}b\subseteq  U$.  As $U$ is a maximal $T^{(n)}$-order, and certainly 
$V^{(n)}\subseteq T^{(n)}$,  it follows that  $U=V^{(n)}$. 
\end{proof}

One consequence of the  theorem is that maximal $T^{(n)}$-orders have a number of pleasant properties, 
as we  next  illustrate.  The undefined terms in the following corollary 
can be found  in \cite[Section~2]{Rog09} and \cite{VdB4}.   
 
\begin{corollary}\label{cohom}  Suppose that $T$  satisfies Assumptions~\ref{ass:T} and \ref{ass:T2}.
For some $n\geq 1$, 
let $U$ be a cg maximal $T^{(n)}$-order  with $\overline{U}\not=\kk$.
Then $\rqgr U$ has  cohomological dimension $\leq 2$, while $U$ has a balanced dualizing complex and 
satisfies the Artin-Zhang  $\chi$ conditions.   
\end{corollary}

\begin{proof} 
By Theorem~\ref{max32},   $U = V^{(n)}$  for  a $g$-divisible maximal $T$-order $V$.
Hence  $\overline{V}\ppe  B(E,\mathcal{N},\tau)$, by  Theorem~\ref{thought9}.
Thus \cite[Lemma~2.2]{Rog09} and \cite[Lemma~8.2(5)]{AZ}  imply that 
$\rqgr \overline{V}$ has cohomological dimension 
one, and that $\overline{V}$ satisfies $\chi$.  The  fact that $V$ satisfies $\chi$ and that $\rqgr V$ 
has cohomological dimension $\leq 2$ 
then follow from \cite[Theorem~8.8]{AZ}. By  \cite[Lemma~4.10(3)]{AS} $V$ is a noetherian $U$-module and  so, by
 \cite[Proposition~8.7(2)]{AZ}, these properties then descend to $U$. 
 (With a little more work one can show that $\rqgr V$ and $\rqgr U$ have  cohomological dimension exactly~$2$.)
 Finally, by  \cite[Theorem~6.3]{VdB4}
this implies the  existence of a balanced dualizing complex.
\end{proof}

Let $U$ be a maximal order in $T$ with $\bbar{U} \neq \kk$.  
Theorem~\ref{max32} also allows us to determine the simple objects in $\rqgr U$, although we do not formalise their geometric structure.  
 
\begin{corollary}\label{points}  Suppose that $T$  satisfies Assumptions~\ref{ass:T} and \ref{ass:T2}.
Let $U$ be a cg maximal   $T$-order  with $\overline{U}\not=\kk$.
Then  the   simple objects in $\rqgr U$ 
are in (1-1) correspondence with the closed points of 
   the elliptic curve $E$ together with a (possibly empty) finite set.   
\end{corollary}

\begin{proof} 
A simple object in $\rqgr U$ equals $\pi(M)$ for  a cyclic critical right $U$-module $M$
with the property that every proper factor of $M$ is finite dimensional. Suppose first that   
$M$ is   $g$-torsion; thus $Mg=0$ by   Lemma~\ref{lem:ann}. Hence, by 
Theorems~\ref{max32}  and \ref{thm:converse}, $\pi(M)\in \rqgr B$, for some TCR $B=B(E,\sN,\tau)$.
Thus, under the equivalence of categories $\rqgr B\simeq \coh(E)$, $\pi(M)$ corresponds to a closed point of $E$.

On the other hand, if   $M$ is not annihilated by $g$, then  Lemma~\ref{lem:ann} implies that
 $M$ is $g$-torsionfree.
   By comparing Hilbert series, it follows that $\GKdim(M/Mg)=\GKdim(M)-1$ and so, as
 $\dim_{\kk}M/Mg<\infty$ by construction,  $\GKdim(M)=1$.  
In particular,   $M'=M[g^{-1}]_0$ 
 is then a  finite dimensional simple $U^\circ$-module and hence
 is annihilated by the minimal nonzero ideal of $U^\circ$ (see Corollary~\ref{A-cog}).
  Pulling back to $U$, this says that $M$ is killed by
 the minimal \cog\ $K$ of $U$. Thus, by Lemma~\ref{lem:ann}, $P=\rann(M)$ 
 is   one of the finitely many prime ideals $P$ minimal over $K$. 
 
 In order to complete the proof we  need to  show that $\pi(M)$ is uniquely determined by $P$. 
 Note  that, as $\dim_{\kk}(M/Mg)<\infty$, 
  we have $\pi(M)\cong\pi(Mg)=\pi(M[-1])$  in $\rqgr U$, 
 and so we do not need to worry about shifts. Next, as $\GKdim(M)=\GKdim(U/P)$,
 $M$ is a (Goldie)  torsion-free $U/P$-module and hence  is isomorphic  to  (a shift of) a 
 uniform right ideal $J$ of  $U/P$. However, given a second uniform right ideal $J'\subseteq U/P$ then 
 $J'$ is isomorphic to (a shift of)  a submodule $L\subseteq J$ (use the proof of  \cite[Corollary~3.3.3]{MR}).
  Once again, $\dim_{\kk}(J/L)<\infty$ and so 
 $\pi(J)\cong \pi(J')$, as required.
 \end{proof}

 \begin{corollary}\label{max33} Suppose that $T$  satisfies Assumptions~\ref{ass:T} and \ref{ass:T2}.
 Let $U\subseteq T$ be a noetherian cg algebra with 
 $D_{gr}(U)=D_{gr}(T)$ and $\overline{U}\not=\kk$. Then $C=U\langle g\rangle $ and $\wh{C}$ are both 
 finitely generated right (and left) $U$-modules. 
 \end{corollary}

\begin{proof} Again let  $X=CU_{\geq n_0}$  be the ideal of $C$ that 
 is finitely generated as a right $U$-module given by Proposition~\ref{max31}. In this case 
 $X$ is a noetherian right $U$-module and hence so is $C\cong xC[n]$, for any $0\not=x\in X_n$. 
 The rest of the result follows from Proposition~\ref{max2}.  
\end{proof}

%%%%%%%%%%%%%%%%%%%%%%%%
%%%%%%%%%%%%%%%%%%%%%%%%%%%%%

 \section{Arbitrary orders}\label{ARBITRARY}

The assumption $\overline{U}\not=\kk$  that appeared in most of the results from Section~\ref{MAX1}  
is annoying but, as Example~\ref{max6}   shows, necessary.
   Fortunately   one can bypass the problem, although at the cost of passing to a Veronese ring.
   In this section we explain the trick and apply it to describe arbitrary cg orders in $T$. 

 Up to now graded homomorphisms of  algebras have been degree-preserving,
 but this will not be the case for the next few results, and so we make the following definition. 
  A homomorphism $A\to B$ between $\mathbb{N}$-graded algebras
  is called \emph{graded of degree $t$} if $\phi(A_n)\subseteq B_{nt}$ for all $n$. The map $\phi$ is called 
  \emph{semi-graded}   \label{semi-graded defn}
 if it is graded   of   degree $t$ for some $t$.

\begin{proposition}\label{AS-trick}
Suppose that $T$ satisfies Assumption~\ref{ass:T} and 
  that $U$ is a cg noetherian subalgebra of $T$ with $U \not\subseteq \kk[g]$.  Then there exist  
  $N, M\in \mathbb{N}$ and a injective  graded homomorphism $\phi: U^{(N)}\to T$ of degree $M$ 
such that $U' = \phi(U^{(N)}) \not\subseteq \kk+gT$.  In addition, $D_{gr}(U) = D_{gr}(U') \subseteq D_{gr}(T)$.
\end{proposition}

 \begin{proof}
For $n\geq 0$, define  $f:\mathbb{N}\to \mathbb{N}\cup\{-\infty\}$ by 
\[\text{$f(n) =   \min \{i : U_n \subseteq g^{n-i} T \}$, \  with  $ f(n) = -\infty$ if $ U_n = 0 $.}\] 
Trivially, $f(n) \in \{0, 1, \dots, n \} \cup \{-\infty\}$  for all $n\geq 0$, and $ f(n) = 0$ if and only if  $U_n = \kk g^n$.

We first claim that  $f(n) + f(m) \leq f(n+m)$ for all $m,n\geq 0$. 
As $A$ is a domain,  this is clear if one of terms  equals $-\infty$, and so  we may assume that 
$f(r)\geq 0$ for $r=n,m,n+m$. 
Write $U_r =X_rg^{r-f(r)}$ for such $r$; thus $X_r\subseteq T$ but $X_r\not\subseteq gT$. 
Since   $gT$ is a completely prime ideal,  $X_nX_m\subseteq T$ but 
$X_nX_m\not\subseteq gT$.
 In other words, $U_nU_m\not\subseteq Y= g^{(n-f(n)+m-f(m)+1)}T$. 
Since $U_nU_m\subseteq U_{n+m}$ it follows that $U_{n+m} \not\subseteq Y$
and hence that  $f(n+m)\geq f(n) + f(m)$, as claimed.

A noetherian cg algebra is finitely generated by the graded Nakayama's Lemma, so 
suppose that $U$ is generated in degrees $\leq r$.  Then $U_n = \sum_{i = 1}^r U_i U_{n-i}$ for 
all $n > r$.  Arguing as in the previous paragraph   shows that 
\begin{equation}\label{AS-tt}
f(n) = \max\{ f(n-i) + f(i) :1\leq i\leq r \}\quad \text{for $n > r$,}
\end{equation}
 with the obvious conventions if any of these 
numbers equals $-\infty$.

We claim that there exists $N$ with $f(N) > 0$ such that $f(nN) = n f(N)$  for all $n \geq 1$.
This follows by exactly the same proof as in \cite[Lemma~2.7]{AS}. 
Namely, choose $1 \leq N \leq r$ such that $\lambda = f(N)/N$ is as large as possible; by induction 
using \eqref{AS-tt} it follows that $f(n) \leq \lambda n$ for all $n \geq 0$, 
and this forces $f(n N) = n f(N)$ for all $n \geq 0$, as claimed.

  Let $M = f(N)$ and note that $M>0$ since  $U \not\subseteq \kk[g] $.   Thus,
  for each $n\geq 0$ we have  $U_{nN}\subseteq g^{nN-nM}T$  but $U_{nN}\not\subseteq g^{nN-nM+1}T$. 
Therefore the  function
$U_{nN} \to T_{nM}$ given by 
$x \mapsto x g^{n(M-N)}$ is well-defined, and it 
defines an injective vector space  homomorphism $\theta : U^{(N)} \to T$ with $\theta(U^{(N)})\not\subseteq \kk+gT$.  
It is routine to see that $\theta$ is an algebra homomorphism which is graded of degree $M$.
The final claim of the proposition  is clear because $D_{gr}(U) = D_{gr}(U^{(N)}) = D_{gr}(U')$.
 \end{proof}

\begin{corollary}\label{max5}
{\rm (1)} Suppose $T$ satisfies Assumption~\ref{ass:T} and that $U$ is a noetherian subring of $T$ generated in a single degree $N$, with $U\not=\kk[g^N]$. 
Then up to a semi-graded isomorphism  we may 
assume that $U\not\subseteq \kk+gT$. 

{\rm (2)}  Suppose also that $T$ satisfies Assumption~\ref{ass:T2}.  If $U$ is a noetherian maximal $T^{(N)}$-order generated in   degree $N$ then, again up to a semi-graded isomorphism,
$U\cong V^{(M)}$ where $(V,F)$ is a maximal order pair and $M \leq N$.
\end{corollary}

\begin{proof}  In this proof, Veronese rings are unregraded; that is, they are given the grading induced from  $T$. 

(1)    Pick $M\in \NN$ minimal such that $U_N \subseteq g^{N-M}T$.
Necessarily, $M \leq N$.  
 Then, either directly or by 
Proposition~\ref{AS-trick},  there is a semi-graded monomorphism  
 $\phi: U=U^{(N)} \to T$ given by $u\mapsto g^{M-N}u$ for $u\in U_N$. Hence  $U\cong \phi(U)$,
 and $\phi(U)\not\subseteq \kk+gT$ by the choice of $M$. 

(2)   As $U$ is an order in $T^{(N)}$,  certainly $\phi(U)$
is an order in $T^{(M)}$. So, suppose that 
  $\phi(U)\subseteq W\subseteq  T^{(M)}$ for some equivalent order $W$; say with $aWb\subseteq \phi(U)$, 
  for $a,b\in \phi(U)$.  Since $M\leq N$,  the map $\phi^{-1}$ extends to give 
   a well-defined semi-graded homomorphism $\psi: T^{(M)}\to T^{(N)}$ 
  defined  by $\gamma\mapsto g^{N-M}\gamma$ for all $\gamma\in T_M$.
 Therefore,  
  $\psi(a) U \psi(b) \subseteq U\subseteq \psi(W)\subseteq T^{(N)}$ and hence  $U=\psi(W)$.
  Thus,  $\phi(U)=W$ is a maximal order in  $T^{(M)}$ with $\phi(U)\not\in \kk+gT$. Now apply 
Theorem~\ref{max32}(2).
\end{proof}

    One question we have been unable to answer is the following.

   \begin{question} Suppose that $U\subseteq T$ is a cg maximal $T$-order or, indeed, a maximal order.
    Then is each Veronese ring $U^{(n)}$ also a maximal $T^{(n)}$-order?
   The question is  open  even when $U$ is noetherian.
 \end{question}
 
  If this question has a positive answer, then one can mimic the proof of Corollary~\ref{max5} for any 
  noetherian  maximal order $U$ to get a precise description of 
  some Veronese ring $U^{(N)}$. However, the best we can do at the moment is 
   to use the    much less precise result given by the next corollary, which 
   also   describes arbitrary noetherian cg subalgebras of $T$.
     
 \begin{corollary}\label{main} Suppose that $T$ satisfies Assumptions~\ref{ass:T} and \ref{ass:T2}.  
 Let $U\subseteq T$ be a noetherian algebra with   
 $D_{gr}(U)=D_{gr}(T)$. Then, up to taking Veronese subrings, $U$ is a iterated subidealiser inside a virtual blow-up of~$T$. 
 More precisely, the following holds. 
 \begin{enumerate}
 \item There is a semi-graded  isomorphism of Veronese rings $U^{(N)}\cong U'$, where $U'\subseteq T$ is a 
 noetherian algebra such that  $D_{gr}(U')=D_{gr}(T)$  and $U'\not\subseteq \kk+gT$.   
 
 \item If $C=U'\langle g\rangle $ and $Z=\wh{C}$, then $Z$ is a finitely generated (left and right) $U'$-module and 
 $Z$ is a noetherian   algebra with $Q_{gr}(Z)=Q_{gr}(T)$.
 The $g$-divisible  algebra $Z$ is described by Corollary~\ref{main-ish}.
  \end{enumerate}
 \end{corollary}
    
  \begin{proof} 
  By \cite[Proposition~5.10]{AZ}, the  Veronese ring $U^{(N)}$ is noetherian and so 
Part~(1)  follows from  Proposition~\ref{AS-trick}. 
Part~(2) then follows from  Corollary~\ref{max33} (and  Corollary~\ref{main-ish}).
\end{proof}

  %%%%%%%%%%%%%%%%%%%
 %%%%%%%%%%%%%%%%%%%%
   
\section{Examples}\label{EXAMPLES} 
 
We end the paper with several examples that illustrate some of the subtleties of the paper.
For simplicity, these examples will all be constructed from $T=S^{(3)}$ for the standard Sklyanin algebra $S$ of 
Example~\ref{Skl-ex}(1); thus $\mu = \deg \sM=9$.

We first  construct  a   $g$-divisible,  maximal $T$-order $U$ 
 that is not a maximal order in $Q_{gr}(U)$, as promised in Section~\ref{MAX-SECT}.
  This shows, in particular, that  the concept of maximal order pairs is indeed  necessary in that section.
 In order to construct the example, we need the following notation.

 \begin{notation}\label{Tmax-notation}  Fix $0\not=x\in S_1$ and let $\divc=p+q+r$ be the hyperplane section of $E$ 
 where $x$ vanishes. We   can and will assume that  no two of $p,q,r$ lie on the same $\sigma$-orbit on $E$, where 
 $S/gS\cong B(E,\mathcal{L},\sigma)$.
 Set $R=T(\divc)$.   By \cite[Example~11.3]{Rog09} 
$R$ has a \cog\ $I=xS_2R$.   Write  $N=xT_1x^{-1}R$ and 
$M=xS_5R+R$. Finally, set $\divd=\sigma^{-2}(\divc) = \sigma^{-2}(p)+\sigma^{-2}(q)+\sigma^{-2}(r)$ and hence $\divd^\tau= 
\sigma^{-5}(\divc).$ 
\end{notation}

As we will see, $U=\End_R(M)$ will (essentially) be the required  maximal $T$-order 
with equivalent maximal order being $F=\End_R(N) $.
The proof will require some detailed computations, which form  the content of the next lemma.
We note that for subspaces of homogeneous pieces of $S$ we use the grading on $S$, but for subspaces 
that live naturally in $T$ we use the $T$-grading.  For example, we write $T_1 S_2 = S_5$.
   
\begin{lemma}\label{Tmax-eg1}  Keep the data  from Notation~\ref{Tmax-notation}.
\begin{enumerate}
\item $NI=xS_5R\subseteq M$ and $ M_{\geq 1} \subseteq N $. Hence $N^{**} =M^{**}=(\widehat{M})^{**}=\wh{M^{**}}$.
\item $U'=\End_R(\wh{M})\subseteq T$ but
\item  $F=\End_R(M^{**})=  \End_R(NI) = x T(\divd^{\tau})x^{-1}$. Moreover, $F\not\subseteq T$.
\end{enumerate}
\end{lemma}

\begin{proof} 
(1)  Clearly $NI=xT_1S_2R = xS_5R \subseteq M = xS_5R+R$.  
  By \cite[Example~11.3]{Rog09}, $R_1=xS_2+\kk g$
and so $R_1x\subseteq xT$. Equivalently,  $R_1\subseteq xT_1x^{-1}\subseteq N$.  
As  $R = T(\divc)$ is generated in degree one by Proposition~\ref{8-special22}(2),  
$R\subseteq xTx^{-1}$.  
In particular,  $M_{\geq 1}  =xS_5R+R_{\geq 1} \subseteq N$.  As $I$ is a \cog, it follows from
 Proposition~\ref{8-special22} and   
Lemma~\ref{thou11}(1) that $N^{**}=(NI)^{**} $ and hence that $M^{**}=N^{**}$.

Now consider $\widehat{M}$.  Since $1\in M$, certainly $MT=T$ and so    $M^{**}=(\wh{M}\,)^{**}= \wh{M^{**}}$
by  Lemma~\ref{lem:finehat}(3).

(2) Since  $MT=T$  we have   $\wh{M}T=T$, from which  the result follows.

(3)   We will first prove that  $\End_R(N) = x T(\divd^{\tau})x^{-1}.$  As in (1),  $R_1=xS_2+\kk g$.
Equivalently, $(x^{-1}Rx)_1=S_2x+\kk g$ is a 7-dimensional subspace of $T_1$ that vanishes at the points 
$\sigma^{-2} (p),$ $  \sigma^{-2} (q)$ and $\sigma^{-2} (r)$.  Now, 
$T(\divd)_1$ is also 7-dimensional by \cite[Theorem~1.1(1)]{Rog09}.
Consequently,   $(x^{-1}Rx)_1=T(\divd)_1$ and so  $x^{-1}Rx=T(\divd)$, 
since both algebras are generated in degree 1 by Proposition~\ref{8-special22}(2).  Therefore, 
\[
x^{-1}Nx = T_1(x^{-1}Rx) = T_1 T(\divd) = T(\divd^{\tau})T_1,
\]
where the final equality follows from \cite[Corollary~4.14]{RSS2}.
Thus 
$ xT(\divd^{\tau})T_1 x^{-1}  = N $ and 
so $\End_R(N) \supseteq G = x T(\divd^{\tau})x^{-1}.$  Since $N=GxT_1x^{-1}$, Lemma~\ref{endo}
implies that $\End_R(N)$ is a finitely generated left $G$-module. But $G$ is a maximal order by 
 \cite[Theorem~1.1(2)]{Rog09},  and so  $\End_R(N) =  G$. Thus,  by Part~(1) and 
Lemma~\ref{cozzens}, $\End_R(N) =\End_R(N^{**})=\End(M^{**}).$
Moreover, $\End_R(N)\subseteq \End_R(NI)$  and we again   have equality by Lemma~\ref{endo}.
 
It remains to  prove that $xT(\divd^{\tau})x^{-1}\not\subseteq T $. This  will follow if we show that 
$\bar{x}\overline{X}\bar{x}^{-1} \not\subseteq \overline{T}$, where  
$X=T(\divd^{\tau})$ and $\overline{X}=\left(X+gT_{(g)}\right)/T_{(g)}.  $
So, assume that $\bar{x}\overline{X}\bar{x}^{-1} \subseteq \overline{T}.$ Then 
$\overline{x}\overline{X}_1\subseteq \overline{T}_1\overline{x}
= \overline{S}_3\overline{x}.$
However, inside $\overline{S}_4$, 
$$ \overline{x}\overline{X}_1 
\ \subseteq\   H^0\left(E, \sL_4(-p-q-r-\sigma^{-6}(p)-\sigma^{-6}(q)-\sigma^{-6}(r))\right)$$
and, since both are 6-dimensional, they are equal. On the other hand, 
 $$\overline{S}_3\overline{x}= H^0\left(E,\sL_4(-\sigma^{-3}(p)-\sigma^{-3}(q)-\sigma^{-3}(r))\right).$$
Inside $\overline{S}_4$,  vanishing conditions at $\leq 12$ distinct points give independent conditions. 
So there exists $z$ that vanishes at the first 6 points $p,\dots, \sigma^{-6}(r)$ but not at the points
 $\sigma^{-3}(p),\sigma^{-3}(q),\sigma^{-3}(r)$. 
This implies that $\overline{x}\overline{X}_1\not\subseteq \overline{T}_1\overline{x},$ and completes the proof of the 
lemma.
\end{proof}

We are now able to give the desired example.
  
  \begin{proposition}\label{Tmax-eg}
  There exists a maximal order pair $(V,F)$ with $V\not=F$. In particular, $V$ is a maximal $T$-order 
  that is not a maximal order.
    
In more detail,  and using  the data from Notation~\ref{Tmax-notation},
$F=\End_R((\wh{M})^{**})  =xT(\divd^{\tau})x^{-1}$  is a blowup of $T$ at   
$\divx=\divc - \tau^{-1}(\divc)+\tau^{-2}(\divc).$  The algebra $F$ is also Auslander-Gorenstein and CM.
\end{proposition}

\begin{proof} As $1\in M$,  Theorem~\ref{thought9} and Lemma~\ref{Tmax-eg1} imply that
$F= \End_R((\wh{M})^{**})=xT(\divd^{\tau})x^{-1}$ is a maximal order with $F\not\subseteq T$.
By Theorem~\ref{thought9}, again,  $V=T\cap F$
is a $g$-divisible  maximal $T$-order, but $V$ is not a maximal order as $V\not=F$.
   That $F$ is Auslander-Gorenstein and CM follows from Proposition~\ref{8-special22}.
 
 Theorem~\ref{thought9} also implies that $F$ is a blowup of $T$ at some virtual divisor $\divy$, so it remains to check that $\divy=\divx$.  By Lemma~\ref{Tmax-eg1},
$F=\End_R(NI)=\End_R(xS_5R)$ and hence $\overline{F}\subseteq \End_{\overline{R}}(\bbar{xS_5}\bbar{ R})$.
Now, for any $n\geq 2,$ one has 
$\bbar{R}_{n-2} = H^0(E, \mathcal{M}(-\divc-\divc^{\tau}-\cdots - \divc^{\tau^{n-3}}) $ and so 
$$(\bbar{xS_5}\bbar{R})_n = H^0(E, \mathcal{M}(-\divc-\divc^{\tau^{2}}-\divc^{\tau^{3}} -\cdots - \divc^{\tau^{n-1}}) 
=H^0\left(E, \mathcal{O}(\divc^{\tau})\mathcal{M}(-\divc)_n\right).$$
Hence 
\[\bbar{F} \subseteq \End_{\bbar R}(\bbar{xS_5 R}) 
= \End_{\bbar{R}}\bigl( \bigoplus_{n \geq 2} H^0(E, \mathcal{O}(\divc^{\tau}) \sM(-\divc)_n)) \bigr).\]

Therefore, by Lemma~\ref{thought8}(1), $\bbar{F} \ppe B(E, \sM(-\divx), \tau)$. 
By \cite[Theorem~1.1(2)]{Rog09} and Riemann-Roch, $\dim \bbar{F}_n=6n = \dim B(E, \sM(-\divx), \tau)$, for $n\geq 1$, 
and hence $\bbar{F}=B(E, \sM(-\divx), \tau)$, as required.
\end{proof}

When $\divy$ is effective, $T(\divy)$ is both  Auslander-Gorenstein and  CM (see Proposition~\ref{8-special22}),
as is the blowup of $T$ at $\divx$ from Proposition~\ref{Tmax-eg}. 
Despite this example, neither the Auslander-Gorenstein nor the   CM condition  is automatic for a blowup of $T$ at  
 virtually effective divisors.

\begin{example}\label{non-AG}
Let $\divx=p-\tau(p)+\tau^2(p)$ for a closed point $p\in E$ and let $U$ be a blowup of $T$ at $\divx$.  
Then $U$ is a maximal order contained in $T$ that is 
neither Auslander-Gorenstein  nor AS Gorenstein   nor CM.
\end{example}

\begin{proof}  By Definition~\ref{def:blowup}  and Corollary~\ref{correct3}(2),  $U=\End_{T(q)}(M)$, where $M=M^{**}$ satisfies $MT=T$ and $q$ is a closed point that is 
$\tau$-equivalent to $\divx$ and hence to $p$.   By \cite[Example~9.5]{RSS2} $T(q)$ has no \cogs\ and so,
   by  Corollary~\ref{correct3}(3), $U $   is a $g$-divisible maximal order contained in $T$.

 Now consider $\overline{U}=U/gU.$ By Theorem~\ref{thought9}, $\overline{U}\ppe B=B(E,\mathcal{M}(-\divx),\tau).$ 
We emphasise that we always  identify $\mathcal{M}(-\divx)$ and $\mathcal{M}$ with the appropriate  subsheaves of the field $\kk(E)$ and $B$ with the corresponding subring of the Ore extension $T_{(g)}/gT_{(g)} \cong \kk(E)[z, z^{-1}; \tau]$. 
 We first want to show that $\overline{U}\not=B$. Since $\deg(\mathcal{M}(-\divx)) =\deg \mathcal{M}-\deg \divx =8$, 
  \cite[Corollary~IV.3.2]{Ha} implies that $\mathcal{M}(-\divx)$ is very ample 
and generated by its sections $B_1=H^0(E, \mathcal{M}(-\divx))$.  On the other hand, the inclusion $U\subseteq T$ forces 
 $\overline{U}\subseteq \overline{T}=B(E,\mathcal{M},\tau)$ and again $\overline{T}_1$ generates $\mathcal M$.
  Therefore, if $\overline{U}=B$ or even if  $\overline{U}_1=B_1$ 
 then   $\mathcal{M}(-\divx)\subseteq \mathcal M$. Since $\divx$ is not effective, this is impossible and so $\overline{U}\not=B$, as claimed.
 
 We now turn to the  homological questions.   By
 \cite[Theorem~5.10]{Le}, $U$ is Auslander-Gorenstein, AS Gorenstein or
   CM, if and only if the same holds for $\overline{U}$. Thus we can concentrate on $\overline{U}$.
 Since $B/\overline{U}$ is a non-zero, finite dimensional vector space, and $B$ is a domain, certainly 
 $\Ext^1_{\overline{U}}(\kk,\overline{U})\not=0$ (on either side). 
 %Since $\GKdim \overline{U}=\GKdim B=2$ this  
% implies that $\overline{U}$ is neither AS Gorenstein nor  CM.   
% Then by \cite[Theorem 6.3]{Le}, $\overline{U}$ cannot be Auslander Gorenstein either. 
% \end{proof}
Since $\GKdim \overline{U}=\GKdim B=2$ this  
certainly implies that $\overline{U}$ is not  CM.
Moreover if we can prove that $\Ext^2_{\overline{U}}(\kk,\overline{U})\not=0$
on either side, then $\overline{U}$ will be neither AS Gorenstein nor  
Auslander-Gorenstein.

 By  
 \cite[Proposition~6.5]{Le},  $\Ext^i_B(\kk,B)=\delta_{i,2}\kk$, 
 up to a shift in degree. 
   Therefore \cite[Corollary~10.65]{Rot}, with $A=\kk$, $B=S$ and $R=C=\overline{U}$,  gives
 \begin{equation}\label{non-AG2}
 \Ext^2_{\overline{U}}(\kk,\,\overline{U}) = \Ext^2_{\overline{U}}(B\otimes_B\kk,\, \overline{U})
  = \Ext^2_B( \kk,\, J), \quad \text{for } J=\Hom_{\overline{U}}(B,\, \overline{U}).
 \end{equation}
 Since $\overline{U}\ppe B$  clearly $L=B/J$ is also a non-zero finite dimensional $\kk$-vector space.
  We claim that the same is true of  $\Ext^2_B(\kk,J)$.
  As $\Ext^1_B(\kk,B)=0$, we have an exact sequence
 \begin{equation}\label{non-AG3}  0\too \Ext^1_B(\kk,L)\too \Ext^2_B(\kk,J)\too \Ext^2_B(\kk,B)\too \cdots.
 \end{equation}
Since  $\dim_{\kk}\Ext^1_B(\kk,L)<\infty$,  the claim will follow once we show that   
  $\Ext^1_B(\kk,L)\not=0$.  
  
   As in \cite[(7.1.2)]{AZ}, let $I(L)$ denote the largest essential extension of $L$ by locally finite dimensional modules. 
 If $soc(L)$ denotes the socle of $L$, then $L$ and $soc(L)$ have the same injective hulls and hence the same
  torsion-injective hulls $I(L)=I(soc(L))$.  By \cite[Lemma~2.2(2)]{Rog09}  $B$ satisfies $\chi$  in the sense of  
  \cite[Definition~3.2]{AZ} and so, by  \cite[Proposition~7.7]{AZ}, $I(L)$ is  a direct sum of copies of 
  shifts of the vector space dual $B^*$. Since this is strictly larger than $L$, $\Ext^1_B(\kk,L)\not=0$ and the claim follows.

  In conclusion, by \eqref{non-AG3} we know that $0<\dim_{\kk}\Ext^2_B(\kk,J)<\infty$ and hence by \eqref{non-AG2}
  it follows that (up to a shift) 
  $\kk\hookrightarrow \Ext^2_{\overline{U}}(\kk,\overline{U}) $ as left $\overline{U}$-modules. 
As noted earlier,  this shows that 
both Gorenstein conditions fail.
\end{proof}

\begin{remark}\label{non-AG4}  
(1) By expanding upon the above proof one can in fact show that 
 $U$ from Example~\ref{non-AG} will  have infinite injective dimension.

(2)  Explicit computation shows that $U$ is not uniquely determined by $\divx$ as a subalgebra of $T$, 
although the factor $\bbar{U}$ is determined  in large degree.
 We do not know whether $U$ is unique up to isomorphism. 
\end{remark}

Let $U$ be a noetherian subring of $T$ with $Q_{gr}(U)=Q_{gr}(T)$.
In Proposition~\ref{max2},  we had to assume that $U\not\subseteq \kk+gT$ in order to find a $g$-divisible, 
equivalent order and this  meant that  the same assumption was needed  for the rest of Section~\ref{MAX1}.
  In our next  example we show that the conclusions of Proposition~\ref{max2} can  fail without this assumption, as does Theorem~\ref{max32}. Thus   Proposition~\ref{AS-trick} is necessary for Section~\ref{ARBITRARY}.

In order to define the ring, pick algebra generators of $T$ in degree 1; say  
$T=\kk\langle a_1,\dots a_r\rangle$, set $ T^g=\kk\langle ga_1,\dots, ga_r\rangle $ 
and write $U=T^g\langle g\rangle\subset T$ for the subring of $T$ generated by $T^g$ and $g$. 

\begin{example}\label{max6} Keep $T^g$ and $U=T^g\langle g\rangle$ as above. Then,
\begin{enumerate}
\item  There is a semi-graded isomorphism $T^g\cong T$. Thus $U$ is   noetherian and there is a semi-graded isomorphism
 $T[x]/(x^2-g) \cong U$ mapping $x$ to $g$. Moreover, $U^{(2)} = T^g$ and so $U^\circ = (T^g)^\circ \cong T^\circ$.
\item $U\subseteq \kk+gT$ and so $\overline{U}=\kk$.
\item $gU$ is a prime ideal of $U$ such that there is a semi-graded isomorphism $U/gU\cong B=T/gT$.
\item $\wh{U}=T$ but $T$ is not  finitely generated as a right (or left) $U$-module.
\item $U$ is a maximal order with $Q_{gr}(U)=Q_{gr}(T)$.
\end{enumerate}
\end{example}

\begin{proof} (1,2)  These are routine computations. 

(3) Under the identification $U=T[x]/(x^2-g)$,  clearly $U/xU = T/gT$. 

(4)  For any $\theta\in T_n$ one has $g^n\theta\in T^g\subseteq U$ and hence $\wh{U}=\wh{T^g}=T$.
 If $T$ were finitely generated as a (right) $U$-module then the  factor $B=T/gT$ would be finitely generated as a 
 module  over the image $(U+gT)/gT = \overline{U}$ of $U$ in $B$. This contradicts (2).

 (5)  Write $U=T[x]/(x^2-g)$; thus $x\in U_1$ but the grading of $T$ is shifted.
If $U$ is not a maximal order then  there exists a cg ring $U\subsetneq V\subset Q(U)$ such that either 
$aV\subseteq U$ or $Va\subseteq U$ for some $0\not= a\in U$. By symmetry we may assume the former, in which case 
$IV=I$ for the nonzero ideal $I=UaV$ of $U$.  
Thus  $I^{(2)}V^{(2)}=I^{(2)}$, and $I^{(2)}\not=0$ since $U$ is a domain. 
 Since  $U^{(2)}= T$ is a maximal order by Proposition~\ref{8-special22}(4), 
 it follows that  $V^{(2)}=U^{(2)}=T$.
Let $f\in V\smallsetminus U$ be homogeneous. Then $f$ appears in odd degree and so 
$fx\in V^{(2)}=U^{(2)}=T$ and $f=tx^{-1}$ for some $t\in T$. However, $T=V^{(2)}\ni f^2=(tx^{-1})^2=t^2g^{-1}$.
Hence $t^2\in gT$ which, since $T/gT$ is a domain, forces $t=gt_1\in gT$. But this implies that 
$f=tx^{-1}=xt_1\in U$, a contradiction. Thus $U$ is indeed a maximal order. Moreover, as $g\in U$, clearly 
each $a_i\in Q_{gr}(U)$ and hence $Q_{gr}(T) = Q_{gr}(U).$
\end{proof}

 In this paper we have only been concerned with two-sided noetherian rings, since we believe that this is the appropriate
  context for noncommutative geometry.   For one-sided noetherian rings there are further 
 examples that can appear as is illustrated by the following example.   
 
 \begin{example}\label{not right noetherian}
 Let $J$ be a right ideal of $T$ such that $g\in J$ and $\GKdim(T/J)=1$. Then the idealiser $A=\idealiser (J)$ is right but not left noetherian.  
 \end{example}
 
 \begin{proof}  Let $\overline{J}=J/gJ$. Since $B=T/gT$ is just infinite
 \cite[Lemma~3.2]{Rog09},  
 $\dim_{\kk}T/TJ<\infty$. Since $TJ=\sum_{i=1}^m t_iJ$ for some $t_j$ it follows that $TJ$ and hence 
 $T$ are finitely generated right $A$-modules. Thus, by  the proof of \cite[Theorem~3.2]{SZ}   $A$ is right
 noetherian. On the other hand,  $B$ is not a 
  finitely generated left $A/gT$-module, and so $gT$ is an ideal of $A$ that cannot be finitely generated 
  as a left $A$-module.
 \end{proof}

%%%%%%%%%%%%%%%%%%%%%%%%%%%%%%%%%%%%%
%%%%%%%%%%%%%%%%%%%%%%%%%%%%%%%% 

 \section*{Index of Notation}\label{index}
\begin{multicols}{2}
{\small  \baselineskip 14pt

$\ppe $, equal in high degree \hfill\pageref{ppe-defn}

$\alpha$-pure  \hfill\pageref{pure}

Allowable divisor layering  $\divd^{\bullet}$   \hfill\pageref{layering-defn}

% Anti-automorphism $\phi$  \hfill\pageref{ass:T}

%AS Gorenstein,    AS regular  \hfill\pageref{AS-defn}

% Auslander-Gorenstein  \hfill\pageref{AG-defn}

Blowup at an arbitrary divisor   \hfill\pageref{def:blowup}

%cg, connected graded  \hfill\pageref{cg-defn}

CM and Gorenstein conditions    \hfill\pageref{AG-defn}

$\divd^\tau=\tau^{-1}(\divd)$    for a divisor $\divd$  \hfill\pageref{tau-notation}

  $\divd_n=\divd+\divd^\tau+\cdots+\divd^{\tau^{n-1}}$
  for a divisor $\divd$  \hfill\pageref{tau-notation}
  
 $D_{gr}(A)$, function skewfield    \hfill\pageref{quot-not}
  
%Equivalent orders \hfill\pageref{equiv-defn1}

$\mathcal{F}_n=\mathcal{F}\otimes \mathcal{F}^\tau\otimes\cdots\otimes
\mathcal{F}^{\tau^{n-1}}$ for a sheaf $\mathcal{F}$   \hfill\pageref{TCR-defn}

$g$-divisible \hfill\pageref{g-div}

% $g$-torsion, $g$-torsionfree  \hfill\pageref{g-tors}

  %$\GKdim$ \hfill\pageref{quot-not}

% $\rgr A$,    graded noetherian right $A$-modules \hfill\pageref{rgr-defn}

Geometric data $(\divy,\divx,k)$   for $A$  \hfill\pageref{normalised-notation}

$\Hom(I,J)=\Hom_{\text{Mod-}A}(I,J)$, $ \Hom_{\rGr A}(I,J) $ \hfill\pageref{hom-defn}

Idealiser $\idealiser(J)$ \hfill\pageref{idealiser-defn}

% $j(M)$, homological grade  \hfill\pageref{AG-defn}

Just infinite   \hfill\pageref{ji-defn}

Left allowable divisor layering  $\divd^{\bullet}$   \hfill\pageref{left-layering-defn}

$\mu=\deg\mathcal{M}$    \hfill\pageref{ass:T}

$M(k,\divd)$ \hfill\pageref{def:M}

Maximal order pair $(V,F)$    \hfill\pageref{max-pair-defn}

Maximal $T$-order    \hfill\pageref{maxorder-defn}

Minimal \cog\ \hfill\pageref{min-cog}

Normalised orbit representative, divisor    \hfill\pageref{normalised-notation}

%$\Omega(A)$, globalisation of $A\subseteq T^\circ$  \hfill\pageref{Omega-defn}

$p_i=\tau^{-i}(p)$ for a point $p$ \hfill\pageref{4.4}, \pageref{tau-notation}

Point modules \hfill\pageref{pt-defn}

  $P(p), P'(p)$    \hfill\pageref{pt-defn2}, \pageref{pt-defn3}

$\rqgr A$,  quotient category of $\rgr A$  \hfill\pageref{rqgr-defn}

 $Q_{gr}(A)$,  graded quotient ring    \hfill\pageref{quot-not}

$Q(i,d,r,p)$ \hfill\pageref{def:Q}

$R^\circ$, localisation of $R$  \hfill\pageref{circ-defn}

%$\sigma$, the automorphism defining $S$ \hfill\pageref{sigma-defn}

Saturation $I^{sat}$, saturated right ideal \hfill\pageref{saturated-defn}

Semi-graded  morphism   \hfill\pageref{semi-graded defn}

%Shift $M[n]$ \hfill\pageref{shift-defn}

% Sklyanin algebras, $S$ \hfill\pageref{Skl-ex}
 
Sporadic ideal \hfill\pageref{cog-defn}, \pageref{cog-defn2}

%Strongly noetherian \hfill\pageref{strong-defn}

 $\tau$, automorphism defining $T$ \hfill\pageref{ass:T}
 
% $\tau$-ample \hfill\pageref{tau-ample}
 
 $\tau$-equivalent divisors and invertible sheaves  \hfill\pageref{tau-equiv-defn}
 
% $T=S^{(d)}$ \hfill\pageref{ass:T}

 $T(\divd)$, effective blowup  \hfill\pageref{def:star}

$T_{(g)}$, graded localisation \hfill\pageref{startup}
 
$T_{\leq \ell}\ast T(\divd)$   \hfill\pageref{def:star}
 
TCR,   twisted coordinate ring $B(X,\mathcal{L},\theta)$  \hfill\pageref{TCR-defn}

 Unregraded ring \hfill\pageref{unregraded-defn}
 
%  Veronese ring $A^{(n)}$  \hfill\pageref{unregraded-defn}

 Virtual blowup    \hfill\pageref{virtual-defn}
 
 Virtually effective divisor  $\divx = \divu-\divv+\tau^{-1}(\divv)$    
 \hfill\pageref{virtual-defn}

$\wh{X},\, \overline{X}$ \hfill\pageref{g-div}

} \end{multicols}

%%%%%%%%%%%%%%%%%%%%%%%%%%%%%%%%%%%%%
  
  %%%%%%%%%%%%%%%%%%%%%%%

\bibliographystyle{amsalpha}

%\bibliography{../../../Dropbox/biblio}

\providecommand{\bysame}{\leavevmode\hbox to3em{\hrulefill}\thinspace}

\providecommand{\MR}{\relax\ifhmode\unskip\space\fi MR }


\providecommand{\MRhref}[2]{%

  \href{http://www.ams.org/mathscinet-getitem?mr=#1}{#2}

}

\providecommand{\href}[2]{#2}

\end{document}